\documentclass[a4paper, british]{amsart}
\usepackage{tikz}
\usepackage{tikz-cd}
\usepackage{comment}
\usepackage{amssymb}
\usepackage{babel}
\usepackage{enumitem}
\usepackage{hyperref}
\usepackage[utf8]{inputenc}
\usepackage{newunicodechar}
\usepackage{mathtools}
\usepackage{varioref}
\usepackage[arrow,curve,matrix]{xy}

\usepackage{colortbl}
\usepackage{graphicx}
\usepackage{tikz}

\usepackage{mathrsfs}

\definecolor{linkred}{rgb}{0.7,0.2,0.2}
\definecolor{linkblue}{rgb}{0,0.2,0.6}

\numberwithin{figure}{section}

\usepackage[hyperpageref]{backref}

\setdescription{labelindent=\parindent, leftmargin=2\parindent}
\setitemize[1]{labelindent=\parindent, leftmargin=2\parindent}
\setenumerate[1]{labelindent=0cm, leftmargin=*, widest=iiii}

\newunicodechar{α}{\ensuremath{\alpha}}
\newunicodechar{β}{\ensuremath{\beta}}
\newunicodechar{χ}{\ensuremath{\chi}}
\newunicodechar{δ}{\ensuremath{\delta}}
\newunicodechar{ε}{\ensuremath{\varepsilon}}
\newunicodechar{Δ}{\ensuremath{\Delta}}
\newunicodechar{η}{\ensuremath{\eta}}
\newunicodechar{γ}{\ensuremath{\gamma}}
\newunicodechar{Γ}{\ensuremath{\Gamma}}
\newunicodechar{ι}{\ensuremath{\iota}}
\newunicodechar{κ}{\ensuremath{\kappa}}
\newunicodechar{λ}{\ensuremath{\lambda}}
\newunicodechar{Λ}{\ensuremath{\Lambda}}
\newunicodechar{ν}{\ensuremath{\nu}}
\newunicodechar{μ}{\ensuremath{\mu}}
\newunicodechar{ω}{\ensuremath{\omega}}
\newunicodechar{Ω}{\ensuremath{\Omega}}
\newunicodechar{π}{\ensuremath{\pi}}
\newunicodechar{Π}{\ensuremath{\Pi}}
\newunicodechar{φ}{\ensuremath{\phi}}
\newunicodechar{Φ}{\ensuremath{\Phi}}
\newunicodechar{ψ}{\ensuremath{\psi}}
\newunicodechar{Ψ}{\ensuremath{\Psi}}
\newunicodechar{ρ}{\ensuremath{\rho}}
\newunicodechar{σ}{\ensuremath{\sigma}}
\newunicodechar{Σ}{\ensuremath{\Sigma}}
\newunicodechar{τ}{\ensuremath{\tau}}
\newunicodechar{θ}{\ensuremath{\theta}}
\newunicodechar{Θ}{\ensuremath{\Theta}}
\newunicodechar{ξ}{\ensuremath{\xi}}
\newunicodechar{Ξ}{\ensuremath{\Xi}}
\newunicodechar{ζ}{\ensuremath{\zeta}}

\newunicodechar{ℓ}{\ensuremath{\ell}}

\newunicodechar{𝔹}{\ensuremath{\bB}}
\newunicodechar{ℂ}{\ensuremath{\bC}}
\newunicodechar{𝔻}{\ensuremath{\bD}}
\newunicodechar{𝔼}{\ensuremath{\bE}}
\newunicodechar{𝔽}{\ensuremath{\bF}}
\newunicodechar{ℕ}{\ensuremath{\bN}}
\newunicodechar{ℙ}{\ensuremath{\bP}}
\newunicodechar{ℚ}{\ensuremath{\bQ}}
\newunicodechar{ℝ}{\ensuremath{\bR}}
\newunicodechar{𝕏}{\ensuremath{\bX}}
\newunicodechar{ℤ}{\ensuremath{\bZ}}
\newunicodechar{𝒜}{\ensuremath{\sA}}
\newunicodechar{ℬ}{\ensuremath{\sB}}
\newunicodechar{𝒞}{\ensuremath{\sC}}
\newunicodechar{𝒟}{\ensuremath{\sD}}
\newunicodechar{ℰ}{\ensuremath{\sE}}
\newunicodechar{ℱ}{\ensuremath{\sF}}
\newunicodechar{𝒢}{\ensuremath{\sG}}
\newunicodechar{ℋ}{\ensuremath{\sH}}
\newunicodechar{𝒥}{\ensuremath{\sJ}}
\newunicodechar{ℒ}{\ensuremath{\sL}}
\newunicodechar{𝒪}{\ensuremath{\sO}}
\newunicodechar{𝒬}{\ensuremath{\sQ}}
\newunicodechar{𝒯}{\ensuremath{\sT}}
\newunicodechar{𝒲}{\ensuremath{\sW}}

\newunicodechar{∂}{\ensuremath{\partial}}
\newunicodechar{∇}{\ensuremath{\nabla}}

\newunicodechar{↺}{\ensuremath{\circlearrowleft}}
\newunicodechar{∞}{\ensuremath{\infty}}
\newunicodechar{⊕}{\ensuremath{\oplus}}
\newunicodechar{⊗}{\ensuremath{\otimes}}
\newunicodechar{•}{\ensuremath{\bullet}}
\newunicodechar{Λ}{\ensuremath{\wedge}}
\newunicodechar{↪}{\ensuremath{\into}}
\newunicodechar{→}{\ensuremath{\to}}
\newunicodechar{↦}{\ensuremath{\mapsto}}
\newunicodechar{⨯}{\ensuremath{\times}}
\newunicodechar{∪}{\ensuremath{\cup}}
\newunicodechar{∩}{\ensuremath{\cap}}
\newunicodechar{⊋}{\ensuremath{\supsetneq}}
\newunicodechar{⊇}{\ensuremath{\supseteq}}
\newunicodechar{⊃}{\ensuremath{\supset}}
\newunicodechar{⊊}{\ensuremath{\subsetneq}}
\newunicodechar{⊆}{\ensuremath{\subseteq}}
\newunicodechar{⊂}{\ensuremath{\subset}}
\newunicodechar{≥}{\ensuremath{\geq}}
\newunicodechar{≠}{\ensuremath{\neq}}
\newunicodechar{≫}{\ensuremath{\gg}}
\newunicodechar{≪}{\ensuremath{\ll}}

\newunicodechar{≤}{\ensuremath{\leq}}
\newunicodechar{∈}{\ensuremath{\in}}
\newunicodechar{∉}{\ensuremath{\not \in}}
\newunicodechar{∖}{\ensuremath{\setminus}}
\newunicodechar{◦}{\ensuremath{\circ}}
\newunicodechar{°}{\ensuremath{^\circ}}
\newunicodechar{…}{\ifmmode\mathellipsis\else\textellipsis\fi}
\newunicodechar{·}{\ensuremath{\cdot}}
\newunicodechar{⋯}{\ensuremath{\cdots}}
\newunicodechar{∅}{\ensuremath{\emptyset}}
\newunicodechar{⇒}{\ensuremath{\Rightarrow}}

\newunicodechar{⁰}{\ensuremath{^0}}
\newunicodechar{¹}{\ensuremath{^1}}
\newunicodechar{²}{\ensuremath{^2}}
\newunicodechar{³}{\ensuremath{^3}}
\newunicodechar{⁴}{\ensuremath{^4}}
\newunicodechar{⁵}{\ensuremath{^5}}
\newunicodechar{⁶}{\ensuremath{^6}}
\newunicodechar{⁷}{\ensuremath{^7}}
\newunicodechar{⁸}{\ensuremath{^8}}
\newunicodechar{⁹}{\ensuremath{^9}}
\newunicodechar{ⁱ}{\ensuremath{^i}}

\newunicodechar{⌈}{\ensuremath{\lceil}}
\newunicodechar{⌉}{\ensuremath{\rceil}}
\newunicodechar{⌊}{\ensuremath{\lfloor}}
\newunicodechar{⌋}{\ensuremath{\rfloor}}

\newunicodechar{≅}{\ensuremath{\cong}}
\newunicodechar{⇔}{\ensuremath{\Leftrightarrow}}

\DeclareFontFamily{OMS}{rsfs}{\skewchar\font'60}
\DeclareFontShape{OMS}{rsfs}{m}{n}{<-5>rsfs5 <5-7>rsfs7 <7->rsfs10 }{}
\DeclareSymbolFont{rsfs}{OMS}{rsfs}{m}{n}
\DeclareSymbolFontAlphabet{\scr}{rsfs}
\DeclareSymbolFontAlphabet{\scr}{rsfs}

\DeclareFontFamily{U}{mathx}{\hyphenchar\font45}
\DeclareFontShape{U}{mathx}{m}{n}{
      <5> <6> <7> <8> <9> <10>
      <10.95> <12> <14.4> <17.28> <20.74> <24.88>
      mathx10
      }{}
\DeclareSymbolFont{mathx}{U}{mathx}{m}{n}
\DeclareFontSubstitution{U}{mathx}{m}{n}
\DeclareMathAccent{\wcheck}{0}{mathx}{"71}

\DeclareMathOperator{\Aut}{Aut}
\DeclareMathOperator{\codim}{codim}

\DeclareMathOperator{\Hom}{Hom}
\DeclareMathOperator{\Id}{Id}

\DeclareMathOperator{\img}{img}

\DeclareMathOperator{\rank}{rank}

\DeclareMathOperator{\reg}{reg}

\DeclareMathOperator{\Sym}{Sym}
\DeclareMathOperator{\supp}{supp}

\newcommand{\sA}{\scr{A}}
\newcommand{\sB}{\scr{B}}
\newcommand{\sC}{\scr{C}}
\newcommand{\sD}{\scr{D}}
\newcommand{\sE}{\scr{E}}
\newcommand{\sF}{\scr{F}}
\newcommand{\sG}{\scr{G}}
\newcommand{\sH}{\scr{H}}

\newcommand{\sJ}{\scr{J}}

\newcommand{\sL}{\scr{L}}

\newcommand{\sN}{\scr{N}}
\newcommand{\sO}{\scr{O}}

\newcommand{\sQ}{\scr{Q}}

\newcommand{\sT}{\scr{T}}

\newcommand{\sV}{\scr{V}}
\newcommand{\sW}{\scr{W}}

\newcommand{\cA}{\mathcal A}
\newcommand{\cC}{\mathcal C}

\newcommand{\cN}{\mathcal N}

\newcommand{\bB}{\mathbb{B}}
\newcommand{\bC}{\mathbb{C}}
\newcommand{\bD}{\mathbb{D}}
\newcommand{\bE}{\mathbb{E}}
\newcommand{\bF}{\mathbb{F}}

\newcommand{\bK}{\mathbb{K}}

\newcommand{\bN}{\mathbb{N}}

\newcommand{\bP}{\mathbb{P}}
\newcommand{\bQ}{\mathbb{Q}}
\newcommand{\bR}{\mathbb{R}}

\newcommand{\bX}{\mathbb{X}}

\newcommand{\bZ}{\mathbb{Z}}

\theoremstyle{plain}
\newtheorem{thm}{Theorem}[section]

\newtheorem{cor}[thm]{Corollary}
\newtheorem{defn}[thm]{Definition}
\newtheorem{fact}[thm]{Fact}
\newtheorem{lem}[thm]{Lemma}

\newtheorem{prop}[thm]{Proposition}

\theoremstyle{remark}

\newtheorem{claim}[thm]{Claim}
\newtheorem{c-n-d}[thm]{Claim and Definition}

\newtheorem{construction}[thm]{Construction}

\newtheorem{example}[thm]{Example}
\newtheorem{explanation}[thm]{Explanation}
\newtheorem{notation}[thm]{Notation}
\newtheorem{obs}[thm]{Observation}
\newtheorem{rem}[thm]{Remark}

\newtheorem*{rem-nonumber}{Remark}

\numberwithin{equation}{thm}

\setlist[enumerate]{label=(\thethm.\arabic*), before={\setcounter{enumi}{\value{equation}}}, after={\setcounter{equation}{\value{enumi}}}}

\newcommand{\into}{\hookrightarrow}

\newcommand{\wtilde}{\widetilde}
\newcommand{\what}{\widehat}

\newcommand\CounterStep{\addtocounter{thm}{1}\setcounter{equation}{0}}

\newcommand{\factor}[2]{\left. \raise 2pt\hbox{$#1$} \right/\hskip -2pt\raise -2pt\hbox{$#2$}}
\newcommand{\Preprint}[1]{#1}
\newcommand{\Publication}[1]{}

\newcommand{\subversionInfo}{}
\newcommand{\svnid}[1]{}
\newcommand{\approvals}[2][Approval]{}

\usepackage{mathpazo}

\renewcommand{\phi}{\varphi}

\DeclareMathOperator{\ch}{ch}
\DeclareMathOperator{\diag}{diag}
\DeclareMathOperator{\diam}{diam}
\DeclareMathOperator{\Div}{Div}
\DeclareMathOperator{\End}{End}

\DeclareMathOperator{\Fix}{Fix}
\DeclareMathOperator{\GL}{GL}
\DeclareMathOperator{\GSp}{GSp}
\DeclareMathOperator{\Hol}{Hol}
\DeclareMathOperator{\loc}{loc}
\DeclareMathOperator{\NS}{NS}
\DeclareMathOperator{\pr}{pr}

\DeclareMathOperator{\Ric}{Ric}
\DeclareMathOperator{\Sp}{Sp}

\DeclareMathOperator{\SU}{SU}
\DeclareMathOperator{\SL}{SL}
\DeclareMathOperator{\td}{td}
\DeclareMathOperator{\tr}{tr}
\DeclareMathOperator{\U}{U}
\DeclareMathOperator{\Vol}{Vol}

\newcommand{\cT}{\mathcal T}
\newcommand{\cW}{\mathcal W}

\theoremstyle{theorem}
\newtheorem{bigthm}{Theorem}
\newtheorem{bigprop}[bigthm]{Proposition}
\newtheorem{propnot}[thm]{Proposition and Notation}
\newtheorem{thmdef}[thm]{Theorem and Definition}
\newtheorem{thmnot}[thm]{Theorem and Notation}
\theoremstyle{remark}
\newtheorem*{explanation*}{Explanation}

\newtheorem{consnot}[thm]{Construction and Notation}
\newtheorem{obsnot}[thm]{Observation and Notation}
\newtheorem*{rem*}{Remark}
\newtheorem{reminder}[thm]{Reminder}

\newtheorem{setupnot}[thm]{Setup and Notation}

\author{Daniel Greb}
\address{Daniel Greb, Essener Seminar für Algebraische Geometrie und Arithmetik, Fakultät für Mathematik, Universität Duisburg--Essen, 45117 Essen, Germany}
\email{\href{mailto:daniel.greb@uni-due.de}{daniel.greb@uni-due.de}}
\urladdr{\href{http://www.esaga.uni-due.de/daniel.greb}{http://www.esaga.uni-due.de/daniel.greb}}

\author{Henri Guenancia}
\address{Henri Guenancia, Department of Mathematics, Stony Brook University, Stony Brook, NY 11794-3651, U.S.A.}
\email{\href{mailto:henri.guenancia@stonybrook.edu}{henri.guenancia@stonybrook.edu}}
\urladdr{\href{http://www.math.stonybrook.edu/~guenancia}{http://www.math.stonybrook.edu/$\sim$guenancia}}

\author{Stefan Kebekus}%
\address{Stefan Kebekus, Mathematisches Institut, Albert-Ludwigs-Universität
  Freiburg, Ernst-Zermelo-Straße 1, 79104 Freiburg im Breisgau, Germany \&
  Freiburg Institute for Advanced Studies (FRIAS), Freiburg im Breisgau, Germany
  \& University of Strasbourg Institute for Advanced Study (USIAS), Strasbourg,
  France}%
\email{\href{mailto:stefan.kebekus@math.uni-freiburg.de}{stefan.kebekus@math.uni-freiburg.de}}
\urladdr{\href{https://cplx.vm.uni-freiburg.de}{https://cplx.vm.uni-freiburg.de}}

\thanks{Daniel Greb is partially supported by the DFG-Collaborative Research
  Center SFB/TRR 45.  Henri Guenancia is partially supported by
  the NSF Grant DMS-1510214.  Stefan Kebekus gratefully acknowledges partial
  support through a joint fellowship of the Freiburg Institute of Advanced
  Studies (FRIAS) and the University of Strasbourg Institute for Advanced Study
  (USIAS)}
  
\keywords{varieties with trivial canonical divisor, klt
  singularities, Kähler-Einstein metrics, stability, holonomy groups, Bochner
  principle, irreducible holomorphic symplectic varieties, Calabi-Yau varieties,
  differential forms, fundamental groups, decomposition theorem}

\subjclass[2010]{14J32, 14E30, 32J27}

\title[Klt varieties with trivial canonical class]{Klt varieties with trivial canonical class.\\ Holonomy, differential forms, and fundamental groups}
\date{\today}

\makeatletter
\hypersetup{
  pdfauthor={\authors},
  pdftitle={\@title},
  pdfsubject={\@subjclass},
  pdfkeywords={\@keywords},
  pdfstartview={Fit},
  pdfpagelayout={TwoColumnRight},
  pdfpagemode={UseOutlines},
  bookmarks,
  colorlinks,
  linkcolor=linkblue,
  citecolor=linkred,
  urlcolor=linkred}
\makeatother

\begin{document}
\begin{abstract}
%
%
\svnid{$Id: abstract.tex 745 2017-11-02 10:38:51Z kebekus $}

We investigate the holonomy group of singular Kähler-Einstein metrics on klt
varieties with numerically trivial canonical divisor.  Finiteness of the number
of connected components, a Bochner principle for holomorphic tensors, and a
connection between irreducibility of holonomy representations and stability of
the tangent sheaf are established.  As a consequence, known decompositions for
tangent sheaves of varieties with trivial canonical divisor are refined.  In
particular, we show that up to finite quasi-étale covers, varieties with
strongly stable tangent sheaf are either Calabi-Yau or irreducible holomorphic
symplectic.  These results form one building block for Höring-Peternell's recent
proof of a singular version of the Beauville-Bogomolov Decomposition Theorem.

\end{abstract}
\maketitle
\enlargethispage{0.1cm}
\approvals[Abstract]{Daniel & yes \\ Henri & yes \\ Stefan & yes}

\tableofcontents

%
%
\svnid{$Id: 01-intro.tex 763 2018-10-27 13:39:14Z kebekus $}

\section{Introduction}
\subversionInfo
\approvals{Daniel & yes \\ Henri & yes \\ Stefan & yes}

The structure of compact Kähler manifolds with vanishing first Chern class is
encapsulated in the following fundamental result.

\begin{thm}[Decomposition Theorem, \cite{Bea83} and references therein]\label{thm:bbd}
  Let $X$ be a compact Kähler manifold with
  $c_1(X) = 0 ∈ H²\bigl(X, \, ℝ \bigr)$.  Then, there exists a finite étale
  cover $\wtilde{X}→ X$ such that $\wtilde{X}$ decomposes as a Kähler manifold
  as follows,
  $$
  \wtilde{X}=T ⨯ \prod\nolimits_i Y_i ⨯ \prod\nolimits_j Z_j,
  $$
  where $T$ is a complex torus, and where the $Y_i$ (resp.\ $Z_j$) are
  irreducible and simply connected Calabi-Yau manifolds (resp.\ holomorphic
  symplectic manifolds).  \qed
\end{thm}

In view of a desired birational classification of varieties with Kodaira
dimension zero and in view of the recent progress in the Minimal Model Program,
it is important to extend the Decomposition Theorem, \textit{mutatis mutandis},
to the setting of varieties with mild singularities.  This turns out to be a
very difficult challenge.  Indeed, the strategy of the proof of
Theorem~\ref{thm:bbd} consists in first using Yau's solution to the Calabi
conjecture in order to equip $X$ with a Ricci-flat Kähler metric, and then
applying the deep theorems of De Rham and Cheeger-Gromoll to split a finite
étale cover of $X$ according to its holonomy decomposition.  The identification
of the factors then follows from the Berger-Simons classification of holonomy
groups combined with the Bochner principle, which states that holomorphic
tensors are parallel.

Now, if $X$ is a singular projective variety with klt singularities and
numerically trivial canonical divisor, one can still achieve the first step.
More precisely, Eyssidieux, Guedj, and Zeriahi constructed ``natural''
Ricci-flat Kähler metrics $ω$ on the regular locus $X_{\reg}$ of $X$,
\cite{MR2505296}; see Section~\ref{Section:KE} for further references.  However,
as we will see in Proposition~\ref{prop:incomplete}, the Kähler manifold
$(X_{\reg}, ω)$ is not geodesically complete unless $X$ is smooth.  The
incompleteness of $ω$ is a major obstacle to using the splitting theorems
mentioned above or the Bochner principle directly in our setup.  Consequently,
it is highly challenging to analyse the geometry of $(X_{\reg}, ω)$ using
differential-geometric techniques alone.

In this paper, we use recent advances in higher-dimensional algebraic geometry
to study the geometry of the Kähler manifold $(X_{\reg},ω)$ and its relation to
the global algebraic geometry of the projective variety $X$.  More precisely, we
will investigate the following.
\begin{itemize}
\item The holonomy group $G = \Hol(X_{\reg}, g)$, where $g$ is the Riemannian
  metric on $X_{\reg}$ induced by $ω$.
\item The algebra of global holomorphic forms
  $\cA = \bigoplus_p H⁰ \bigl(X_{\reg},\, Ω_{X_{\reg}}^{p} \bigr)$.
\item The fundamental group $π_1(X_{\reg})$.
\end{itemize}

Motivated by a decomposition theorem for tangent sheaves of klt varieties with
numerically trivial canonical divisor established by Greb, Kebekus, and
Peternell in \cite{GKP16} and building on Bost's criteria for algebraic
integrability of foliations, Druel recently obtained the singular version of the
Decomposition Theorem for such varieties of dimension $\dim X ≤ 5$, see
\cite{Dru16}.  Using Druel's strategy as well as the results presented in this
paper, in particular Proposition~\ref{prop:new_intro_prop} and
Theorem~\ref{thm:C} below, Höring and Peternell very recently gave a proof of
the singular version of the Decomposition Theorem in
\cite{HoeringPeternellbbdecomp}, thus completing the long quest for such a
result in the singular category.

\subsection{The Bochner principle}
\approvals{Daniel & yes \\ Henri & yes \\ Stefan & yes}

There exist strong connections between the three objects listed above.  First of
all, it is in general significantly simpler to compute the neutral component
$G°$ of $G$, as $G°$ is invariant under finite étale covers, and since all
possible isomorphism classes were classified by Berger-Simons.  The group of
connected components, $G/G°$, is then controlled by the fundamental group of the
variety via the canonical surjection $π_1(X_{\reg}) \twoheadrightarrow G/G°$.
Finally, the link between $G$ and $\mathcal A$ is provided in the smooth case by
the Bochner principle, which is a straightforward application of the maximum
principle to Bochner's formula.  One of our main results consists in the
following generalisation of the Bochner principle to the singular setting.

\begin{bigthm}[Bochner principle, Theorem~\ref{thm:holonomy}]\label{thm:A}
  Let $X$ be a projective klt variety with $K_X$ numerically trivial.  Let $H$
  be an ample divisor on $X$, and let $ω_H ∈ c_1(H)$ be the singular Ricci-flat
  Kähler metric constructed by \cite{MR2505296}, with associated Riemannian
  metric $g_H$ on $X_{\reg}$.  Then, every holomorphic tensor on $X_{\reg}$ is
  parallel with respect to $g_H$.
\end{bigthm}

The proof of this result is much more involved than in the smooth case.  We
first establish a Bochner principle for subbundles of tensor bundles,
Theorem~\ref{generalized:holonomy}, using the analysis developed in
\cite{Guenancia}.  To obtain Theorem~\ref{thm:A}, this is subsequently combined
with group-theoretic arguments and with the existence of a certain ``holonomy
cover'', which we explain in Theorem~\ref{thm:B} below.

\subsection{The holonomy cover}
\approvals{Daniel & yes \\ Henri & yes \\ Stefan & yes}

As we explained above, it is difficult in general to compute the full holonomy
group, but rather easy to get our hands on its neutral component $G°$, cf.\
Proposition~\ref{prop:factors}.  In the smooth setup, the passage from this
component to the full holonomy group is facilitated by the \emph{a priori
  control} over the fundamental group of Ricci-flat manifolds given by the
Cheeger-Gromoll Theorem.  In our setup, we have the following major potential
problem: even if the (restricted) holonomy of the metric on $X_{\reg}$ has no
flat factors, the fundamental group $π_1(X_{\reg})$ might be infinite, and it is
therefore not clear that we can make the holonomy group connected by taking a
\textit{finite} étale cover of $X_{\reg}$.

However, we can overcome this potential topological obstruction by relying on
recent progress in higher-dimensional algebraic geometry.  Our two main
technical ingredients for this part are Druel's integrability theorem,
\cite[Thm.~1.4]{Dru16}, and the theorem on the existence of maximally
quasi-étale covers of klt varieties, \cite[Thm.~1.5]{GKP13}.  Combining these
results with more elementary differential-geometric considerations, we get the
following.

\begin{bigthm}[Holonomy cover, Theorem~\ref{thm:holonomyCover} and Proposition~\ref{prop:torusCover:decomposition}]\label{thm:B}
  Setting as in Theorem~\ref{thm:A}.  Then, there exist normal projective
  varieties $A$ and $Z$, and a quasi-étale cover $γ: A⨯Z → X$ such that the
  following properties hold.
  \begin{enumerate}
  \item\label{il:B1} The variety $A$ is Abelian, of dimension
    $\dim A = \wtilde{q}(X)$, the augmented irregularity of $X$.
  \item\label{il:B2} The variety $Z$ has canonical singularities, linearly
    trivial canonical divisor, and augmented irregularity $\wtilde{q}(Z)=0$.
  \item\label{il:B3} There exist a flat Kähler form $ω_A$ on $A$ and a singular
    Ricci-flat Kähler metric $ω_Z$ on $Z$ such that
    $γ^*ω_H ≅ \pr_1^*ω_A + \pr_2^*ω_Z$ and such that the holonomy group of
    the corresponding Riemannian metric on $A ⨯ Z_{\reg}$ is connected.
    \end{enumerate}
\end{bigthm}

\begin{reminder}
  The augmented irregularity of $X$ is defined by Kawamata to be the maximal
  irregularity of any quasi-étale cover of $X$, see Section~\ref{ssec:augIrreg}.
\end{reminder}

The singular Kähler-Einstein metric $ω_H$ constructed by \cite{MR2505296} does
depend on the choice of the ample divisor $H$.  However, using
Theorems~\ref{thm:A} and \ref{thm:B} to relate holonomy, restricted holonomy,
and holomorphic differential forms, we will show that the isomorphism class of
the restricted holonomy group $G°$ is in fact independent of $H$, allowing us to
speak of \emph{the} restricted holonomy.  The following proposition makes this
precise.

\begin{bigprop}[Restricted holonomy is independent of polarisation, Corollary~\ref{cor:independence}]\label{prop:independence}
  Setting as in Theorem~\ref{thm:A}.  Then, the isomorphism class of the
  restricted holonomy group $\Hol(X_{\reg},g_H)°$ does not depend on the ample
  polarisation $H$.
\end{bigprop}

We will see in Section~\ref{univpropegz} that the construction of Ricci-flat
Kähler metrics by \cite{MR2505296} is well-behaved under quasi-étale cover, and
then so is the restricted holonomy.  In contrast, Section~\ref{ssec:loupbir}
shows by way of example that the restricted holonomy changes dramatically under
birational modifications, even under crepant blowing up.

\subsection{Decomposition of the tangent sheaf}
\approvals{Daniel & yes \\ Henri & yes \\ Stefan & yes}

Using the holonomy principle as well as the classification of Ricci-flat
restricted holonomy groups, and comparing with the decomposition of the tangent
sheaf of varieties with trivial canonical divisor established in \cite{GKP16},
we can obtain more precise information about the tangent sheaf of $Z$, which
splits according to the restricted holonomy representation of $X$.

\begin{bigprop}[Decomposition of the tangent sheaf, Proposition~\ref{prop:nnnew_intro_prop}]\label{prop:new_intro_prop}
  Setup and notation as in Theorem~\ref{thm:B}.  Then, the cover $γ : A ⨯ Z → X$
  can be chosen such that in addition to properties \ref{il:B1}--\ref{il:B3},
  there exists a direct sum decomposition of the tangent sheaf of $Z$,
  $$
  𝒯_Z = \bigoplus\nolimits_{i∈ I} ℰ_i ⊕ \bigoplus\nolimits_{j∈ J}ℱ_j,
  $$
  where the reflexive sheaves $ℰ_i$ (resp.\ $ℱ_j$) satisfy the following
  properties.
  \begin{enumerate}
  \item\label{il:new_i1} The subsheaves $ℰ_i ⊆ 𝒯_Z$ (resp.\ $ℱ_j ⊆ 𝒯_Z$) are
    foliations with trivial determinant, of rank $n_i ≥ 3$ (resp.\ of even rank
    $2m_j ≥ 2$).  Moreover, they are strongly stable in the sense of
    \cite[Def.~7.2]{GKP16}.
  \item\label{il:new_i2} On $Z_{\reg}$, the $ℰ_i$ (resp.\ $ℱ_j$) are locally
    free and correspond to holomorphic subbundles $E_i$ (resp.~$F_j$) of
    $TZ_{\reg}$ that are parallel with respect to the Levi-Civita connection of
    $g_Z$, the Riemannian metric on $Z_{\reg}$ induced by $ω_Z$.  Moreover,
    their holonomy groups are $\SU(n_i)$ and $\Sp(m_j)$, respectively.
  \item\label{il:new_i3} If $x ∈ X_{\reg}$ and $(a,z) ∈ γ^{-1}(x)$, then the
    splitting
    $$
    T_x X ≅ T_{(a,z)} (A ⨯ Z) = T_a A ⊕ \bigoplus\nolimits_{i∈ I} E_{i,z}
    ⊕ \bigoplus\nolimits_{j∈ J}F_{j,z}
    $$
    corresponds to the decomposition of $T_x X$ into irreducible representations
    under the action of the restricted holonomy group $\Hol(X_{\reg},g_H)$ at
    $x$.
  \end{enumerate}
\end{bigprop}

Proposition~\ref{prop:new_intro_prop} is a significant refinement of the
decompositions obtained in \cite{GKP16} and \cite{Guenancia} and, as already
mentioned above, is one of the ingredients in Höring-Peternell's proof for the
singular analogue of the Decomposition Theorem in any dimension, see
\cite{HoeringPeternellbbdecomp}.  The irreducible pieces appearing in their
result are the ones described in \cite[Sect.~8.B]{GKP16}.  We will discuss these
in the next paragraph.

\subsection{Irreducible pieces of the decomposition}
\approvals{Daniel & yes \\ Henri & yes \\ Stefan & yes}

Smooth Calabi-Yau manifolds and irreducible holomorphic symplectic manifolds are
defined by two conditions: one is algebraic, expressed in terms of the algebra
of holomorphic forms, and the other one is topological, namely simple
connectedness.  In particular, their holonomy group is connected and their
algebra of holomorphic forms cannot be made any larger by taking finite étale
covers.  In the singular setting, \cite[Def.~8.16]{GKP16} proposed the following
purely algebro-geometric definition.

\begin{defn}[CY and IHS]\label{def:CY_and_IHS}
  Let $X$ be a normal projective variety with $𝒪_X ≅ ω_X$ of dimension at
  least two, having at worst canonical singularities.
  \begin{enumerate}
  \item We call $X$ \emph{Calabi-Yau (CY)} if
    $H⁰ \bigl(Y, Ω_{Y}^{[p]} \bigr) = \{0\}$ for all numbers $0 < p< \dim X$ and
    all finite, quasi-étale covers $Y → X$.
  \item We call $X$ \emph{irreducible holomorphic symplectic (IHS)} if there
    exists a holomorphic symplectic two-form $σ ∈ H⁰ \bigl(X, Ω_X^{[2]} \bigr)$
    such that for all finite, quasi-étale covers $γ: Y → X$, in particular for
    $X$ itself, the exterior algebra of global reflexive forms is generated by
    $γ^{[*]}σ$.
  \end{enumerate}
\end{defn}

The combination of Theorem~\ref{thm:A} and Theorem~\ref{thm:B} above answers the
natural question posed in \cite{GKP16} concerning the characterisation of these
two classes of varieties in terms of the strong stability of their tangent
sheaf.

\begin{bigthm}[Strongly stable varieties, Corollary~\ref{cor:dichotomy_CY_IHS}]\label{thm:C}
  Let $X$ be a projective klt variety of dimension at least two, with
  numerically trivial canonical divisor.  Assume that $𝒯_X$ is strongly stable
  in the sense of \cite[Def.~7.2]{GKP16}.  Then, there exists a quasi-étale
  cover $γ: Y → X$ such that $Y$ is Calabi-Yau or irreducible holomorphic
  symplectic.
\end{bigthm}

In conclusion, a variety with klt singularities and numerically trivial
canonical class admits a quasi-étale cover which is a CY or IHS variety if and
only if its tangent bundle is strongly stable.  Example~\ref{ex:no2form}
discusses a variety with canonical singularities, trivial canonical bundle and
with no reflexive forms of intermediate degree that admits a quasi-étale cover
which is an IHS variety.

In light of these results, one would like to propose an alternative definition
of Calabi-Yau and irreducible holomorphic symplectic varieties that does not
involve looking at quasi-étale covers.  In other words, one would like to
replace the assumption on simple connectedness.  We want to emphasise that there
does not seem to be an easy topological condition that would play its role in
the singular setting.
\begin{itemize}
\item The assumption ``$π_1(X)=\{1\}$'' is not the right one.
  Example~\ref{ex:kummer} discusses a singular Kummer surface that is is simply
  connected and has the same algebra of reflexive forms as a smooth $K3$.
  However, the example is a quotient of an Abelian variety.
\item The assumption ``$π_1(X_{\reg})=\{1\}$`` might seem like a good condition,
  but even in the IHS case, we do not know that $π_1(X_{\reg})$ is actually
  finite.  Even worse, in the CY case $π_1(X_{\reg})$ could \emph{a priori} be
  infinite with infinite completion for all we know.
\end{itemize}

However, due to the finiteness statement for $G/G°$ contained in
Theorem~\ref{thm:B}, a good condition to impose in place of simple connectedness
is \emph{holonomy connectedness}.  This leads to the following characterisation,
proven in Section~\ref{prop:page}.
  
\begin{bigprop}[Characterisation of CY and IHS by holonomy, Proposition~\ref{prop:intrinsicn}]\label{prop:intrinsic}
  Setting as in Theorem~\ref{thm:A}.  Then, the following conditions are
  equivalent.
  \begin{enumerate}
  \item\label{il:intr1} $X$ is a Calabi-Yau variety.
  \item\label{il:intr2} $\Hol(X_{\reg},g_H)$ is connected and
    $H⁰ \bigl( X, Ω_{ X}^{[p]} \bigr) = \{0\}$ for all $0 < p < n$.
  \item\label{il:intr3} $\Hol(X_{\reg},g_H)$ is isomorphic to $\SU(n)$.
  \end{enumerate}
  Analogously, the following conditions are equivalent.
  \begin{enumerate}
  \item\label{il:intr4} $X$ is an irreducible holomorphic symplectic variety.
  \item\label{il:intr5} $\Hol(X_{\reg},g_H)$ is connected, and there exists a
    holomorphic symplectic two-form $σ ∈ H⁰ \bigl(X, Ω_X^{[2]} \bigr)$ such that
    $⊕_{p=0}^n H⁰(X,Ω_X^{[p]})=ℂ[σ]$.
  \item\label{il:intr6} $\Hol(X_{\reg},g_H)$ is isomorphic to
    $\Sp\!\left(\frac n2\right)$.
  \end{enumerate}
\end{bigprop}
  
We refer to Section~\ref{sec:examples} where lots of examples are given that
compare our notions of CY and IHS varieties to some other ones existing in the
literature and that emphasise the subtlety of the dichotomy question once
singularities are allowed.

\subsection{Stability and irreducible (restricted) holonomy}
\approvals{Daniel & yes \\ Henri & yes \\ Stefan & yes}

Theorem~\ref{thm:B} and the methods of its proof also establish correspondences
between the algebro-geometric notion of stability and irreducibility of the
differential-geometric holonomy representation.

\begin{bigprop}[Stability and irreducibility, Corollary~\ref{cor:sirr0} and \ref{cor:sirr2}]
  Setting as in Theorem~\ref{thm:A}.  Let $x∈ X_{\reg}$, let $G$ (resp.\ $G°$)
  be the holonomy group (resp.\ restricted holonomy group) of $(X_{\reg}, g_H)$
  at $x$, and let $V:=T_{x}X$.  Then, $𝒯_X$ is stable with respect to any ample
  polarisation (resp.\ strongly stable) if and only if the representation $G↺ V$
  (resp.\ $G°↺~V$) is irreducible.
\end{bigprop}

The second result provides an algebro-geometric criterion for the vanishing of
the augmented irregularity of a variety $X$ in terms of properties of
holomorphic tensors on $X$ itself.  This is achieved by considering invariants
of the restricted holonomy representation.

\begin{bigthm}[Augmented regularity and symmetric differentials, Theorem~\ref{thm:augm_irreg_rev}]\label{thm:augm_intro}
  Setting as in Theorem~\ref{thm:A}.  Let $x∈ X_{\reg}$, let $G°$ be the
  restricted holonomy group of $(X_{\reg}, g_H)$ at $x$, and let $V:=T_{x}X$.
  Then, the following are equivalent.
  \begin{enumerate}
  \item We have $H⁰\bigl(X,\, \Sym^{[m]} Ω¹_X \bigr)=\{0\}$ for all $m∈ ℕ^+$.
  \item The augmented irregularity of $X$ vanishes, $\wtilde{q}(X) = 0$.
  \item The set of $G°$-invariant vectors in $V$ is trivial, $V^{G°} = \{0\}$.
  \end{enumerate}
\end{bigthm}

\subsection{Fundamental groups}
\approvals{Daniel & yes \\ Henri & yes \\ Stefan & yes}

In the smooth case, Theorem~\ref{thm:bbd} shows that the fundamental group of
$X$ is virtually Abelian, that is, an extension of $ℤ^r$ by a finite group.  In
particular, if the augmented irregularity of $X$ vanishes, $\wtilde{q}(X)=0$,
then $π_1(X)$ is actually finite.  The proof of this fundamental result relies
on the Cheeger-Gromoll theorem, which says that once an Euclidean space of
maximal dimension is split off the universal cover of a complete Ricci-flat
manifold $X$, the remaining factor is compact.  As already mentioned above, the
proof of the Cheeger-Gromoll theorem uses completeness $X$ in a fundamental way,
so none of these methods apply to the non-complete Riemannian manifolds
$(X_{\reg}, g_H)$ considered here.

Fortunately, the algebraicity of $X$ opens the door to alternative techniques,
which allow us to prove the following finiteness result.

\begin{bigthm}[Finiteness of $π_1$, Theorem~\ref{thm:finite_fundamental}, \ref{thm:oddCY}, Corollary~\ref{cor:finite_fundamental}, \ref{cor:IHS_simplyconnected}, \ref{thm:oddCY2}.]\label{thm:D}
  Let $X$ be a projective klt variety with numerically trivial canonical
  divisor.
  \begin{enumerate}
  \item\label{pi1i} If $𝒯_X$ is strongly stable and if $\dim X$ is even, then
    $π_1(X)$ and $\what{π}_1(X_{\reg})$ are finite.  If $X$ is IHS or an
    even-dimensional CY, then $X$ is simply connected.
  \item\label{pi1ii} If $\wtilde{q}(X)=0$, then $π_1(X_{\reg})$ does not admit
    any finite-dimensional representation with infinite image (over any field).
    Moreover, for each $n ∈ ℕ$, the fundamental group $π_1(X_{\reg})$ admits
    only finitely many $n$-dimensional complex representations up to
    conjugation.
  \end{enumerate}
\end{bigthm}

\begin{rem}
  Under the assumption that $\wtilde{q}(X)=0$, the conclusions of
  Item~\ref{pi1ii} remain valid for $π_1(X)$ as well, given the surjection
  $π_1(X_{\reg})\twoheadrightarrow π_1(X)$ induced by the open immersion
  $X_{\reg}↪ X$, cf.~\cite[0.7.B on p.~33]{FL81}.
\end{rem}

About Item~\ref{pi1i}: the proof of finiteness of $π_1(X)$ strongly relies on
the Bochner principle, and on methods introduced by Campana in \cite{Ca95},
which connect positivity properties of cotangent bundles to the size of
fundamental groups.  To pass from $X$ to $X_{\reg}$, we consider a maximally
quasi-étale cover.

As for Item~\ref{pi1ii}, the key point is a result of
Brunebarbe-Klingler-Totaro, \cite{MR3127814}, which links the existence of
representations with infinite image to the existence of symmetric differentials.
This, in turn, can be interpreted via the Bochner principle in terms of
invariant vectors in the symmetric power of the standard representation of
$\SU(n)$ or $\Sp(n)$, see Theorem~\ref{thm:augm_intro} above.

\subsection{Outline of the paper}
\approvals{Daniel & yes \\ Henri & yes \\ Stefan & yes}

The core of the paper consists in proving Theorems~\ref{thm:A} and \ref{thm:B}
above.  Interestingly enough, and unlike the strategy executed in the smooth
case, the proof of Theorem~\ref{thm:A} relies on the conclusions of
Theorem~\ref{thm:B}, which is proved in Part~\ref{part:II}.  Part~\ref{part:III}
is devoted to proving Theorem~\ref{thm:A} using the results achieved in the
earlier parts.  More precisely, the content of the individual sections can be
summarised as follows.

\subsubsection*{Part~\ref*{part:I}}
\approvals{Daniel & yes \\ Henri & yes \\ Stefan & yes}

We recall the definitions and basic properties of the fundamental
differential-geometric objects that will be used throughout the paper, namely
the holonomy groups of a Kähler manifold and singular Kähler-Einstein metrics.
We also analyse the behaviour of these objects with respect to quasi-étale
covers.

\subsubsection*{Part~\ref*{part:II}}
\approvals{Daniel & yes \\ Henri & yes \\ Stefan & yes}

This part is mostly taken up by the proof of Theorem~\ref{thm:holonomyCover}.
The starting point is the classification of restricted holonomy,
Proposition~\ref{prop:factors}, which shows that the usual dichotomy $\SU$ vs.\
$\Sp$ continues to hold in the singular setting.

\subsubsection*{Part~\ref*{part:III}}
\approvals{Daniel & yes \\ Henri & yes \\ Stefan & yes}

We prove Theorem~\ref{generalized:holonomy} stating that $𝒯_X$ or more generally
any of its reflexive tensor powers is the direct orthogonal sum of stable
parallel subbundles; the arguments follows \cite{Guenancia} closely.
Capitalising on this and on the results of Part~\ref{part:II}, we establish the
Bochner principle for forms and unfold a first list of applications in
connection with augmented irregularity and characterisations of quotients of
Abelian varieties, cf.~Theorem~\ref{thm:augm_irreg_rev} as well as
Corollaries~\ref{cor:forms} and \ref{qab1}.

\subsubsection*{Part~\ref*{part:IV}}
\approvals{Daniel & yes \\ Henri & yes \\ Stefan & yes}

We investigate the strongly stable case.  Corollary~\ref{cor:dichotomy_CY_IHS}
explains the relation to CY and IHS varieties in more detail.
Section~\ref{sec:alghol} relates our results to the algebraic holonomy group
introduced by Balaji and Kollár.  A number of examples illustrate the complexity
of the situation.  Finally, we study finiteness properties for the fundamental
groups (resp.\ algebraic fundamental groups) of $X$ and $X_{\reg}$.

\subsection{Acknowledgements}
\approvals{Daniel & --- \\ Henri & yes \\ Stefan & yes}

We would like to thank Stéphane Druel, Jochen Heinloth, Mihai Păun, Wolfgang
Soergel, Matei Toma and Burt Totaro for helpful discussions, and user
\href{https://mathoverflow.net/users/43122/holonomia}{Holonomia} on
\href{https://mathoverflow.net}{mathoverflow.net} for answering our questions.
We thank the anonymous referee for her/his time and efforts.

\part{Preparations}\label{part:I}
%
%
\svnid{$Id: 02-notation.tex 758 2018-10-01 15:48:05Z guenancia $}

\section{Notation and conventions}\label{sec:notation}
\subversionInfo
\approvals{Daniel & yes \\ Henri & yes \\ Stefan & yes}

The main motivation to study varieties with klt singularities comes from Minimal
Model theory.  On the other hand, many of our techniques originate in
Differential Geometry.  Hence, to make the paper more accessible for algebraic
geometers, this chapter carefully sets up notation and gives very quick
explanations of the terminology and of the fundamental principles used later.

\subsection{Global conventions}
\approvals{Daniel & yes \\ Henri & yes \\ Stefan & yes}

Throughout the present paper, all varieties will be defined over the complex
numbers.  We will freely switch between the algebraic and analytic context if no
confusion is likely to arise.  If extra care is warranted, we denote the
analytic space associated with an algebraic variety $X$ by $X^{an}$.  We follow
the notation used in the standard reference books \cite{Ha77, KM98}.  In
particular, varieties are always assumed to be irreducible and reduced.

\subsection{Differential-geometric notions}

\subsubsection{Differentials and vector fields}
\approvals{Daniel & yes \\ Henri & yes \\ Stefan & yes}

Throughout this paper, we will clearly distinguish between bundles and their
associated sheaves of smooth (resp.\ holomorphic) sections.

\begin{notation}[Tangent bundle]\label{not:Apq1}
  Given a connected complex manifold $X$ of complex dimension $n$, denote the
  holomorphic tangent bundle by $TX$, the holomorphic tangent sheaf by $𝒯_X$,
  and the sheaf of $\cC^∞$-sections of $TX$ as $\cT_X$.  If $x ∈ X$ is any
  point, let $T_xX$ be the associated tangent space, $\dim_ℂ T_xX = n$.  The
  complexified tangent bundle decomposes as $TX^{ℂ} = T^{1,0}X ⊕ T^{0,1}X$,
  where $TX$ and $T^{1,0}X$ are naturally isomorphic as complex vector bundles.
\end{notation}

\begin{notation}[$\cC^∞$-functions and forms]\label{not:Apq2}
  Given a connected complex manifold $X$, let $\cA_X$ denote the sheaf of
  complex-valued $\cC^∞$-functions on $X$.  The symbols $\cA^p_X$ denote the
  sheaves of complex, $\cC^∞$-differential $p$-forms.  Likewise, $\cA^{p,q}_X$
  are the sheaves of complex, $\cC^∞$-differential of type $(p,q)$.  The sheaves
  of holomorphic differentials are denoted by $Ω^p_X$.  There are a few standard
  variants of the notation that we will also use.  If $E$ is a complex vector
  bundle, we denote the vector space of global, smooth, $E$-valued $(p,q)$-forms
  by $\cA^{p,q}(X,E)$.  If $E$ is equipped with a Hermitian metric $h$, we write
  $\cA^{p,q}(X,\End(E,h))$ for the space of forms with values in Hermitian
  endomorphisms.  Vector spaces of real forms will be denoted as
  $\cA^{p,q}_{ℝ}(•)$.
\end{notation}

\begin{notation}[Chern connections]\label{not:cc}
  Let $(X,ω)$ be a connected Kähler manifold of dimension $n := \dim_ℂ X$.  Let
  $h$ be the associated Hermitian metric on $T^{1,0}X$ and $g$ the associated
  Riemannian metric on $X$\Preprint{, cf.~\cite[App.~4.A]{Huy05}}.  Owing to the
  Kähler property, the Levi-Civita connection of $g$ coincides with the Chern
  connection of $h$ after identifying $TX$ and $T^{1,0}X$\Preprint{, see
    \cite[Prop.~4.A.9]{Huy05}}.  We obtain induced complex connections on all
  tensor bundles $TX^{⊗ p} ⊗ T^*X^{⊗ q}$, which by minor abuse of notation we
  all denote by $D$.
\end{notation}

\subsubsection{Holonomy}
\approvals{Daniel & yes \\ Henri & yes \\ Stefan & yes}

The holonomy group of a Riemannian manifold is the core notion of this paper.

\begin{notation}\label{nota:holonomy}
  Let $(M,g)$ be a connected Riemannian manifold.  Given a point $m ∈ M$, we
  view $(T_mM, g_m)$ as a Euclidean vector space and denote the associated
  Riemannian holonomy group by $\Hol(M, g)_m$, which we view as a subgroup of
  the orthogonal group $\operatorname{O}(T_mM, g_m)$.  Its identity component,
  the restricted holonomy group, will be denoted by $\Hol(M, g)°_m$.
\end{notation}

The following relation between the holonomy group and the topology of the
underlying manifold is crucial for our arguments.

\begin{reminder}[\protect{Holonomy and fundamental group, \cite[Prop.~2.3.4]{MR1787733}}]\label{rem:fundamentalgroupsurjection}
  Let $(M,g)$ be a connected Riemannian manifold.  Then, for any point $m ∈ M$,
  as the restricted holonomy group coincides with those elements in
  $\Hol(M, g)_m$ that arise via parallel transport along contractible paths, we
  have a surjective group homomorphism
  \begin{equation}\label{eq:fundamentalgroupsurjection}
    π_1(M, m) \twoheadrightarrow \factor{\Hol(M, g)_m}{\Hol(M, g)°_m}.
  \end{equation}
\end{reminder}

\Publication{In the preprint version of this paper,
  \href{https://arxiv.org/abs/1704.01408}{arXiv:1704.01408}, we recall some
  standard facts about holonomy including: the holonomy of a Kähler manifold,
  the holonomy principle and the correspondence between holonomy invariant
  subspaces and parallel subbundles.} 
\Preprint{
  \begin{reminder}[\protect{Holonomy of Kähler manifolds, \cite[Sect.~10.C]{MR867684}}]\label{remi:holonomy}
    Let $(X,ω)$ be a connected Kähler manifold.  Let $h$ be the associated
    Hermitian metric on $TX$.  Given any point $x ∈ X$, view $T_xX$ as a complex,
    Hermitian vector space.  Since the complex structure on $X$ is parallel, the
    Riemannian holonomy group $\Hol(X, g)_x$ is contained in the unitary group
    $\U(T_xX)$.  The Kähler manifold $(X,ω)$ is Ricci-flat if and only if the
    restricted holonomy group $\Hol(X, g)°_x$ is contained in $\SU(T_xX)$.
  \end{reminder}
  
  \begin{reminder}[\protect{The holonomy principle for complex tensors, \cite[Fundamental Principle~10.19]{MR867684}}]\label{remi:cc}
    Let $(X,ω)$ be a connected Kähler manifold with associated Riemannian metric
    $g$.  Choose $x ∈ X$ and consider the holonomy group $\Hol(X, g)_x$.  The
    representation of $\Hol(X, g)_x$ on $T_xX$ naturally can be complexified to
    a representation on $T_x^ℂ X$, and as the original representation commutes
    with the complex structure $J_x ∈ \End(T_xX)$, it also induces
    representations on $T^{1,0}_xX$ and $T_x^{0,1}X$.  The aforementioned
    isomorphism of $TX$ and $T^{1,0}X$ yields an isomorphism of representations
    $T_x X ≅ T^{1,0}_xX$, and, as the representation of $\Hol(X, g)_x$ on
    $T_x X$ is unitary with respect to $h_x$, we hence obtain equivariant
    isomorphisms $T_x^* X ≅ \overline{T_x^{1,0}X} = T_x^{0,1}X$.  By
    standard linear algebra, we obtain induced representations on the tensor
    spaces $T_x X^{⊗ p} ⊗_{ℂ} T_x^* X^{⊗ q}$.  On the other hand, we have seen
    in Notation~\ref{not:cc} that the Levi-Civita- resp.~Chern-connection on
    $TX$ resp.~$T^{1,0}X$, cf.~Notation~\ref{not:cc}, induces connections $D$ on
    all the complex tensor bundles $TX^{⊗ p} ⊗ T^*X^{⊗ q}$.
    
    The \emph{holonomy principle} states that evaluation at $x$ and parallel
    transport establish a $ℂ$-linear one-to-one correspondence between tensors
    $t ∈ H⁰\bigl(X, \cT_X^{⊗ p} ⊗_{\cA_X} (\cT_X^*)^{⊗ q} \bigr)$ with vanishing
    covariant derivative, $D t =0$, which are automatically holomorphic as the
    $(0,1)$-part of $D$ coincides with the complex structure $\bar{∂}$, and
    invariant vectors $t_x$ in the corresponding unitary
    $\Hol(X, g)_x$-representation $T_x X^{⊗ p} ⊗_{ℂ} T_x^* X^{⊗ q}$.  For more
    on the compatibility of parallel transport and the holomorphic structure, we
    refer to Section~\ref{subsect:parallelity_and_holomorphy} below.
  \end{reminder}}

The symplectic group appears prominently in the classification of holonomy
groups.  We use the following convention.

\begin{notation}[Symplectic groups]
  Let $n∈ ℕ^+$ and let $σ := \sum_{k=1}^n dz_k Λ dz_{n+k}$ be the standard
  complex symplectic form on $ℂ^{2n}$.  We denote by $\Sp(n,ℂ)$ the
  \emph{complex symplectic group}, that is, the subgroup of $\GL(2n,ℂ)$
  consisting of transformations preserving $σ$.  We denote by $\Sp(n)$ its
  compact real form, $\Sp(n):=\Sp(n, ℂ) ∩ \U(2n)$, the \emph{unitary symplectic
    group}.
\end{notation}

\subsection{Local decomposition of Kähler manifolds}
\approvals{Daniel & yes \\ Henri & yes \\ Stefan & yes}

The relation between holonomy groups and stability properties of the tangent
sheaves will be established by using the following folklore result, which we
include here for lack of a reference\Publication{, and which we furthermore
  prove in the preprint version of this paper}.

\begin{prop}[Local decomposition of Kähler manifolds]\label{prop:locdecomp}
  Let $(X, ω)$ be a simply connected Kähler manifold (not necessarily compact or
  complete) and $x ∈ X$ a point.  Then there exists an open neighbourhood
  $U=U(x)$, Kähler manifolds $(U_i, ω_i)_{i=0…m}$ and an isomorphism of Kähler
  manifolds
  \begin{equation}\label{eq:isom}
    φ : (U,ω) → \left( \bigtimes^m_{i=0} U_i,\, \sum_{i=0}^m π_i^* ω_i \right)
  \end{equation}
  such that the following holds when we write $(x_0, …, x_m)$ for $φ(x)$ and
  $g$, $g_i$ for the associated Riemannian metrics on $X$ and on the $U_i$,
  respectively.
  \begin{enumerate}
  \item\label{il:p1} The action of the holonomy group
    $H := \Hol\bigl( X,\, g\bigr)_x$ on $T_xX$ respects the orthogonal
    decomposition $T_xX = \bigoplus_{i=0}^m T_{x_i}U_i$ induced by $φ$.
  \item\label{il:p2} The holonomy group is a direct product,
    $H = ⨯_{i=1}^m H_i$, where each factor $H_i$ acts irreducibly on the summand
    $T_{x_i}U_i$ and trivially on all the other summands.\Publication{\qed}
  \end{enumerate}
\end{prop}

\begin{rem}
  The summand $T_{x_0}U_0$ is precisely the set of $H$-fixed vectors.
\end{rem}

\Preprint{
  \begin{proof}[Proof of Proposition~\ref{prop:locdecomp}]
    Apply \cite[Thm.~5.4 on p.~185]{MR1393940} to the Riemannian manifold
    $(X,g)$, in order to find a neighbourhood $U = U(x) ⊆ X$ that admits an
    isometry to a Riemannian product, $φ : U → ⨯ U_i$, satisfying \ref{il:p1}
    and \ref{il:p2}.  The proof of \cite[Thm.~8.1 on p.~172]{MR1393941} applies
    nearly verbatim to show that the $U_i$ carry natural complex structures with
    respect to which the $g_i$ are Kähler and that make the isometry $φ$
    biholomorphic.
  \end{proof}}

\begin{defn}[Totally decomposed actions]\label{def:totDecomp}
  If $V$ is a Hermitian vector space and $G ⊆ U(V)$ any group, we call the
  action $G ↺ V$ \emph{totally decomposed} if there exists a $G$-invariant
  orthogonal decomposition $V = V_0 ⊕ V_1 ⊕ ⋯ ⊕ V_m$, and a product
  decomposition, $G = G_1 ⨯ ⋯ ⨯ G_m$, where each factor $G_i$ acts non-trivially
  and irreducibly on $V_i$ and trivially on all the other summands.
\end{defn}

\begin{rem}[Uniqueness of decompositions]\label{rem:totDecomp}
  If $V$ is a unitary vector space and $G ⊆ U(V)$ any group, there are usually
  many $G$-invariant orthogonal decompositions of $V$ into complex vector
  spaces, $V = V_0 ⊕ V_1 ⊕ ⋯ ⊕ V_m$, where $G$ acts trivially on $V_0$ and
  non-trivially and irreducibly on the remaining summands.  If the $G$-action is
  totally decomposed, then the decomposition is unique up to permutation of the
  $(V_i)_{i>0}$.  For a proof, observe that the product structure of $G$ implies
  that a vector $\vec v ∈ V$ is contained in one of the $V_i$ if and only if the
  linear span of its $G$-orbit is an irreducible representation space (and then
  equal to $V_i$).  The product structure of $G$ is also unique, as
  $G_i = ∩_{j≠ i} \Fix(V_j)$.
\end{rem}

\subsection{Varieties and sets}
\approvals{Daniel & yes \\ Henri & yes \\ Stefan & yes}

Normal varieties are $S_2$, which implies that regular functions can be extended
across sets of codimension two.  The following notation will be used.

\begin{notation}[Big and small subsets]
  Let $X$ be a normal, quasi-projective variety.  A Zariski-closed subset
  $Z ⊂ X$ is called \emph{small} if $\codim_X Z ≥ 2$.  A Zariski-open subset
  $U ⊆ X$ is called \emph{big} if $X∖ U$ is small.
\end{notation}

\subsection{Morphisms}
\approvals{Daniel & yes \\ Henri & yes \\ Stefan & yes}

Galois morphisms appear prominently in the literature, but their precise
definition is not consistent.  We will use the following definition, which does
not ask Galois morphisms to be étale.

\begin{defn}[Covers and covering maps, Galois morphisms]\label{def:cover}
  A \emph{cover} or \emph{covering map} is a finite, surjective morphism
  $γ : Y → X$ of normal, quasi-projective varieties.  The covering map $γ$ is
  called \emph{Galois} if there exists a finite group $G ⊆ \Aut(Y)$ such that
  $X$ is isomorphic to the quotient map $Y → Y/G$.
\end{defn}

\begin{defn}[Quasi-étale morphisms]\label{defn:quasietale}
  A morphism $γ : Y → X$ between normal varieties is called \emph{quasi-étale}
  if $γ$ is of relative dimension zero and étale in codimension one.  In other
  words, $γ$ is quasi-étale if $\dim Y = \dim X$ and if there exists a closed,
  subset $Z ⊆ Y$ of codimension $\codim_Y Z ≥ 2$ such that
  $γ|_{Y ∖ Z} : Y ∖ Z → X$ is étale.
\end{defn}

\begin{reminder}[KLT is invariant under quasi-étale covers]\label{remi:qec}
  Let $γ: Y → X$ be a quasi-étale cover.  By definition, $γ$ is then
  quasi-étale, finite and surjective.  If $K_X$ is $ℚ$-Cartier, then
  $K_Y = γ^* K_X$ is $ℚ$-Cartier as well.  If there exists a $ℚ$-divisor $D$ on
  $X$ that makes $(X,D)$ is klt, then $(Y,γ^*D)$ is klt as well,
  \cite[Prop.~5.20]{KM98}.
\end{reminder}

\begin{rem}[Quasi-étale covers vs.~étale covers of $X_{\reg}$]\label{rem:qevecxr}
  If $γ: Y → X$ is any quasi-étale cover, purity of the branch locus implies
  that $γ$ is étale over $X_{\reg}$.  Conversely, if $γ° : Y° → X^{an}_{\reg}$
  is a finite and locally biholomorphic morphism of complex manifolds, recall
  from \cite[Sect.~XII.5]{SGA1} that $γ°$ and $Y°$ are algebraic, and can be
  compactified to a quasi-étale cover $γ : Y → X$, see also
  \cite[Thm.~3.4]{DethloffGrauert}.  Consequently, we obtain obvious
  equivalences between the categories of quasi-étale covers of $X$, of finite
  and étale covers of $X_{\reg}$, and of finite sets with transitive action of
  $π_1(X^{an}_{\reg})$.
\end{rem}

Quasi-étale morphisms appear naturally in the context considered in this paper,
as exemplified by the following important result.

\begin{prop}[Global index one cover]\label{prop:global_index_one_cover}
  Let $X$ be a projective, klt variety with numerically trivial canonical
  divisor, $K_X \equiv 0$.  Then, there exists a quasi-étale cover $γ: Y → X$
  such that $Y$ has canonical singularities and linearly trivial canonical
  divisor, $K_Y \sim 0$.
\end{prop}
\begin{proof}
  By Nakayama's partial solution of the Abundance Conjecture,
  \cite[Cor.~4.9]{Nakayama04}, there exists a number $m ∈ ℕ^+$ such that $m·K_Y$
  is linearly equivalent to zero.  Let $γ : Y → X$ be the associated global
  index-one cover, \cite[Def.~5.19]{KM98}, which is quasi-étale.  By
  Reminder~\ref{remi:qec}, the variety $Y$ is then klt.  Moreover, by
  construction $K_Y \sim 0$, and the singularities of $Y$ are hence canonical.
\end{proof}

\subsection{Sheaves}
\approvals{Daniel & yes \\ Henri & yes \\ Stefan & yes}

Reflexive sheaves are in many ways easier to handle than arbitrary coherent
sheaves, and we will therefore frequently take reflexive hulls.  The following
notation will be used.

\begin{notation}[Reflexive hull]
  Given a normal, quasi-projective variety $X$ and a coherent sheaf $ℰ$ on $X$
  of rank $r$, write
  $$
  Ω^{[p]}_X := \bigl(Ω^p_X \bigr)^{**}, \quad ℰ^{[m]} := \bigl(ℰ^{⊗ m}
  \bigr)^{**}, \quad \Sym^{[m]} ℰ := \bigl( \Sym^m ℰ \bigr)^{**}
  $$
  and $\det ℰ := \bigl( Λ^r ℰ \bigr)^{**}$.  Given any morphism $f : Y → X$,
  write $f^{[*]} ℰ := (f^* ℰ)^{**}$.
\end{notation}

\subsection{Augmented irregularity}\label{ssec:augIrreg}
\approvals{Daniel & yes \\ Henri & yes \\ Stefan & yes}

The irregularity of normal, projective varieties is generally not invariant
under quasi-étale maps, even if the varieties in question are klt.  The notion
of ``augmented irregularity'' addresses this issue.

\begin{defn}[Augmented irregularity]
  Let $X$ be a normal, projective variety.  We denote the irregularity of $X$ by
  $q(X) := h¹ \bigl( X,\, 𝒪_X \bigr)$ and define the \emph{augmented
    irregularity} as
  $$
  \wtilde{q}(X) := \sup \bigl\{ q(Y) \,\bigl|\, Y \text{ a quasi-étale cover of
  } X \bigr\} ∈ ℕ ∪ \{∞\}.
  $$
\end{defn}

\begin{lem}[Augmented irregularity and quasi-étale covers]\label{lem:aug}
  Let $X$ be a normal, projective variety.  Then, the following holds.
  \begin{enumerate}
  \item\label{il:AUG2} The augmented irregularity is invariant under quasi-étale
    covers.  More precisely, if $Y → X$ is quasi-étale, then
    $\wtilde{q}(Y) = \wtilde{q}(X)$.
  \item\label{il:AUG1} If $X$ is a projective, klt variety with numerically
    trivial canonical class, then $\wtilde{q}(X) ≤ \dim X$.  In particular, the
    augmented irregularity is finite in this case.
  \item\label{il:AUG3} If $\wtilde{X} → X$ is a birational morphism of
    projective varieties with canonical singularities, and if $K_{\wtilde{X}}$
    is numerically trivial, then $\wtilde{q}(\wtilde{X}) ≤ \wtilde{q}(X)$.
  \end{enumerate}
\end{lem}
\begin{proof}
  To prove Item~\ref{il:AUG2}, recall from \cite[Lem.~4.1.14]{Laz04-I}
  (``Injectivity Lemma'') that the irregularity increases in covers and observe
  that any two quasi-étale covers are dominated by a common third.
  Item~\ref{il:AUG3} is shown in \cite[Lem.~4.4]{Dru16}.  To prove
  Item~\ref{il:AUG1}, let $γ: Y → X$ be a global index one cover, whose
  existence is guaranteed by Proposition~\ref{prop:global_index_one_cover}.  We
  may then apply \cite[Rem.~3.4]{GKP16} and Item~\ref{il:AUG2} to conclude.
\end{proof}

%
%
\svnid{$Id: 03-KE.tex 758 2018-10-01 15:48:05Z guenancia $}

\section{Singular Kähler-Einstein metrics}
\label{Section:KE}
\subversionInfo
\approvals{Daniel & yes \\ Henri & yes \\ Stefan & yes}

In this section, we recall the construction of Eyssidieux, Guedj and Zeriahi
\cite{MR2505296}\footnote{but see also \cite{MR0944606, CaLa, Tian-Zhang,
    Zhang06, Paun, DemPal} for further contributions}, which produces so-called
\emph{singular Kähler-Einstein metrics} on certain singular varieties such as
projective klt varieties with torsion canonical bundle.  These objects induce
genuine Kähler-Einstein metrics on the regular locus of the variety, and we will
study the resulting holonomy groups.

\subsection{Existence of Ricci-flat Kähler metrics}
\approvals{Daniel & yes \\ Henri & yes \\ Stefan & yes}

Let $X$ be a projective, klt variety whose canonical class $K_X$ is numerically
trivial.  In this setting, Eyssidieux, Guedj and Zeriahi have shown that each
Kähler class $α ∈ H²\bigl( X,\, ℝ \bigr)$ contains a unique Ricci-flat Kähler
metric $ω_α$ that is smooth on $X_{\reg}$, satisfies $\Ric ω_α = 0$ there, and
has bounded potentials near the singularities.  To give a precise account of
their construction, we need to specify notation first.  We refer to reader to
\cite[Sect.~7.1]{MR2505296} and \cite[Sect.~4.6.2]{BEG} for proofs and further
details.

\begin{notation}\label{not:maiset}
  Let $X$ be a projective, klt variety whose canonical class $K_X$ is
  numerically trivial.  By Proposition~\ref{prop:global_index_one_cover} there
  exists a positive number $m ∈ ℕ^+$ such that $m·K_X$ is linearly trivial.
  Choose one such $m$, choose a global generator
  $σ ∈ H⁰ \bigl( X,\, 𝒪_X(m·K_X) \bigr)$, choose a constant Hermitian metric
  $\|·\|$ on the trivial bundle $𝒪_X(m·K_X)$ and consider the following volume
  form on $X_{\reg}$,
  \begin{equation}\label{eq:rvform}
    i^n · (-1)^{ \frac{n(n+1)}{2}}· \left(\frac{σ Λ \overline{σ}}{\|σ\|²}\right)^{\frac 1m} \quad\text{where } n := \dim X.
  \end{equation}
  Finally, let $μ_X$ be the associated positive measure on $X$ obtained as the
  trivial extension of the measure on $X_{\reg}$ associated with the volume form
  in \eqref{eq:rvform} above.
\end{notation}

\begin{rem}[Independence on choices]\label{rem:measure_choiceindependent}
  The volume form in \eqref{eq:rvform} and the measure $μ_X$ are easily seen to
  be independent of the choice of $m$ and $σ$.  The measure $μ_X(X)^{-1}·μ_X$ is
  independent of the number $m$, the form $σ$ and of the choice of the constant
  Hermitian metric $\|·\|$.
\end{rem}

\begin{thm}[\protect{Existence of Ricci-flat Kähler metrics, cf.~\cite[Thm.~7.5]{MR2505296}}]\label{thm:EGZ}
  Let $X$ be a projective, klt variety whose canonical class $K_X$ is
  numerically trivial.  Given any ample $H ∈ \Div(X)$ with class
  $[H] ∈ H² \bigl(X,\, ℝ \bigr)$, there exists a unique closed positive current
  $ω_H$ on $X$ such that the following holds.
  \begin{enumerate}
  \item\label{il:emerson} Denoting the de Rham cohomology class of the current
    $ω_H$ by $[ω_H]$, we have an equality of cohomology classes,
    $[ω_H]= [H] ∈ H² \bigl(X ,\, ℝ \bigr)$.
  \item\label{il:lake} The current $ω_H$ has bounded potentials.
  \item\label{il:palmer} Using Notation~\ref{not:maiset}, the positive measure
    $(ω_H)^{\dim X}$ obtained as the top intersection of $ω_H$ puts no mass on proper analytic subsets and satisfies
    $$
    (ω_H)^{\dim X} = [H]^{\dim X}·μ_X(X)^{-1}·μ_X.
    $$
  \end{enumerate}
  Furthermore, the current $ω_H$ is smooth on $X_{\reg}$ and induces a genuine
  Ricci-flat Kähler metric there.  \qed
\end{thm}

\Publication{The preprint version of this paper,
  \href{https://arxiv.org/abs/1704.01408}{arXiv:1704.01408}, explains
  Theorem~\ref{thm:EGZ} in more detail.}\Preprint{
  \begin{explanation*}[Item~\ref{il:lake}, bounded potentials]
    Choose a smooth reference Kähler metric $ω$ on $X$ with $[ω] = [H]$ and say
    that ``$ω_H$ has bounded potentials'' if one can write $ω_H = ω+dd^c φ$ for
    some \emph{bounded} function $φ$.  This condition is independent of the
    choice of the Kähler metric $ω ∈ [H]$, as any two cohomologous Kähler
    metrics differ by the $dd^c$ of a smooth function, cf.~\cite[Sect.~4.6.1,
    discussion below Def.~4.6.2]{BEG}.
  \end{explanation*}
  
  \begin{explanation*}[Item~\ref{il:palmer}, top intersection of currents and Ricci curvature]
    Bedford-Taylor theory allows to define the top intersection $(ω_H)^{\dim X}$
    of the current $ω_H$, which is a positive measure on $X$ that puts no mass
    on analytic, or more generally pluripolar sets.  Indeed, if $E$ is
    pluripolar, write $ω_H = dd^c φ$ for some local psh function $φ$, take a
    locally defined psh function $V$ with $\{V=-∞\} = E$ and apply
    \cite[Thm.~2.2(b)]{Dem85} with $k = n$, and with $φ_1 := φ$, …, $φ_n := φ$.
    The equation in Item~\ref{il:palmer} expresses that $\Ric ω_H=0$ in the
    sense of currents.
  \end{explanation*}}

\begin{rem}[KLT versus canonical singularities in Theorem~\ref{thm:EGZ}]
  The statement \cite[Thm.~7.5]{MR2505296} assumes that $X$ has canonical
  singularities.  The proof does not use this assumption and works verbatim for
  klt spaces in case where $Δ=0$, compare with \cite[Thm.~7.12]{MR2505296}.
\end{rem}

\subsection{Universal property of the EGZ construction}\label{univpropegz}
\approvals{Daniel & yes \\ Henri & yes \\ Stefan & yes}

The construction of Theorem~\ref{thm:EGZ} has the following universal property.

\begin{prop}[Universal property of the EGZ construction]\label{prop:univ}
  Let $X$ be projective, klt variety whose canonical class $K_X$ is numerically
  trivial.  Let $H ∈ \Div(X)$ be ample.  Given a quasi-étale cover $γ : Y → X$,
  then $H_Y := γ^*H$ is ample.  Recalling from Reminder~\ref{remi:qec} that $Y$
  is klt and that $K_Y = γ^* K_X$ is numerically trivial, Theorem~\ref{thm:EGZ}
  applies both to $(X,H)$ and to $(Y, H_Y)$, and defines genuine Ricci-flat
  Kähler metrics $ω_H$ and $ω_{H_Y}$ on $X_{\reg}$ and $Y_{\reg}$, respectively.
  These metrics agree on the smooth, open set
  $Y° := γ^{-1} \bigl( X_{\reg} \bigr)$.  In other words,
  $ω_{H_Y}|_{Y°} = (γ|_{Y°})^* ω_H$.
\end{prop}

\Preprint{
  \begin{rem*}[Pull-back of currents]\label{rem:pboc}
    Recall that unlike arbitrary currents, closed positive $(1,1)$-currents can
    be pulled back under holomorphic maps.  More precisely, if $f : Y → X$ is a
    holomorphic map between two normal compact Kähler spaces and if $T$ is a
    closed positive $(1,1)$-current on $X$, then one can write $T = ω + dd^cφ$
    for some smooth form $ω$ and some $ω$-psh function $φ$.  The closed positive
    $(1,1)$-current $f^*T$ on $Y$ is defined to be $f^*T := f^*ω + dd^c (φ◦f)$.
    This pull-back is easily seen to extend the usual pull-back of forms.
  \end{rem*}}

\begin{proof}[Proof of Proposition~\ref{prop:univ}]
  Choose a global generator $σ ∈ H⁰ \bigl( X,\, 𝒪_X(m·K_X) \bigr)$ and a
  constant Hermitian metric $\|·\|$ on the trivial bundle $𝒪_X(m·K_X)$.  We
  obtain a positive measure $μ_X$ on $X$, as in Notation~\ref{not:maiset}.
  Consider the pull-back form $γ^*σ$, which generates $𝒪_Y(m·K_Y)$, as well as
  the constant pull-back metric $\|·\|$ on this trivial bundle.  We obtain the
  associated measure $μ_Y$ on $Y$.  Write $n := \dim X$ and consider the
  positive numbers
  $$
  C := [H]^n·μ_X(X)^{-1} > 0 \quad \text{and} \quad C_Y := [H_Y]^n·μ_Y(Y)^{-1} >
  0.
  $$

  We claim that $γ^*ω_H = ω_{H_Y}$, from which Proposition~\ref{prop:univ} would
  follow.  For this, it suffices check that the pull-back current $γ^*ω_H$
  satisfies Properties~\ref{il:emerson}--\ref{il:palmer} that uniquely
  characterise $ω_{H_Y}$.  Since $γ^* ω_H$ clearly lives in $[H_Y]$ and has
  bounded potentials, it only remains to show that its top intersection
  $(γ^* ω_H)^n$ equals $\wtilde{C}·μ_Y$.  Pulling back the relation
  $ω_H^n= C·μ_X$ by $γ$, we obtain an equality of measures on
  $Y°=γ^{-1}(X_{\reg})$,
  \begin{equation}\label{eq:x1}
    (γ^*ω_H)^n = C·i^n (-1)^{\frac{n(n+1)}{2}}\left( \frac{γ^*σ Λ
        \overline{γ^*σ}}{\|γ^*σ\|²}\right)^{\frac 1m}.
  \end{equation}
  Over $Y°$, where $γ$ is étale, the section $γ^*σ$ trivialises
  $𝒪_Y\bigl( γ^*(m·K_X) \bigr) = 𝒪_Y\bigl(m·K_Y\bigr)$.  We hence conclude from
  Remark~\ref{rem:measure_choiceindependent} that the following measures
  coincide on $Y°$,
  \begin{align}
    \label{eq:x2} μ_Y & = i^n (-1)^{\frac{n(n+1)}{2}}\left( \frac{γ^*σ Λ \overline{γ^*σ}}{\|γ^*σ\|²}\right)^{\frac 1m} && \text{on } Y°.\\
    \intertext{Combining \eqref{eq:x1} and \eqref{eq:x2}, we get that}
    \label{eq:x3} (γ^*ω_H)^n & = C·μ_Y \quad \quad \quad \text{on } Y°.
  \end{align}
  But none of the two measures in \eqref{eq:x3} charges pluripolar sets: for the
  LHS, this is because $γ^*ω_H$ has bounded potentials; for the RHS this is
  almost by definition.  On the other hand, $Y ∖ Y°$ is a proper
  analytic subset of $Y$, hence pluripolar.  We conclude that the two measures
  coincide, and that the equality $(γ^*ω_H)^n = C·μ_Y$ holds globally on $Y$.
  Finally, observe that
  $$
  μ_Y(Y) = μ_Y ( Y° ) = (\deg γ)· μ_X(X_{\reg}) =
  \frac{(\deg γ) · [H]^n}{C} = \frac{[H_Y]^n}{C}.
  $$
  This proves the desired equality $C = [H_Y]^n·μ_Y(Y)^{-1} = C_Y$.
\end{proof}

\begin{rem}
  It follows from the above result that the current $γ^* ω_H$ is smooth and
  Kähler on the open set $Y_{\reg} ⊇ γ^{-1}(X_{\reg})$ although $γ$ may not be
  étale there.  For a typical example, consider a situation where $X$ has
  quotient singularities and admits a global, smooth, quasi-étale cover
  $γ: Y → X$.
\end{rem}

\subsection{Product situations}
\approvals{Daniel & yes \\ Henri & yes \\ Stefan & yes}

Another straightforward though useful property of the EGZ construction is its
compatibility with product structures.  \Preprint{In the smooth case, this
  follows from the fact that the Ricci curvature of a product metric is the
  product of the Ricci forms, and that cohomology classes also behave naturally.
  In the singular case, one has to argue at the level of the Monge-Ampère
  equations satisfied by the Kähler-Einstein metrics.}

\begin{prop}[EGZ construction in product situation]\label{prop:univProd}
  Let $X_1$ and $X_2$ be projective, klt varieties whose canonical classes
  $K_{X_{•}}$ are numerically trivial.  Let $H_{•} ∈ \Div(X_{•})$ be ample.
  Then, $X = X_1 ⨯ X_2$ is klt, with trivial canonical class and
  $H := (\pr_1)^*H_1⨯(\pr_2)^*H_2 ∈ Div(X)$ is ample.  Theorem~\ref{thm:EGZ}
  applies to $(X_{•}, H_{•})$ and to $(X,H)$, and defines genuine Ricci-flat
  Kähler metrics $(ω_{H_{•}})$ and $ω_H$ on $(X_{•})_{\reg}$ and on $X_{\reg}$,
  respectively.  With this notation, the Ricci-flat Kähler manifold
  $(X_{\reg}, ω_H)$ is isomorphic to the product
  $\bigl((X_1)_{\reg}, ω_{H_1}\bigr) ⨯ \bigl((X_2)_{\reg}, ω_{H_2}\bigr)$.
\end{prop}
\begin{proof}
  The current $ω:=\pr_1^*ω_{H_1}+\pr_2^*ω_{H_2}$ on $X$ has bounded potentials
  and satisfies $[ω]=[H]$.  In order to prove the identity of current $ω = ω_H$
  on $X$ from which the proposition follows, it will therefore suffice to check
  Property~\ref{il:palmer} of Theorem~\ref{thm:EGZ}.  To this end, observe that
  if $m ∈ ℕ$ is divisible enough, one has a natural isomorphism
  $$
  H⁰ \bigl(X,\, 𝒪_X(m·K_X)\bigr) ≅ H⁰\bigl(X_1,\, 𝒪_{X_1}(m·K_{X_1})\bigr) ⊗
  H⁰\bigl(X_2,\, 𝒪_{X_2}(m·K_{X_2})\bigr).
  $$
  It follows that the product measure $μ_{X_1}⊗μ_{X_2}$ coincides with $μ_X$ up
  to a constant.  In particular,
  $$
  \frac{μ_X}{μ_X(X)} = \frac{μ_{X_1} ⊗ μ_{X_2}}{μ_{X_1}(X_1)·μ_{X_2}(X_2)}
  $$
  and therefore
  $$
  \frac{[H]^{\dim X}·μ_X}{μ_X(X)} = \frac{[H_1]^{\dim X_1}·[H_2]^{\dim
      X_2}·μ_{X_1} ⊗ μ_{X_2}}{μ_{X_1}(X_1)·μ_{X_2}(X_2)}.
  $$
  The product current $ω$ therefore satisfies Property~\ref{il:palmer}, as
  desired.
\end{proof}

%
%
\svnid{$Id: 04-standardSetting.tex 763 2018-10-27 13:39:14Z kebekus $}

\section{The standard setting}\label{sec:setting}
\subversionInfo
\approvals{Daniel & yes \\ Henri & yes \\ Stefan & yes}

The goal of this section is to set up the framework for the rest of the article.
In addition, we recollect some results about holonomy that we will use
repeatedly later on.  This includes locality of restricted holonomy, relation
with parallel transport and behaviour under quasi-étale covers.

\subsection{The standard setting}
\approvals{Daniel & yes \\ Henri & yes \\ Stefan & yes}

Throughout the present paper we will be working in the setup of
Theorem~\ref{thm:EGZ}, and use the metrics produced there in order to compute
holonomy groups.  We will use the following notation.

\begin{setupnot}[Standard setting]\label{setting:holonomy}
  Let $X$ be a projective, klt variety with numerically trivial canonical class
  $K_X$.  Write $n := \dim_ℂ X$ and assume that $n ≥ 2$.  Fix an ample Cartier
  divisor $H$ on $X$ and a smooth point $x ∈ X_{\reg}$.  Let $ω_H$ denote the
  singular, Ricci-flat Kähler metric constructed in Theorem~\ref{thm:EGZ}.
  Write $g_H$ for the associated Riemannian metric on $X_{\reg}$, and $h_H$ for
  the associated Hermitian metric on $TX_{\reg}$.  We consider the complex,
  Hermitian vector space $V := T_xX$ with Hermitian form $h_{H, x}$ and write
  $$
  G := \Hol\bigl( X_{\reg},\, g_H\bigr)_x ⊆ \U(V, h_{H,x}) \quad\text{and}\quad
  G° := \Hol\bigl( X_{\reg},\, g_H\bigr)°_x ⊆ \SU(V, h_{H,x})
  $$
  for the (restricted) holonomy\Preprint{, as in Reminder~\vref{remi:holonomy}}.
\end{setupnot}

\subsection{Geodesic incompleteness}
\approvals{Daniel & yes \\ Henri & yes \\ Stefan & yes}

The following result implies that many of the standard arguments used in the
study of manifolds with vanishing first Chern class are not at our disposal.

\begin{prop}[Geodesic incompleteness in the standard setting]\label{prop:incomplete}
  In the standard Setting~\ref{setting:holonomy}, assume that $X_{\reg} ≠ X$.
  In other words, assume that $X$ does have non-trivial singularities.  Then,
  $(X_{\reg}, g_H)$ is not geodesically complete.
\end{prop}
\begin{proof}
  The current $T_H$ has bounded potentials.  The volume of $X_{\reg}$ is
  therefore computed as
  $$
  \Vol(X_{\reg}, g_H) = \int_{X_{\reg}} ω_H^n = [H]^n.
  $$
  In particular, the volume is finite.  Using that $(X_{\reg}, g_H)$ has
  non-negative Ricci curvature, a result of Yau then asserts that
  $(X_{\reg}, g_H)$ cannot be geodesically complete unless $X_{\reg}$ is
  compact, \cite[Cor.\ on page 25]{SY}.
\end{proof}

\Preprint{
  \begin{rem*}
    In general, it is not know whether the diameter of $(X_{\reg}, g_H)$ is
    finite or not.  However, if $X$ admits a smoothing or a crepant resolution,
    the diameter $\diam(X_{\reg}, g_H)$ is proven to be finite by
    \cite{MR2538503} and \cite[App.~B]{RZ}.
  \end{rem*}
}

\subsection{Real analytic structure}
\approvals{Daniel & yes \\ Henri & yes \\ Stefan & yes}

Since $(X_{\reg}, g_H)$ is Einstein as a Riemannian manifold, a theorem of
DeTurck-Kazdan implies that $X_{\reg}$ admits a smooth atlas with real analytic
transition functions such that $g_H$ is real analytic in each coordinate chart,
see \cite[Thm.~5.26]{MR867684}.  This has the following consequence.

\begin{prop}[Behaviour of restricted holonomy under restriction]\label{prop:dtkd}
  In the standard Setting~\ref{setting:holonomy}, if $U = U(x) ⊆ X_{\reg}$ is
  any (analytically) open neighbourhood of $x$, then the restricted holonomy
  groups $G°$ and $\Hol\bigl( U,\, g_H|_U\bigr)°_x$ agree.
\end{prop}
\begin{proof}
  Since $(X_{\reg}, g_H)$ is real-analytic, both groups in question equal the
  local holonomy group at $x$, \cite[Thm.~10.8 on p.~101]{MR1393940}.
\end{proof}

\subsection{Quasi-étale covers in the standard setting}\label{holandcovers}
\approvals{Daniel & yes \\ Henri & yes \\ Stefan & yes}

Working in the standard setting, we will frequently try to simplify geometry by
passing to a suitable quasi-étale cover ---this could be an index-one cover,
which simplifies the singularities, or the holonomy cover that will be
introduced in Section~\ref{sec:coverings}.  The present section introduces
notation and recalls a number of basic facts that will later be used.

\subsubsection{Standard notation for quasi-étale covers}
\approvals{Daniel & yes \\ Henri & yes \\ Stefan & yes}

Given a quasi-étale cover $γ : Y → X$, we may use Theorem~\ref{thm:EGZ} to
construct a Ricci-flat Kähler metric $Y_{\reg}$.  The universal property of the
EGZ construction, Proposition~\ref{prop:univ}, will then allow to compare this
metric to the one that we have on $X$.  The following notation will be used
consistently throughout the paper.

\begin{notation}[Standard notation for quasi-étale covers]\label{not:qec}
  In the standard Setting~\ref{setting:holonomy}, let $γ : Y → X$ be a
  quasi-étale cover and $y ∈ γ^{-1}(x)$ be a point lying over $x$.  Writing
  $Y° := γ^{-1}(X_{\reg})$, the following commutative diagram summarises the
  situation,
  $$
  \xymatrix{ %
    \{y\} \ar@{^(->}[r] \ar[d] & Y° \ar@{^(->}[r] \ar[d]^{γ°\text{, étale}} & Y_{\reg} \ar@{^(->}[r] & Y \ar[d]^{γ\text{, quasi-étale}} \\
    \{x\} \ar@{^(->}[r] & X_{\reg} \ar@{^(->}[rr] && X.
  }
  $$
  Recalling from Reminder~\ref{remi:qec} that $Y$ is klt and that $K_Y = γ^*K_X$
  is numerically trivial, Theorem~\ref{thm:EGZ} applies to $(Y, H_Y)$ where
  $H_Y := γ^*H$, and defines a Ricci-flat Kähler metric on $Y_{\reg}$, which we
  write as $ω_{H_Y}$.  The associated Riemannian metric on $Y_{\reg}$ will be
  written as $g_{H_Y}$, the associated Hermitian metric $h_{H_Y}$.  We consider
  the complex, Hermitian vector space $V_Y := T_y Y$ with Hermitian form
  $h_{H_Y,y} $, as well as the following subgroups of $\U(V_Y, h_{H_Y,y})$,
  \begin{align}
    \label{il:Az} I\hphantom{°} := \quad \Hol(Y°,g_{H_Y})_y & \;⊆\; \Hol( Y_{\reg}, g_{H_Y})_y \quad =: G_Y &\text{and} \\
    \label{il:Bz} I° := \quad \Hol(Y°,g_{H_Y})°_y & \;⊆\; \Hol( Y_{\reg}, g_{H_Y})°_y \quad =: G_Y°.
  \end{align}
  The universal property of the EGZ construction, Proposition~\ref{prop:univ},
  asserts that
  $$
  g_{H_Y}\bigl|_{Y°} = (γ°)^*g_H \quad\text{and}\quad ω_{H_Y}\bigl|_{Y°} =
  (γ°)^*ω_H.
  $$
  In particular, we may use the isomorphism $dγ°|_y$ to identify the Hermitian
  vector spaces $V_Y$ and $V$, and to view $I$, $G_Y$, $I°$ and $G_Y°$ as
  subgroups of $\U(V, h_{H,x})$.
\end{notation}

\subsubsection{Behaviour of holonomy under covers}\label{subsubsect:holonomy_under_covers}
\approvals{Daniel & yes \\ Henri & yes \\ Stefan & yes}

We conclude the present section by pointing out a few relations between the
groups introduced in Notation~\ref{not:qec}.  The proof of the following
elementary fact is left to the reader.

\begin{fact}[Behaviour of holonomy under covers, I]\label{fact:bhuc}
  In the setting of Notation~\ref{not:qec}, the following diagram is commutative
  \begin{equation}\label{eq:vbvc}
    \begin{gathered}
      \xymatrix{
        \protect{\begin{Bmatrix}
            \text{loops in $Y°$ starting}\\
            \text{and ending in } y
          \end{Bmatrix}} \ar[rrr]^(.6){\text{parallel transport}} \ar@{^(->}[d]_{\scriptsize γ° ◦\: (•) }^{\text{injective}} &&& \U \bigl( V_Y \bigl) \ar[d]^{(dγ°|_y) \:◦\: (•) \:◦\: (dγ°|_y)^{-1}}_{≅} \\
        \protect{\begin{Bmatrix}
            \text{loops in $X_{\reg}$ starting}\\
            \text{and ending in } x
          \end{Bmatrix}} \ar[rrr]_(.6){\text{parallel
            transport}} &&& \U \bigl( V \bigr).  & \mathclap{\hspace{3.9cm}\square} }
    \end{gathered}
  \end{equation}
\end{fact}

\begin{rem}[Behaviour of holonomy under covers, II]\label{rem:bhuc}
  In the setting of Notation~\ref{not:qec}, identifying $\U \bigl(V_Y \bigl)$
  with $\U \bigl(V \bigr)$ using the isomorphism of Diagram~\eqref{eq:vbvc}, the
  diagram allows to view $I$ as a subgroup of $G$.  More precisely,
  \begin{multline*}
    I = \bigl\{ g ∈ G \,|\, g \text{ is parallel transport along a loop in } X_{\reg}\\
    \text{with homotopy class in } (γ°)_*\: π_1(Y°,y) \bigr\}.
  \end{multline*}
  This description together with Reminder~\ref{rem:fundamentalgroupsurjection}
  presents $I$ as a union of connected components of $G$ and therefore shows
  that the maximal connected subgroups agree, $I° = G°$.
\end{rem}

\begin{lem}[Behaviour of holonomy under quasi-étale coverings, I]\label{lem:bhqec2}
  In the setting of Notation~\ref{not:qec}, the natural inclusions \eqref{il:Az}
  and \eqref{il:Bz} are equalities.  In other words, $I° \:=\: G°_Y$ and
  $I \:=\: G_Y$.
\end{lem}
\begin{proof}
  The equality $I° = G°_Y$ has been shown in Proposition~\ref{prop:dtkd} above.
  To prove that $I = G_Y$, observe that $Y°$ is a big subset of $Y_{\reg}$.  The
  corresponding fundamental groups therefore agree.  We obtain a commutative
  diagram
  $$
  \xymatrix{ %
    π_1 \bigl(Y° \bigr) \ar@{->>}[r] \ar@{=}[d] & I/I° \ar[d]^{α} \\
    π_1 \bigl(Y_{\reg} \bigr) \ar@{->>}[r] & G_Y/G°_Y,
  }
  $$
  which implies that $α$ is surjective.  Using that $I° = G°_Y$, an application
  of the Snake Lemma then yields the desired equality $I = G_Y$.
\end{proof}

\begin{cor}[Behaviour of holonomy under quasi-étale coverings, II]\label{cor:bhqec2}
  In the setting of Notation~\ref{not:qec}, viewing $G_Y$ as a subgroup of
  $\U(V, h_{H,x})$, we have equality $G° = G°_Y$ and an inclusion $G_Y ⊆ G$ that
  is given as follows,
  \begin{multline}\label{eq:byv}
    G_Y = \bigl\{ g ∈ G \,|\, g \text{ is parallel transport along a loop in } X_{\reg}\\
    \text{ with homotopy class in } (γ°)_* \: π_1(Y°,y) \bigr\}.
  \end{multline}
\end{cor}
\begin{proof}
  Consider the groups $I°$ and $I$.  Remark~\ref{rem:bhuc} immediately
  identifies $G°$ with $I°$, and the right hand side of \eqref{eq:byv} with $I$.
  The identifications of Lemma~\ref{lem:bhqec2} thus end the proof.
\end{proof}

\Preprint{
\subsection{Parallel transport of invariant subspaces}\label{subsect:parallelity_and_holomorphy}
\approvals{Daniel & yes \\ Henri & yes \\ Stefan & yes}

The following theorem asserts that parallel subbundles of $TX_{\reg}$ and
$T^*X_{\reg}$ are holomorphic.  We use this result in Section~\ref{sec:candecrd}
to find a canonical direct sum decomposition in the sheaf of reflexive forms,
which relates (strong) stability with irreducibility properties of the
(restricted) holonomy action and explains the behaviour of these properties
under quasi-étale coverings.

\begin{prop}[Parallel transport of invariant subspaces]\label{prop:HPB}
  In the standard Setting~\ref{setting:holonomy}, let $N ⊆ V$ be any
  $G$-invariant linear subspace, and $\widehat{N} ⊆ TX_{\reg}$ be the associated
  parallel, smooth subbundle obtained by parallel transport.  Then,
  $\widehat{N}$ is a holomorphic subbundle of $TX_{\reg}$.  Ditto for subbundles
  of $T^*X_{\reg}$.
\end{prop}
\begin{proof}
  Consider $(TX_{\reg}, h_H)$ as a holomorphic Hermitian vector bundle and apply
  \cite[Chap.~I, Prop.~4.18]{Kob87} to $\widehat{N}$ in order to get the desired
  result.
\end{proof}

The result above can also be expressed as a statement about sheaves.

\begin{cor}[Holonomy principle for bundles as a statement about sheaves]\label{cor:HPB}
  In the standard Setting~\ref{setting:holonomy}, let $N ⊆ V$ be any
  $G$-invariant linear subspace, and $\widehat{N} ⊆ TX_{\reg}$ be the associated
  parallel, smooth subbundle obtained by parallel transport.  Then, there exists
  a locally free subsheaf of $𝒪_{X_{\reg}}$-modules
  $j : \sN° ↪ 𝒯_{X_{\reg}}$ with locally free quotient, and a commutative
  diagram
  $$
  \xymatrix{ %
    \sN° ⊗_{𝒪_{X_{\reg}}} \cA_{X_{\reg}} \ar[r]^{j ⊗ \Id} \ar[d]_{≅} & 𝒯_{X_{\reg}} ⊗_{𝒪_{X_{\reg}}} \cA_{X_{\reg}} \ar[d]^{≅} \\
    \cN \ar[r]^{j} & \,\cT_{X_{\reg}} %
  }
  $$
  where $\cN$ denotes the sheaf of smooth sections of $\widehat{N}$.  Ditto for
  subbundles of $T^*X_{\reg}$.  \qed
\end{cor}
}

\part{Holonomy}\label{part:II}
%
%
\svnid{$Id: 05-restrictedHolonomy.tex 740 2017-11-02 09:38:44Z kebekus $}

\section{The classification of restricted holonomy}\label{sec:restrictedHolonomy}
\subversionInfo
\approvals{Daniel & yes \\ Henri & yes \\ Stefan & yes}

After fixing our notation regarding the restricted holonomy decomposition, we
apply the standard classification of irreducible restricted holonomy groups
appearing in our setup.

\subsection{Notation}
\approvals{Daniel & yes \\ Henri & yes \\ Stefan & yes}

The following construction and notation will be used throughout.

\begin{consnot}[Decomposition induced by restricted holonomy]\label{constr:canonDecompRestr}
  Assume the Standard Setting~\ref{setting:holonomy} and recall from
  \cite[Cor.~10.41]{MR867684} that the action $G° ↺ V$ is totally decomposed.
  That is, there exist decompositions
  \begin{equation}\label{eq:cd1}
    V = V_0 ⊕ V_1 ⊕ ⋯ ⊕ V_m \quad\text{and}\quad G° = G°_1 ⨯ ⋯ ⨯ G°_m,
  \end{equation}
  where each factor $G°_i$ acts non-trivially and irreducibly on $V_i$ and
  trivially on all the other summands.  We refer to \eqref{eq:cd1} as the
  \emph{canonical decomposition of the Hermitian $G°$-space $V$}.  The induced
  decomposition of the dual space $V^* = T^*_xX$ will analogously be called
  \emph{canonical decomposition of the Hermitian $G°$-space $V^*$}.
\end{consnot}


\begin{rem}[Behaviour under quasi-étale covers]\label{rem:cdqec0}
  If $γ : Y → X$ is a quasi-étale cover, and $y ∈ γ^{-1}(x)$ is any point,
  recall from Corollary~\ref{cor:bhqec2} that $V=T_xX$ and $V_Y = T_yY$ are
  canonically identified, and that the restricted holonomy groups agree,
  $G° = {G}_Y°$.  In this sense, the canonical decomposition of
  Construction~\ref{constr:canonDecompRestr} is invariant under passing to
  quasi-étale covers.
\end{rem}

The action $G° ↺ V$ carries algebro-geometric information: We will see in
Corollary~\ref{cor:flatDecomp} that the dimension of $V_0$ equals the augmented
irregularity $\wtilde{q}(X)$.  Furthermore, we will see in
Corollary~\ref{cor:sirr2} that $𝒯_X$ is strongly stable if and only if the
action $G° ↺ V$ is irreducible.

\subsection{Classification}
\approvals{Daniel & yes \\ Henri & yes \\ Stefan & yes}

Using Proposition~\ref{prop:locdecomp}, Proposition~\ref{prop:dtkd}, as well as
Berger's classification \cite[Thm.~10.108]{MR867684}, one easily obtains the
following result\Publication{, a full proof of which is given in the preprint
  version of this paper,
  \href{https://arxiv.org/abs/1704.01408}{arXiv:1704.01408}}.

\begin{prop}[Classification of restricted holonomy]\label{prop:factors}
  Setting and notation as in \ref{constr:canonDecompRestr}.  Given any index
  $1 ≤ i ≤ m$ and writing $n_i := \dim V_i$, one of the following holds true.
  \begin{enumerate}
  \item The group $G°_i$ is isomorphic to $\SU(n_i)$.
  \item The number $n_i$ is even, and the group $G°$ is isomorphic to
    $\Sp\!\left(\frac {n_i}2\right)$.
  \end{enumerate}
  The action $G°_i ↺ V_i$ is isomorphic to the standard action of the respective
  group.\Publication{\qed}
\end{prop}
\Preprint{
  \begin{proof}
    Choose a simply connected open neighbourhood $W = W(x) ⊆ X_{\reg}$.
    Applying Proposition~\ref{prop:locdecomp} (``Local decomposition of Kähler
    manifolds'') to $(W,ω_H|_W)$, we obtain a neighbourhood $U=U(x) ⊆ W$, Kähler
    manifolds $(U_i,ω_i)$ and an identification of $(U,ω_H|_U)$ with the product
    of Kähler manifolds, $(U_i,ω_i) ⨯ ⋯ ⨯ (U_{m'},ω_{m'})$.  Since $(U,ω_H|_U)$
    is Ricci-flat, the same will hold for all the $(U_i,ω_i)$.  We maintain the
    notation of Proposition~\ref{prop:locdecomp} for the remainder of the proof.
    Shrinking the $U_i$ if necessary, we can assume without loss of generality
    that all $U_i$ as well as $U$ are simply connected.  We will write $G°_0$
    for the trivial group.
  
    Proposition~\ref{prop:locdecomp} also asserts that the action of the
    (restricted) holonomy group
    $H := \Hol\bigl( U, g_H|_U)_x = \Hol\bigl( U, g_H|_U)°_x$ on $V=T_xU$ is
    totally decomposed.  The associated decompositions of $H$ and $T_x U$ are
    given as follows,
    $$
    H = \bigtimes_{i=1}^{m'} H_i \quad\text{and}\quad T_x U =
    \bigoplus_{i=0}^{m'} T_{x_i}U_i.
    $$
    
    Recall from Proposition~\ref{prop:dtkd} (``Behaviour of restricted holonomy
    under restriction'') that the natural inclusion $H ⊆ G°$ is in fact an
    equality.  We observed in Remark~\ref{rem:totDecomp} that any two
    decompositions of $G° ↺ V$ agree up to permutation, which shows that
    $m = m'$.  Renumbering the $U_i$ if needed, we may even assume that
    $V_i = T_{x_i}U_i$ and $H_i = G°_i$ for all $i$.  In summary, we see that
    $(U_0,ω_0)$ is flat, whereas all other $(U_i,ω_i)_{i>0}$ are Ricci-flat,
    holonomy-irreducible Kähler manifolds.  The description of the
    $G°_i = \Hol\bigl( U_i, g_i \bigr)°_{x_i}$ is now a standard consequence of
    the classification of non-symmetric, restricted holonomy groups,
    \cite[Thm.~10.108]{MR867684}, and the fact that Ricci-flat locally symmetric
    spaces are flat, \cite[Cor.\ on p.~232]{Simons}.
  \end{proof}}

%
%
\svnid{$Id: 06-holonomy.tex 756 2018-09-30 09:41:37Z kebekus $}

\section{The canonical decomposition of the tangent sheaf}\label{sec:candecrd}
\subversionInfo
\approvals{Daniel & yes \\ Henri & yes \\ Stefan & yes}

In this section, we use standard holonomy considerations to obtain a canonical
decomposition of the tangent sheaf.  This decomposition generalises both
\cite[Thm.~1.3]{GKP16} and \cite[Thm.~A]{Guenancia}, and yields a
differential-geometric characterisation of stability for the tangent bundle,
Corollary~\ref{cor:sirr0}.

\subsection{Construction of the decomposition}
\approvals{Daniel & yes \\ Henri & yes \\ Stefan & yes}

Just like the action of the restricted holonomy, the action of the full holonomy
group $G ↺ V$ is totally decomposed and induces a canonical decomposition of
$V$.

\begin{consnot}[Decomposition induced by holonomy]\label{cons:canonDecomp}
  Assume the Standard Setting~\ref{setting:holonomy} and recall from
  \cite[Cor.~10.38]{MR867684} that the action $G ↺ V$ is totally decomposed.
  That is, there exist decompositions
  \begin{equation}\label{eq:cd2}
    V = W_0 ⊕ W_1 ⊕ ⋯ ⊕ W_k \quad\text{and}\quad G = G_1 ⨯ ⋯ ⨯ G_k
  \end{equation}
  where each factor $G_i$ acts non-trivially and irreducibly on $W_i$ and
  trivially on all the other summands.  As before, we abuse notation and refer
  to \eqref{eq:cd2} as the \emph{canonical decomposition of the Hermitian
    $G$-space $V$}.  We denote the associated smooth, parallel bundles by
  $\widehat{W_j} ⊆ TX_{\reg}$.
\end{consnot}

We have seen in Remark~\ref{rem:totDecomp} that decomposition \eqref{eq:cd2} is
unique up to permutation of the $W_j$ and $G_j$ of positive index.  We now fix
one choice of indices.

\begin{obsnot}[Comparison of the canonical decompositions]\label{obs:ccd}
  Assume the Standard Setting~\ref{setting:holonomy} and recall that $G°$ is a
  normal subgroup of $G$.  This implies that elements of $G$ map $G°$-orbits to
  $G°$-orbits.  The description of the $V_i$ found in
  Construction~\ref{constr:canonDecompRestr} therefore implies that any element
  of $G$ stabilises $V_0$, and permutes the remaining $(V_i)_{i>0}$.
  Renumbering the $V_i$ and $W_j$, if necessary, we may therefore assume without
  loss of generality that there exists an index $ℓ$ and a strictly increasing
  sequences of indices $0 = ℓ_0 < ℓ_1 < ℓ_2 < ⋯$ such that the following holds,
  \begin{enumerate}
  \item The space $V_0$ decomposes as $V_0 = W_0 ⊕ ⋯ ⊕ W_ℓ$.
  \item For every positive $j$, we have $W_{ℓ+j} = V_{ℓ_{j-1}+1} ⊕ ⋯ ⊕ V_{ℓ_j}$.
    The summands are isomorphic as $G°$-representation spaces.
  \end{enumerate}
  \Preprint{Figure~\vref{fig:ssfs} illustrates the setting.
    \begin{figure}
      \centering
      $$
      \begin{matrix}
        G° ↺ V = & \cellcolor{gray!20} \vphantom{V_{0_0}}V_0 & ⊕ & \cellcolor{gray!20} V_1 ⊕ ⋯ ⊕ V_{ℓ_1} & ⊕ & \cellcolor{gray!20} V_{ℓ_1+1} ⊕ ⋯ ⊕ V_{ℓ_2} & ⊕ ⋯\\
        \hphantom{°}G ↺ V = & \cellcolor{gray!20} W_0 ⊕ ⋯ ⊕ W_ℓ & ⊕ & \cellcolor{gray!20} W_{ℓ+1} & ⊕ & \cellcolor{gray!20} W_{ℓ+2} & ⊕ ⋯\\
      \end{matrix}
      $$
      \caption{Decompositions induced by restricted and full holonomy}
      \label{fig:ssfs}
    \end{figure}}
\end{obsnot}

\begin{rem}[Behaviour under quasi-étale covers]\label{rem:cdqec1}
  Continuing Remark~\ref{rem:cdqec0}, if $γ : Y → X$ is a quasi-étale cover, and
  $y ∈ γ^{-1}(x)$ is any point, recall from Corollary~\ref{cor:bhqec2} that
  $V = T_xX$ and $V_Y = T_{y}Y$ are canonically identified, and that the
  holonomy group $G_Y$ is a subgroup of $G$ under this identification.  In this
  sense, the canonical decomposition on $Y$ refines the canonical decomposition
  on $X$.
\end{rem}

\subsection{The decomposition the tangent sheaf}\label{ssec:dofd}
\approvals{Daniel & yes \\ Henri & yes \\ Stefan & yes}

The holonomy principle for bundles immediately yields a canonical direct sum
decomposition of the tangent.

\begin{consnot}[Canonical decomposition of $𝒯_X$]\label{cons:decpW}
  Assume the setup of Construction and Notation~\ref{cons:canonDecomp}.  If
  $\cW_j$ denotes the sheaf of smooth sections of $\what{W_j} ⊆ TX_{\reg}$,
  \Preprint{the holonomy principle for bundles, Corollary~\ref{cor:HPB},
    yields}\Publication{we have} a decomposition
  $𝒯_{X_{\reg}} = 𝒲°_0 ⊕ ⋯ ⊕ 𝒲°_k$ in the category of coherent analytic sheaves
  such that
  $$
  \cT_{X_{\reg}} ≅ \cW_0 ⊕ ⋯ ⊕ \cW_k ≅ \bigl( 𝒲°_0 ⊗_{𝒪_X}
  \cA_{X_{\reg}} \bigr) ⊕ ⋯ ⊕ \bigl( 𝒲°_k ⊗_{𝒪_X} \cA_{X_{\reg}} \bigr)
  $$
  Writing $𝒲_i := ι_* 𝒲°_i$, where $ι : X_{\reg} → X$ is the inclusion, we
  obtain a splitting
  \begin{equation}\label{decomp}
    𝒯_X = 𝒲_0 ⊕ ⋯ ⊕ 𝒲_k.
  \end{equation}
  As a direct summand of a coherent analytic sheaf, each one of the analytic
  sheaves $𝒲_i$ is coherent and hence algebraic by GAGA.  Moreover, each $𝒲_i$
  is reflexive by construction.  Write $\sV°_0 := 𝒲°_0 ⊕ ⋯ ⊕ 𝒲°_ℓ$ and
  $\sV_0 := ι_* \sV°_0$.  We refer to \eqref{decomp} as the \emph{canonical
    decomposition of $𝒯_X$}.
\end{consnot}

\begin{rem}[Behaviour under quasi-étale covers]\label{rem:cdqec2}
  If $γ : Y → X$ is a quasi-étale cover, there are now two decompositions on
  $Y$: the canonical decomposition on $Y$, and the reflexive pull-back of the
  canonical decomposition on $X$,
  \begin{align}
    \label{eq:df1} 𝒯_Y & = 𝒲_{Y,0} ⊕ ⋯ ⊕ 𝒲_{Y,ℓ_Y} ⊕ 𝒲_{Y,ℓ_Y+1} ⊕ ⋯ ⊕ 𝒲_{Y,{k}_Y} \\
    \label{eq:df2} & = γ^{[*]}𝒲_0 ⊕ ⋯ ⊕ γ^{[*]}𝒲_ℓ ⊕ γ^{[*]}𝒲_{ℓ+1} ⊕ ⋯ ⊕ γ^{[*]}𝒲_k.
  \end{align}
  Remark~\ref{rem:cdqec1} implies that \eqref{eq:df1} refines \eqref{eq:df2}.
  Remark~\ref{rem:cdqec0} implies that
  $$
  γ^{[*]} \sV_0 = γ^{[*]} 𝒲_0 ⊕ ⋯ ⊕ γ^{[*]} 𝒲_ℓ \quad\text{and}\quad {\sV}_{Y,0}
  = 𝒲_{Y,0} ⊕ ⋯ ⊕ 𝒲_{Y,ℓ_Y}
  $$
  agree.
\end{rem}

We conclude the present subsection with a first description of the summands that
appear in the decomposition.  Once the Bochner principle for reflexive tensors
is established in Theorem~\ref{thm:holonomy}, we will be able say more, see
Corollary~\ref{cor:holbun2w}.

\begin{prop}[Summands in the canonical decomposition of $𝒯_X$]\label{prop:holbun2x}
  In the setting of Construction~\ref{cons:decpW}, the following holds.
  \begin{enumerate}
  \item\label{il:Cx} The Hermitian holomorphic vector bundles $𝒲°_0$, …, $𝒲°_ℓ$
    are unitary flat.  In particular, the locally free sheaves $𝒲°_0$, …, $𝒲°_ℓ$
    and $\sV°_0$ are holomorphically flat.
  \item\label{il:Ax} If $\widehat{F} ⊆ TX_{\reg}$ is any parallel subbundle and
    $i > 0$ is any index, then the projection map $p_i : F → W_i$ at $x$ is
    either zero or surjective.
  \item\label{il:Bx} The locally free sheaf $𝒲°_0$ is holomorphically trivial
    and if $i>0$ is any index, then $𝒲_i$ is stable of slope zero with respect
    to any ample polarisation on $X$.  In particular, $𝒯_X$ is polystable with
    respect to any ample polarisation on $X$.
  \end{enumerate}
\end{prop}

\begin{rem}[Non-flatness of remaining summands]
  Improving on Item~\ref{il:Cx}, we will later see in
  Corollary~\ref{cor:flatDecomp} that none of the remaining summands $𝒲°_{ℓ+1}$,
  …, $𝒲°_k$ is holomorphically flat.
\end{rem}

\begin{proof}[Proof of Proposition~\ref{prop:holbun2x}]
  ---

  \subsubsection*{Proof of Item~\ref{il:Cx}}
  
  Choose an integer $0 ≤ j ≤ ℓ$.  As the holomorphic vector bundle $𝒲°_j$ is
  obtained by parallel transport, it is a direct summand of $𝒯_{X_{\reg}}$ as a
  \emph{Hermitian} holomorphic subbundle.  Furthermore, $𝒲°_j$ is acted
  trivially upon by the restricted holonomy group $G°$, and hence its restricted
  holonomy group is trivial.  Shrinking to small simply connected
  neighbourhoods, one can therefore find parallel local frames for $𝒲°_j$.  It
  follows that the Hermitian structure of $𝒲°_j$ is flat, and the $(1,0)$-part
  of the Chern connection is a flat holomorphic connection.

  \subsubsection*{Proof of Item~\ref{il:Ax}}

  Parallel transport stabilises both $\widehat{F}$ and $\widehat{W}_i$, and
  commutes with the projection map $\pr_i$.  As a consequence, it follows that
  the rank of $\pr_i$ is constant.  To prove surjectivity, it will therefore
  suffice to show surjectivity at the point $x$.  In other words, writing
  $F := \widehat{F}|_{\{x\}}$, we need to show surjectivity of the composition
  \begin{equation}\label{eq:xxgx}
    \xymatrix{ %
      F \ar[rr]_(.3){\text{inclusion}} && V = W_0 ⊕ ⋯ ⊕ W_k
      \ar[rr]_(.7){\text{projection}} && W_i.  %
    }
  \end{equation}
  The morphisms in \eqref{eq:xxgx} are linear maps of $G$-representation spaces.
  Since the representation space $W_i$ is irreducible by assumption, any
  equivariant map with image in $W_i$ must either be zero, or surjective.
  Item~\ref{il:Ax} follows.

  \subsubsection*{Proof of Item~\ref{il:Bx}}

  The triviality of $𝒲°_0$ is clear, as one can parallel transport any basis of
  $W_0$ to obtain a trivialising set of holomorphic sections of
  $𝒲°_0$\Preprint{, cf.~Reminder~\ref{remi:cc}}.  Now, recall from
  \cite[Thm~A.(ii)]{Guenancia} that there exists a decomposition $𝒯_X = ⊕ ℱ_j$
  with the following properties.
  \begin{enumerate}
  \item The $ℱ_j$ are reflexive, and stable of slope zero with respect to any
    ample polarisation of $X$.
  \item The restrictions $ℱ°_j := ℱ_j\bigl|_{X_{\reg}}$ are locally free.  The
    associated subbundles of $F°_j ⊆ TX_{\reg}$ are parallel\footnote{This is
      shown in the proof of \cite[Thm.~A(ii)]{Guenancia}, on page~35, a few
      lines ahead of §5.}.
  \end{enumerate}
  In particular, it follows that $𝒯_X$ is semistable of slope zero with respect
  to any polarisation, and hence so are the direct summands $𝒲_{•}$ in the
  canonical decomposition.  Now, given any index $i > 0$, we will show that
  $𝒲_i$ is isomorphic to one of the $ℱ_j$, hence stable with respect to any
  polarisation.  We start by choosing an index $j$ such that the projection map
  $p_{ji} := \pr_i|_{ℱ_j} : ℱ_j → 𝒲_i$ is not zero.

  As a non-trivial map from a stable sheaf to a semistable sheaf of the same
  slope, $p_{ji}$ is clearly injective.  We claim that $p_{ji}$ is also
  surjective.  Since both $ℱ_j$ and $𝒲_i$ are reflexive, it suffices to show
  surjectivity of the restricted map $p°_{ji} : ℱ°_j → 𝒲°_i$.  That, however,
  has been established in Item~\ref{il:Ax} above.
\end{proof}

\subsection{Canonical decomposition vs.\ earlier results}
\approvals{Daniel & yes \\ Henri & yes \\ Stefan & yes}

Next, we explain how the canonical decomposition of $𝒯_X$ relates to earlier
work, and how its uniqueness improves known results.

\subsubsection{Uniqueness}
\approvals{Daniel & yes \\ Henri & yes \\ Stefan & yes}

The decomposition of a polystable sheaf into stable summands is unique up to
non-canonical isomorphism, but not unique in general.  However, as soon as a
singular Ricci-flat Kähler metric is fixed, the canonical decomposition of $𝒯_X$
is \textit{unique} up to permutation of the factors $𝒲_1, …, 𝒲_k$, as follows
from Remark~\ref{rem:totDecomp}.  The factor $𝒲_0$, although unique, is
\textit{not} stable as soon as its rank is larger than one, and it does not
admit a unique decomposition into stable subsheaves.

\subsubsection{Comparison with earlier results}\label{ssect:comp}
\approvals{Daniel & yes \\ Henri & yes \\ Stefan & yes}

It follows from the discussion above that the polystability decomposition of
\cite[Thm.~A]{Guenancia} in case $K_X \equiv 0$ is isomorphic to the canonical
decomposition of $𝒯_X$, unless there exists a trivial summand of rank at least
two.

Moreover, the proof of Item~\ref{il:Bx} above implies that on the quasi-étale
cover $γ:~Y → X$ whose existence is established in \cite[Thm.~1.3]{GKP16} the
summands of the decomposition of the tangent sheaf $𝒯_Y$ produced by
\emph{loc.~cit.} are isomorphic to the ones in the canonical
decomposition~\eqref{decomp} of $𝒯_Y$.  In particular, the summands in the
canonical decomposition of $𝒯_Y$ are strongly stable in the sense of
\cite[Def.~7.2]{GKP16} and have trivial determinant.  The latter property
furthermore implies integrability by \cite[Thm.~7.11]{GKP16}.

\subsection{Stability and irreducibility of the holonomy representations}
\approvals{Daniel & yes \\ Henri & yes \\ Stefan & yes}

The canonical decomposition relates stability of $𝒯_X$ to irreducibility of the
holonomy representation.  The following corollaries are immediate consequences
of Proposition~\ref{prop:holbun2x}.  Later, Corollary~\ref{cor:sirr2} will also
relate strong stability and irreducibility of $G° ↺ V$.

\begin{cor}[Stability and irreducibility, I]\label{cor:sirr0}
  In the standard Setting~\ref{setting:holonomy}, the following statements are
  equivalent.
  \begin{enumerate}
  \item The sheaf $𝒯_X$ is stable with respect to any ample polarisation.
  \item The holonomy representation $G ↺ V$ is irreducible.  \qed
  \end{enumerate}
\end{cor}

Proposition~\ref{prop:holbun2x} also applies to describe the holonomy
representation in case where $𝒯_X$ is stable only with respect to a movable
curve class, or with respect to nef divisors.  Stability with respect to a
movable class is discussed in the paper \cite{GKP15}.  We refer to
\cite[Def.~2.11]{GKP15} for a precise definition.

\begin{cor}[Stability and irreducibility, II]\label{cor:sirr1}
  In the standard Setting~\ref{setting:holonomy}, assume that one of the
  following holds.
  \begin{enumerate}
  \item\label{il:kasper} There exist nef Cartier divisors $H_1$, …, $H_{n-1}$ on
    $X$ such that $𝒯_X$ is stable with respect to $(H_1, …, H_{n-1})$.
  \item\label{il:seppel} The variety $X$ is $ℚ$-factorial, and there exists a
    movable curve class $α ∈ N_1(X)_{ℝ}$ such that $𝒯_X$ is $α$-stable.
  \end{enumerate}
  Then, the holonomy representation $G ↺ V$ is irreducible.
\end{cor}
\begin{proof}
  We discuss case \ref{il:seppel} only; the other case is completely similar.
  If the holonomy representation is reducible, Construction~\ref{cons:decpW}
  yields a proper decomposition $𝒯_X = 𝒲_0 ⊕ ⋯ ⊕ 𝒲_k$, with $k>0$ or $k=0$ and
  $\dim 𝒲_0>1$.  In the second case, $𝒯_X$ is trivial, contradiction.  In the
  first case, stability of $𝒯_X$ with respect to $α$ implies the slope
  inequality $μ_α (𝒲_i) < μ_α \bigl(𝒯_X\bigr) = 0$ for all $i$, which
  contradicts the identity
  $$
  0 = μ_α \bigl(𝒯_X \bigr) = \sum_{i=1}^k \frac{\rank(𝒲_i)}{n}·μ_α(𝒲_i).
  \eqno\qedhere
  $$
\end{proof}

%
%
\svnid{$Id: 07-coveringConstructions.tex 763 2018-10-27 13:39:14Z kebekus $}

\section{Covering constructions}\label{sec:coverings}
\subversionInfo

\subsection{Main result}
\approvals{Daniel & yes \\ Henri & yes \\ Stefan & yes}

The quotient $G/G°$ frequently appears as an obstruction to extending locally
defined parallel tensors to global objects.  The difference between holonomy and
restricted holonomy clearly goes away once we pass to the universal covering
space of $X_{\reg}$, but this comes at the price of potentially losing all
algebraic structures, as we have no a priori information on the fundamental
group of $X_{\reg}$.  The following result, which deals with this issue, is
crucial for all our subsequent arguments and is therefore one of the main
results of this paper.

\begin{thmnot}[Holonomy cover]\label{thm:holonomyCover}
  In the standard Setting~\ref{setting:holonomy}, there exists a quasi-étale
  cover $γ : Y → X$ and a point $y ∈ γ^{-1}(x)$, such that holonomy and
  restricted holonomy agree, $G_Y° = G_Y$.  Further, there exist normal,
  projective varieties $A$ and $Z$ and an isomorphism $Y ≅ A⨯Z$, such that
  the following additional properties hold.
  \begin{enumerate}
  \item\label{il:FSCC0x} The variety $A$ is Abelian, of dimension
    $\dim A = \wtilde{q}(X)$.
  \item\label{il:FSCC1x} The variety $Z$ has canonical singularities, trivial
    canonical bundle, and augmented irregularity $\wtilde{q}(Z) = 0$.
  \item\label{il:FSCC2x} The summands ${V}_{Y,0}$ and ${W}_{Y,0}$ of $V_Y$ both
    coincide with $\pr_1^* (TA)|_y$.
  \item\label{il:FSCC3x} The summand ${W}_{Y,1} ⊕ ⋯ ⊕ {W}_{Y, k}$ of $V_Y$
    coincides with $\pr_2^* (TZ)|_y$.
  \end{enumerate}
  \Preprint{Figure~\vref{fig:holonomyCover} illustrates the canonical
    decompositions of $V_Y$.  }Quasi-étale covers with these properties will be
  called \emph{holonomy covers}.\Preprint{
    \begin{figure}
      \centering
      $$
      \begin{matrix}
        G_Y° ↺ V_Y = & \cellcolor{gray!20} V_{Y,0} & ⊕ & \cellcolor{gray!20} V_{Y,1} & ⊕ & \cellcolor{gray!20} V_{Y,2} & ⊕ ⋯\\
        G_Y ↺ V_Y = & \cellcolor{gray!20} W_{Y,0} & ⊕ & \cellcolor{gray!20} W_{Y,1} & ⊕ & \cellcolor{gray!20} W_{Y,2} & ⊕ ⋯\\[-2mm]
        & \underbrace{\hspace{.8cm}}_{\mathclap{= \pr_1^*(TA)_y}} & & \multicolumn{4}{c}{\underbrace{\hspace{3.9cm}}_{= \pr_2^*(TZ)_y}} \\
      \end{matrix}
      $$
      \caption{Canonical decompositions on a holonomy cover}
      \label{fig:holonomyCover}
    \end{figure}}
\end{thmnot}

\Preprint{Theorem~\ref{thm:holonomyCover} has a number of consequences for the
  holonomy, augmented irregularity and canonical decomposition of the initial
  variety.  The following might be of independent interest.}

\begin{cor}[Consequences for the standard setting]\label{cor:flatDecomp}
  In the standard Setting~\ref{setting:holonomy}, the following will hold.
  \begin{enumerate}
  \item\label{il:FDC1} The quotient group $G/G°$ is finite.  In particular, the
    factors $G_1$, …, $G_{ℓ}$ of the holonomy group $G$ are all finite.
  \item\label{il:FDC2} The augmented irregularity $\wtilde{q}(X)$ equals
    $\dim V_0$.
  \item\label{il:FDC3} None of the summands $𝒲°_{ℓ+1}$, …, $𝒲°_k$ is
    holomorphically flat.
  \end{enumerate}
\end{cor}

Theorem~\ref{thm:holonomyCover} will be shown in
Section~\ref{ssec:pf:holonomyCover}.  Corollary~\ref{cor:flatDecomp} is proven
immediately afterwards in Section~\ref{ssect:pocfd}.  To prepare for the proof
of Theorem~\ref{thm:holonomyCover}, Sections~\ref{sec:weakHolonomyCover} and
\ref{ssec:torusCover} construct covers that realise partial aspects of the
holonomy cover.  Combining all results obtained thus far,
Section~\ref{subsect:proof_of_intro_prop} proves
Proposition~\ref{prop:new_intro_prop}.

\subsection{The weak holonomy cover}
\label{sec:weakHolonomyCover}
\approvals{Daniel & yes \\ Henri & yes \\ Stefan & yes}

The canonical decomposition of $V$ induced by the holonomy group $G$ does not in
general refine the decomposition induced by restricted holonomy $G°$\Preprint{,
  see Figure~\vref{fig:ssfs} for an illustration}.  The following result shows
that this problem vanishes once we pass to a suitable quasi-étale cover.

\begin{propnot}[Weak holonomy cover]\label{prop:weakHolonomyCover}
  In the standard Setting~\ref{setting:holonomy}, there exists a quasi-étale
  cover $γ : Y → X$ and a point $y ∈ γ^{-1}(x)$, such that the following holds.
  \begin{enumerate}
  \item\label{il:nelson1} Using Notation~\ref{not:qec}, the holonomy action
    $G_Y ↺ V_Y$ stabilises the canonical decomposition of $V_Y$ that is induced
    by the restricted holonomy action $G_Y° ↺ V_Y$.  In particular, the
    canonical decomposition induced by $G_Y$ is a refinement of the canonical
    decomposition induced by $G_Y°$.

  \item\label{il:napoleon1} A quasi-étale cover of pointed spaces satisfies
    \ref{il:nelson1} if and only if it admits a factorisation via $γ$.  In this
    case, the factorisation is unique.  In particular, the covering $γ$ is
    unique up to canonical isomorphism.
  \end{enumerate}
  \Preprint{Figure~\vref{fig:ssfs2} illustrates the canonical decompositions of
    $V_Y$.  }We refer to the covering $γ$ as the \emph{weak holonomy
    cover}.\Preprint{
    \begin{figure}[h]
      \centering
      $$
      \begin{matrix}
        G_Y° ↺ V_Y = & \cellcolor{gray!20} V_{Y,0} & ⊕ & \cellcolor{gray!20} V_{Y,1} & ⊕ & \cellcolor{gray!20} V_{Y,2} & ⊕ ⋯\\
        G_Y ↺ V_Y = & \cellcolor{gray!20} W_{Y,0} ⊕ ⋯ ⊕ W_{Y,ℓ_Y} & ⊕ & \cellcolor{gray!20} W_{Y,ℓ_Y+1} & ⊕ & \cellcolor{gray!20} W_{Y,ℓ_Y+2} & ⊕ ⋯
      \end{matrix}
      $$
      \caption{Canonical decompositions on the weak holonomy cover}
      \label{fig:ssfs2}
    \end{figure}}
\end{propnot}
\begin{proof}
  Write $V = V_0 ⊕ ⋯ ⊕ V_m$ for the canonical decomposition of $V$ induced by
  the action of $G°$.  We have seen in Observation~\ref{obs:ccd} that every
  element $g ∈ G$ stabilises $V⁰$ and permutes the remaining summands.  From
  this and from Reminder~\ref{rem:fundamentalgroupsurjection} we obtain a
  morphism
  \begin{equation}\label{eq:xmbx}
    \xymatrix{ %
      π_1(X_{\reg},x) \ar@{->>}[rr] \ar@/^5mm/[rrrr]^{σ} && G/G° \ar[rr] && \operatorname{Permutations}\{1, …, m\}.
    }
  \end{equation}
  Recalling from Remark~\ref{rem:qevecxr} that there exists an equivalence
  between quasi-étale covers and finite sets with transitive action of
  $π_1(X_{\reg},x)$, the morphism \eqref{eq:xmbx} thus gives the desired cover.
\end{proof}

For the next corollary, recall that a reflexive sheaf is ``strongly stable'' if
its reflexive pull-back to any quasi-étale cover is stable with respect to any
polarisation $(H_1, \dots, H_{n-1})$ there.  We refer to \cite{GKP16} for a more
detailed discussion, in particular concerning the role of varieties with
strongly stable tangent sheaf in the structure theory of varieties with
numerically trivial canonical divisor.

\begin{cor}[Strong stability and irreducibility of restricted holonomy]\label{cor:sirr2}
  In the standard Setting~\ref{setting:holonomy}, the following assertions are
  equivalent.
  \begin{enumerate}
  \item\label{il:tick} The sheaf $𝒯_X$ is strongly stable.
  \item\label{il:trick} The restricted holonomy representation $G° ↺ V$ is
    irreducible.
  \end{enumerate}
\end{cor}

\begin{proof}[Proof of Corollary~\ref*{cor:sirr2}]
  The implications are proven separately.

  \subsubsection*{Implication \ref{il:tick} $⇒$ \ref{il:trick}}

  Consider a weak holonomy cover $γ: Y → X$, as constructed in
  Proposition~\ref{prop:weakHolonomyCover}.  Assuming \ref{il:tick}, the sheaf
  $𝒯_Y$ will then be stable with respect to any ample polarisation.  By
  Corollary~\ref{cor:sirr0}, this implies that the holonomy action $G_Y ↺ V_Y$
  is irreducible.  By choice of $Y$, either the restricted holonomy $G_Y°$ is irreducible or it is trivial. Identifying $V$ and $V_Y$ as explained in Notation~\ref{not:qec},
  we have seen in Corollary~\ref{cor:bhqec2} that the groups $G°$ and $G_Y°$ are
  naturally identified, and that their natural representations $G° ↺ V$ and
  $G_Y° ↺ V_Y$ are equivalent.

  If the restricted holonomy $G_Y°$ is irreducible, then so is $G°$ and we are done. If $G_Y°$ is trivial, then so is $G°$ and $\sT_{X_{\rm reg}}$ is flat. By \cite[Cor.~1.16]{GKP13}, there exists a finite, quasi-étale cover $A\to X$ where $A$ is an abelian variety. As $\dim X\ge 2$, $\sT_X$ is not strongly stable, a contradiction.  The implication \ref{il:tick} $⇒$ \ref{il:trick} follows.

  \subsubsection*{Implication $\neg$\ref{il:tick} $⇒$ $\neg$\ref{il:trick}}
  
  We start the proof by constructing a sequence of quasi-étale coverings as
  follows,
  $$
  \xymatrix{%
    Y \ar[rrr]_{α\text{, index-one cover}} \ar@/^.4cm/[rrrrrr]^{γ} &&& X' \ar[rrr]_{β\text{, given by $\neg$\ref{il:tick}}} &&& X:
  }
  $$
  As we assume that $𝒯_X$ is \emph{not} strongly stable, we find a quasi-étale
  cover $β: X' → X$ and ample divisors $H_1, …, H_{n-1}$ on $X'$ such that
  $𝒯_{X'}$ is not stable with respect to $(H_1, …, H_{n-1})$.  The space $X'$ is
  again klt, and its canonical class $K_{X'}$ is again numerically trivial.  Let
  $α : Y → X'$ be a global index-one cover, whose existence is guaranteed by
  Proposition~\ref{prop:global_index_one_cover}.  Moreover, setting
  $H_{Y,i} := α^* H_i$, the sheaf $𝒯_{Y} = α^{[*]} 𝒯_{X'}$ is \emph{not} stable
  with respect to $(H_{Y,1}, …, H_{Y,n-1})$.  Together with
  \cite[Prop.~5.7]{GKP16} these properties imply that $𝒯_{Y}$ is in fact
  \emph{not} stable with respect to any tuple of ample bundles.
  Corollary~\ref{cor:sirr0} therefore implies that the holonomy action
  $G_Y ↺ V_Y$ is \emph{not} irreducible, and then neither is the action of the
  restricted holonomy group, $G_Y° ↺ V_Y$.  Using Corollary~\ref{cor:bhqec2} to
  identify the representations $G° ↺ V$ and $G_Y° ↺ V_Y$, the implication
  $\neg$\ref{il:tick} $⇒$ $\neg$\ref{il:trick} thus follows.
\end{proof}

\subsection{Torus covers}\label{ssec:torusCover}
\approvals{Daniel & yes \\ Henri & yes \\ Stefan & yes}

We have seen in Proposition~\ref{prop:holbun2x} that the first summands of the
canonical decomposition, $𝒲_0$, …, $𝒲_{ℓ}$ are locally free on $X_{\reg}$, and
are holomorphically flat there.  The following theorem gives a geometric
explanation for this observation.

\begin{propnot}[Torus cover]\label{prop:torusCover}
  In the standard Setting~\ref{setting:holonomy}, there exist normal, projective
  varieties $A$ and $Z$, and a quasi-étale cover $γ : Y → X$ such that $Y$ has a
  product structure, $Y= A⨯Z$, with the following properties.
  \begin{enumerate}
  \item\label{il:FSCC0} The variety $A$ is Abelian, of dimension
    $\dim A = \wtilde{q}(X)$.
  \item\label{il:FSCC1} The variety $Z$ has canonical singularities, trivial
    canonical bundle, and augmented irregularity $\wtilde{q}(Z) = 0$.
  \end{enumerate}
  A quasi-étale cover $Y → X$ where $Y = A⨯Z$ has a product structure with
  properties \ref{il:FSCC0} and \ref{il:FSCC1} will be called a \emph{torus
    cover}.
\end{propnot}
\begin{proof}
  We construct a sequence of projective varieties and quasi-étale covers as
  follows,
  $$
  \xymatrix{ %
    A⨯Z \ar[rrr]^{c}_{\text{splitting off torus}} &&& X'' \ar[rr]^{b}_{\text{index-one}} && X' \ar[rr]^{a}_{\text{realising }\wtilde{q}} && X.
  }
  $$
  As quasi-étale covers of $X$, all varieties will again be klt, with
  numerically trivial canonical class.  To be more precise, let $a : X' → X$ be
  a quasi-étale cover that realises the augmented irregularity,
  $q(X') = \wtilde{q}(X)$.  Such a cover exists by definition, because
  $\wtilde{q}(X)$ is finite by Item~\ref{il:AUG1} of Lemma~\ref{lem:aug}.  Next,
  we consider a global index-one cover $b : X'' → X'$, as given by
  Proposition~\ref{prop:global_index_one_cover}.  Finally, recall from
  \cite[Thm.~1.3]{GKP16} that there exists an Abelian variety $A$, a canonical
  variety $Z$ with $\wtilde{q}(Z) = 0$ and a quasi-étale cover $c: A⨯Z → X''$.
  Recalling from \cite[Lem.~4.1.14]{Laz04-I} (``Injectivity Lemma'') that the
  irregularity increases in covers, the equality of augmented and actual
  irregularity still holds on $A⨯Z$.  In other words,
  $\wtilde q(X) = q(A⨯Z) = \dim A$.  We consider the composition $γ : A⨯Z → X$.
  The construction clearly satisfies Items~\ref{il:FSCC0} and \ref{il:FSCC1}.
\end{proof}

The canonical decompositions on a torus cover are described as follows.

\begin{prop}[Canonical decompositions on torus cover]\label{prop:torusCover:decomposition}
  In the standard Setting~\ref{setting:holonomy}, let $γ: Y → X$ be a torus
  cover, $Y = A ⨯ Z$.  Then, the Kähler manifold $(Y_{\reg}, ω_{H_Y})$ splits as
  $(A, ω_{H_A}) ⨯ (Z_{\reg}, ω_{H_Z})$, where $ω_{H_A}$ is flat and $ω_{H_Z}$ is
  Ricci-flat.  If $y = (a,z) ∈ γ^{-1}(x)$, the summands in the canonical
  decompositions of $V_Y$ relate to the product structure of $Y$ as follows.
  \begin{enumerate}
  \item\label{il:FSCC2} The summands ${V}_{Y,0}$ and ${W}_{Y,0}$ both equal
    $\pr_1^* (TA)|_{y}$.
  \item\label{il:FSCC3} The summand ${W}_{Y,1} ⊕ ⋯ ⊕ {W}_{Y, k}$ equals
    $\pr_2^* (TZ)|_{y}$.
  \end{enumerate}
  \Preprint{Figure~\vref{fig:torusCover} illustrates the canonical
    decompositions.
    \begin{figure}
      \centering
      $$
      \begin{matrix}
        G_Y° ↺ V_Y = & \cellcolor{gray!20} V_{Y,0} & ⊕ & \cellcolor{gray!20} V_{Y,1} ⊕ ⋯ ⊕ V_{Y,ℓ_1} & ⊕ & \cellcolor{gray!20} V_{Y,ℓ_1+1} ⊕ ⋯ ⊕ V_{Y,ℓ_2} & ⊕ ⋯\\
        G_Y ↺ V_Y = & \cellcolor{gray!20} W_{Y,0} & ⊕ & \cellcolor{gray!20} W_{Y,1} & ⊕ & \cellcolor{gray!20} W_{Y,2} & ⊕ ⋯\\[-2mm]
        & \underbrace{\hspace{1cm}}_{\mathclap{= \pr_1^*(TA)_{y}}} & &
        \multicolumn{4}{c}{\underbrace{\hspace{8cm}}_{= \pr_2^*(TZ)_{y}}}
      \end{matrix}
      $$
      \caption{Canonical decompositions on a torus cover}
      \label{fig:torusCover}
    \end{figure}}
\end{prop}

\subsubsection{Preparation for the proof of Proposition~\ref*{prop:torusCover:decomposition}}
\approvals{Daniel & yes \\ Henri & yes \\ Stefan & yes}

Proposition~\ref{prop:torusCover:decomposition} relies on a remarkable result of
Druel concerning algebraic integrability of foliations, \cite{Dru16}.  Druel
relates flat subsheaves of $𝒯_X$ to torus factors in suitable covers, but does
so only for terminal varieties.  We apply Druel's result in the following,
slightly indirect manner.

\begin{lem}[Flat summands and augmented irregularity]\label{lem:FDC0}
  Let $X$ be a projective klt variety with numerically trivial canonical divisor
  and assume that there exists a direct sum decomposition $𝒯_X = ℰ ⊕ ℱ$ where
  the locally free sheaf $ℰ|_{X_{\reg}}$ is holomorphically flat.  Then,
  $\wtilde{q}(X) ≥ \rank ℰ$.
\end{lem}
\begin{proof}
  ---

  \subsubsection*{Step 1: Proof in case where $X$ is terminal}

  Choose an ample Cartier divisor $H ∈ \Div(X)$ and recall from
  Proposition~\ref{prop:holbun2x} or \cite[Thm.~A]{Guenancia} that $𝒯_X$ is
  $H$-polystable.  The sheaf $ℱ$ thus decomposes into a direct sum of $H$-stable
  sheaves, $ℱ = ℱ_1 ⊕ ⋯ ⊕ ℱ_m$.  The intersections of $[H]^{n-1}$ with the first
  Chern class of the $ℱ_i$ clearly vanishes.  Moreover, it follows from
  Flenner's version of the Mehta-Ramanathan Theorem,
  \cite[Thm.~7.1.1]{MR2665168}, and the Bogomolov inequality,
  \cite[Thm.~3.4.1]{MR2665168}, that $c_2(ℱ_i)·[H]^{n-2} ≥ 0$ for all $i$.
  Renumbering if necessary, we find a number $k$ such that
  $c_2(ℱ_i)·[H]^{n-2} = 0$ if and only if $i≤k$.  Using the assumption that $X$
  is terminal, Druel then shows in \cite[Cor.~5.8]{Dru16} that
  $$
  \wtilde{q}(X) = \rank \bigl( ℰ ⊕ ℱ_1 ⊕ ⋯ ⊕ ℱ_k \bigr) ≥ \rank ℰ.
  $$

  \subsubsection*{Step 2: Proof in case where $X$ is canonical and $ℰ$ is locally free and flat}

  In this case, consider a terminalisation, that is, a birational crepant
  morphism $τ: \what{X} → X$ where $\what{X}$ is terminal and $ℚ$-factorial.
  The existence of a terminalisation is shown in \cite[Cor.~1.4.3]{BCHM10}.  As
  the pull-back of a flat bundle, $τ^*ℰ$ is clearly flat.  More is true.  The
  extension theorem for differential forms, \cite[Thm.~1.4]{GKKP11}, gives an
  injection $τ^*ℰ ↪ 𝒯_{\what X}$ and the polystability result of
  Guenancia, \cite[Thm.~A]{Guenancia}, shows that $τ^*ℰ$ is in fact a direct
  summand there.  We obtain inequalities
  \begin{align*}
    \wtilde{q}(X) & ≥ \wtilde{q}(\what{X}) && \text{by Item~\ref{il:AUG3} of Lemma~\ref{lem:aug}} \\
                  & ≥ \rank τ^* ℰ = \rank ℰ && \text{by Step 1}.
  \end{align*}

  \subsubsection*{Step 3: Proof in general}

  Let $α : X' → X$ be a global index-one cover, see
  Proposition~\ref{prop:global_index_one_cover}, which is quasi-étale.  The
  variety $X'$ is then canonical with trivial canonical class.  Next, let
  $β : Y → X'$ be a maximally quasi-étale cover, as given by
  \cite[Thm.~1.5]{GKP13}.  The composition $γ: Y → X$ is then quasi-étale, the
  variety $Y$ is canonical with trivial canonical class, and the algebraic
  fundamental groups $\what{π}_1(Y)$ and $\what{π}_1(Y_{\reg})$ agree.  The
  extension theorem for flat sheaves, \cite[Thm.~1.14]{GKP13} thus asserts that
  $γ^{[*]} ℰ$ is locally free and flat.  Recalling from Item~\ref{il:AUG3} of
  Lemma~\ref{lem:aug} that $\wtilde{q}(X) = \wtilde{q}(Y)$, an application of
  Step~2 to the direct sum decomposition
  $$
  𝒯_{Y} = γ^{[*]}ℰ ⊕ γ^{[*]}ℱ
  $$
  thus finishes the proof of Lemma~\ref{lem:FDC0} in the general case.
\end{proof}

\subsubsection{Proof of Proposition~\ref*{prop:torusCover:decomposition}}\label{ssect:potfd2}
\approvals{Daniel & yes \\ Henri & yes \\ Stefan & yes}

Assumption~\ref{il:FSCC1} on the augmented irregularity implies that $q(Z) = 0$.
The ample Cartier divisor $H_Y ∈ \Div(Y)$ is therefore linearly equivalent to a
sum $H_Y \sim \pr_1^* H_A + \pr_2 H_Z$, where $H_A$ and $H_Z$ are ample divisors
on $A$ and $Z$, respectively.  %
\Preprint{For a proof, choose a point $a ∈ A$ and consider the fibre $Y_a$.  The
  fibre is isomorphic to $Z$, and can be seen as an image of a section
  $σ : Z → Y$.  Set $H_Z := σ^*H_Y$ and $H_Y' := H_Y - \pr_2^* H_Z$.  The
  restriction of $H_Y'$ to $Y_a$ is linearly trivial.  Since $q(Z)=0$, the
  Picard group of $Z$ is discrete, and the restriction of $H_Y'$ to other fibres
  $((H_Y)_b)_{b ∈ A}$ is likewise trivial.  The push-forward
  $(\pr_1)_* 𝒪_A(H_Y')$ is thus invertible.  Let $H_A$ be an associated Cartier
  divisor, and observe that $H_Y$ is linearly equivalent to
  $\pr_1^* H_A + \pr_2 H_Z$, as desired.}%
\Publication{The preprint version of this paper,
  \href{https://arxiv.org/abs/1704.01408}{arXiv:1704.01408}, contains a detailed
  proof of this fact.}

As a consequence of the decomposition of $H_Y$, recall from
Proposition~\ref{prop:univProd} that the Ricci-flat Kähler manifold
$\bigl(Y_{\reg}, ω_{H_Y}\bigr)$ is a product of $\bigl(A, ω_{H_A}\bigr)$ and
$\bigl(Z_{\reg}, ω_{H_Z}\bigr)$.  Recall from \cite[Sect.~10.35 and
Cor.~10.48]{MR867684} that the (restricted) holonomy group then also decomposes.
More precisely, writing $y = (a,z)$,
\begin{align*}
  G_Y° & = \Hol\bigl( A,\, g_{H_A}\bigr)°_a ⨯ \Hol\bigl( Z_{\reg},\, g_{H_Z}\bigr)°_z \\
  G_Y & = \Hol\bigl( A,\, g_{H_A}\bigr)_a ⨯ \Hol\bigl( Z_{\reg},\, g_{H_Z}\bigr)_z.
\end{align*}

By the classical Bochner principle, the elements of a global holomorphic frame
for $TA$ are parallel with respect to $g_{H_A}$.  It follows that $g_{H_A}$ is
flat and that the holonomy group $\Hol\bigl( A,\, g_{H_A}\bigr)_a$ is trivial.
This implies that $\pr_1^* (TA)|_y ⊆ {W}_{Y,0}$.  On the other hand,
Lemma~\ref{lem:FDC0} asserts that the canonical decomposition of $𝒯_Z$ does not
contain a flat summand.  Item~\ref{il:Cx} of Proposition~\ref{prop:holbun2x}
therefore implies that the restricted holonomy of $Z$ has no nonzero fixed
points, from which we obtain the remaining inclusion
${V}_{Y,0 }⊆ \pr_1^* (TA)|_y$.  In summary, we have seen that
$$
\pr_1^* (TA)|_y ⊆ {W}_{Y,0} ⊆ {V}_{Y,0} ⊆ \pr_1^* (TA)|_y.
$$
Proposition~\ref{prop:torusCover:decomposition} is thus shown.  \qed

\subsection{Proof of Theorem~\ref*{thm:holonomyCover}}\label{ssec:pf:holonomyCover}
\approvals{Daniel & yes \\ Henri & yes \\ Stefan & yes}

The following lemma provides the last missing piece for the proof of
Theorem~\ref{thm:holonomyCover}.  Using the classification of restricted
holonomy groups found in Section~\ref{sec:restrictedHolonomy}, it relates
finiteness of $G/G°$ to the augmented irregularity of the underlying space.  It
crucially uses the existence of a ``maximally quasi-étale cover'',
\cite[Thm.~1.5]{GKP13}, and hence a global algebro-geometric result.

\begin{lem}[Finiteness of connected components]\label{lem:bfjjd0}
  In the standard Setting~\ref{setting:holonomy}, assume that $\wtilde{q}(X)=0$.
  Then, $G/G°$ is finite.  In particular, there exists a quasi-étale cover
  $γ : Y → X$ where holonomy and restricted holonomy agree, ${G}_Y = {G}_Y°$.
\end{lem}
\begin{proof}
  If $G/G°$ is finite, then the kernel of the natural surjection
  $π_1(X_{\reg}) → G/G°$ will yield the desired quasi-étale cover $γ$ claimed in
  the lemma above.
  
  Moreover, by Corollary~\ref{cor:bhqec2} finiteness of $G/G°$ follows from
  finiteness of $G_Y/G_Y°$, where $Y→ X$ is any quasi-étale cover.  Note that
  every such cover will also have vanishing augmented irregularity.  Based on
  this observation, let us make two reduction steps.  First, using the
  assumption that $\wtilde{q}(X)=0$, Item~\ref{il:FSCC2} of
  Proposition~\ref{prop:torusCover:decomposition} asserts that $V_0 = \{0\}$.
  Replacing $X$ by its weak holonomy cover, cf.\
  Proposition~\ref{prop:weakHolonomyCover}, we may therefore assume without loss
  of generality that the canonical decompositions of $V$ agree.  More precisely,
  we have $V_0 = W_0 = \{0\}$ and $V_i = W_i$ for all remaining indices
  $i$\Preprint{, as illustrated in Figure~\vref{fig:ssfs4}}.  Second, we can
  further replace $X$ by a maximally quasi-étale cover, as given by
  \cite[Thm.~1.5]{GKP13}.  This guarantees that any complex representation of
  $π_1(X_{\reg})$ factors via a representation of $π_1(X)$ by \cite[Thm.~1.14
  and its proof]{GKP13}.
  
  To prove that $G/G°$ is finite, recall that the holonomy group $G$ is a
  subgroup of $\U(V)$ and $G°$ is its maximal connected subgroup.  The group $G$
  is therefore contained in the normaliser of $G°$ in $\U(V)$.  In the special
  situation at hand, where $G$ and $G°$ are both totally decomposed with
  identical canonical decompositions, there is more we can say:
  $$
  G ⊆ N_1 ⨯ ⋯ ⨯ N_m \quad\text{where}\quad N_i := \operatorname{Norm}\bigl( G°_i ⊆ \U(V_i) \bigr).
  $$
  The groups $G°_i$ have been classified in Proposition~\ref{prop:factors}, and
  their normalisers are found by an elementary computation -- compare with
  \cite[Prop.~10.114]{MR867684} but observe that we unlike \cite{MR867684}
  consider normalisers in the unitary group and not in the orthogonal group.
  Writing $n_i := \dim V_i$, either one of the following holds.
  \begin{itemize}
  \item The group $G°_i$ is isomorphic to $\SU(n_i)$ and $N_i/G_i°≅ \U(1)$.
  \item The number $n_i$ is even, $G°_i ≅ \Sp\!\left(\frac {n_i}2\right)$,
    and $N_i/G°_i ≅ \U(1)$.
  \end{itemize}
  In summary, we see that $G/G° ⊊ \U(1)^{⨯ m}$ is Abelian.
  
  By the second reduction made at the beginning of the proof, the complex
  representation $π_1(X_{\reg}) \twoheadrightarrow G/G° ⊊ \U(1)^{⨯ m}$
  factors via $π_1(X)$, and hence further via the Abelianisation
  $H_1\bigl( X,\, ℤ \bigr)$.  In other words, we obtain a surjection
  $$
  H_1\bigl( X,\, ℤ \bigr) \twoheadrightarrow \factor{G}{G°}.
  $$
  We claim that $H_1\bigl( X,\ ℤ \bigr)$ is finite.  Indeed, otherwise
  $H¹\bigl( X,\, ℤ \bigr)$ would be of positive rank.  Moreover, the exponential
  sequence on the complex space $X^{an}$, \cite[proof of Thm.~6 in
  Chap.~K]{MR1059457}, together with GAGA would yield an embedding
  $H¹\bigl( X,\, ℤ \bigr) ↪ H¹\bigl( X,\, 𝒪_X \bigr)$.  But the latter space
  vanishes owing to $\wtilde q(X)=0$, a contradiction.  This concludes the proof
  of Lemma~\ref{lem:bfjjd0}.\Preprint{
    \begin{figure}
      \centering
      $$
      \begin{matrix}
        G° ↺ V = & \cellcolor{gray!20} V_0 & ⊕ & \cellcolor{gray!20} V_1 & ⊕ & \cellcolor{gray!20} ⋯ & ⊕ & \cellcolor{gray!20} V_m \\
        \hphantom{°}G ↺ V = & \cellcolor{gray!20} W_0 & ⊕ & \cellcolor{gray!20} W_1 & ⊕ & \cellcolor{gray!20} ⋯ & ⊕ & \cellcolor{gray!20} W_k \\[-2mm]
        & \underbrace{\hspace{.8cm}}_{= \{0\}}
      \end{matrix}
      $$
      \caption{Canonical decompositions in the proof of Lemma~\ref{ssec:pf:holonomyCover} after simplification.}
      \label{fig:ssfs4}
    \end{figure}}
\end{proof}

With all preparations in place, the proof of Theorem~\ref{thm:holonomyCover} is
now easy.

\begin{proof}[Proof of Theorem~\ref{thm:holonomyCover}]
  Let $γ$ be a composition of quasi-étale covers,
  $$
  \xymatrix{ %
    A⨯Z' \ar[r]^b & A⨯Z \ar[rr]^{a}_{\text{torus cover}} && X
  }
  $$
  where $a$ is a torus cover of $X$, and $b$ is the product of the identity
  $\Id_A$ and the cover of $Z$ constructed in Lemma~\ref{lem:bfjjd0}.  One
  checks immediately that $γ$ is a torus cover, and also satisfies remaining
  conditions spelled out in Theorem~\ref{thm:holonomyCover}.
\end{proof}

\subsection{Proof of Corollary~\ref*{cor:flatDecomp}}\label{ssect:pocfd}
\approvals{Daniel & yes \\ Henri & yes \\ Stefan & yes}

Let $γ : Y → X$ be a holonomy cover.  Item~\ref{il:FDC1} is a consequence of
Corollary~\ref{cor:bhqec2} and Lemma~\ref{lem:bfjjd0}.  Item~\ref{il:FDC2}
follows immediately from \ref{il:FSCC0x} and \ref{il:FSCC2x} combined.
Item~\ref{il:FDC3} follows from Lemma~\ref{lem:FDC0}, using the observation that
$𝒲°_i$ is holomorphically flat if and only if its pull-back $γ^*𝒲_i°$ is
holomorphically flat on $A⨯Z_{\reg}$.  \qed

\subsection{Decomposition of the tangent sheaf}\label{subsect:proof_of_intro_prop}
\approvals{Daniel & yes \\ Henri & yes \\ Stefan & yes}

Combining the existence of a holonomy cover with the classification of
restricted holonomy, we obtain a holonomy cover whose canonical decomposition of
the tangent sheaf admits are particularly precise description.

\begin{prop}[Decomposition of the tangent sheaf]\label{prop:nnnew_intro_prop}
  In the standard Setting~\ref{setting:holonomy}, there exists a holonomy cover
  $γ : A ⨯ Z → X$ such that in addition to properties \ref{il:B1}--\ref{il:B3},
  there exists a direct sum decomposition of the tangent sheaf of $Z$,
  $$
  𝒯_Z = \bigoplus\nolimits_{i∈ I} ℰ_i ⊕ \bigoplus\nolimits_{j∈ J}ℱ_j,
  $$
  where the reflexive sheaves $ℰ_i$ (resp.\ $ℱ_j$) satisfy the following
  properties.
  \begin{enumerate}
  \item\label{iln:new_i1} The subsheaves $ℰ_i ⊆ 𝒯_Z$ (resp.\ $ℱ_j ⊆ 𝒯_Z$) are
    foliations with trivial determinant, of rank $n_i ≥ 3$ (resp.\ of even rank
    $2m_j ≥ 2$).  Moreover, they are strongly stable in the sense of
    \cite[Def.~7.2]{GKP16}.
  \item\label{iln:new_i2} On $Z_{\reg}$, the $ℰ_i$ (resp.\ $ℱ_j$) are locally
    free and correspond to holomorphic subbundles $E_i$ (resp.~$F_j$) of
    $TZ_{\reg}$ that are parallel with respect to the Levi-Civita connection of
    $g_Z$, the Riemannian metric on $Z_{\reg}$ induced by $ω_Z$.  Moreover,
    their holonomy groups are $\SU(n_i)$ and $\Sp(m_j)$, respectively.
  \item\label{iln:new_i3} If $x ∈ X_{\reg}$ and $(a,z) ∈ γ^{-1}(x)$, then the
    splitting
    $$
    T_x X ≅ T_{(a,z)} (A ⨯ Z) = T_a A ⊕ \bigoplus\nolimits_{i∈ I} E_{i,z}
    ⊕ \bigoplus\nolimits_{j∈ J}F_{j,z}
    $$
    corresponds to the decomposition of $T_x X$ into irreducible representations
    under the action of the restricted holonomy group $\Hol(X_{\reg},g_H)$ at
    $x$.
  \end{enumerate}
\end{prop}
\begin{proof}
  Theorem~\ref{thm:holonomyCover} yields a quasi-étale cover $γ: A ⨯ Z → X$ that
  satisfies the properties listed in \ref{il:FSCC0x}--\ref{il:FSCC3x}.  Observe
  that all these remain true if we replace $Z$ by a further quasi-étale cover.
  According to Section~\ref{ssect:comp}, we may therefore choose $Z$ such that
  the summands in the canonical decomposition $𝒯_Z = 𝒲_0 ⊕ ⋯ ⊕ 𝒲_k$ are strongly
  stable foliations with trivial determinant.  By construction, the $𝒲_i$ are
  locally free on $Z_{\reg}$ and correspond to holomorphic subbundles of
  $TZ_{\reg}$ that are parallel with respect to the Levi-Civita connection.
  
  We prove Item~\ref{iln:new_i2} next.  Corollary~\ref{cor:flatDecomp}
  guarantees that none of the summands $𝒲_i$ is holomorphically flat.  Since
  determinants are trivial, this already guarantees that $\rank 𝒲_i ≥ 2$ for all
  $i$.  More is true.  Theorem~\ref{thm:holonomyCover} guarantees holonomy and
  restricted holonomy agree, $G_Y° = G_Y$, and the classification of the
  restricted holonomy, Proposition~\ref{prop:factors} guarantees that the
  associated factors of the holonomy group
  $G = \Hol(Z_{\reg}, g_Z) = G_1 ⨯ ⋯ ⨯ G_k$ are isomorphic to either $\SU(n_i)$
  or to $\Sp(m_j)$, respectively.  Writing $ℰ_i$ for those summands with
  holonomy $\SU(n_i)$ and $ℱ_j$ for the others, Item~\ref{il:new_i2} follows.

  Item~\ref{iln:new_i1} summarises the results of the two paragraphs above.
  Item~\ref{iln:new_i3} holds by construction.
\end{proof}

\part{The Bochner principle on singular spaces}\label{part:III}
%
%
\svnid{$Id: 08-BP.tex 763 2018-10-27 13:39:14Z kebekus $}

\section{The Bochner principle for reflexive tensors and bundles}
\subversionInfo
\label{sec:holprin}

\subsection{The Bochner Principle}
\approvals{Daniel & yes \\ Henri & yes \\ Stefan & yes}

The goal of this section is to establish the following ``Bochner principle'',
which generalises \cite[Thm.~C]{Guenancia}.

\begin{thm}[Bochner principle for bundles]\label{generalized:holonomy}
  In the standard Setting~\ref{setting:holonomy}, given any two numbers
  $p,q ∈ ℕ$, consider the following sheaf of ``reflexive tensors'',
  $$
  ℰ := \bigl(𝒯_X^{⊗ p}⊗(𝒯_X^*)^{⊗ q} \bigr)^{**}.
  $$
  Denote the associated bundle on $X_{\reg}$ by $E$ and recall that the holonomy
  group $G$ acts on the vector space $E_x$ in a canonical manner.  Then, the
  following holds.
  \begin{enumerate}
  \item\label{il:l1} There exists a direct sum decomposition $ℰ = ℰ_1 ⊕ ⋯ ⊕ ℰ_k$
    whose summands are stable of slope zero with respect to any polarisation.
    The induced subbundles $ℰ|_{X_{\reg}}$ of $E$ are parallel with respect to
    the connection induced by the Chern connection of $(TX_{\reg},h_H)$.  In
    particular, $ℰ$ is polystable with respect to any polarisation.
  \item\label{il:l2} The fibre map $ℰ → E_x$ induces a one-to-one correspondence
    between arbitrary direct summands of $ℰ$ and $G$-invariant complex subspaces
    of $E_x$.
  \end{enumerate}
\end{thm}

A non-zero section $σ$ of $ℰ$ is the same thing as a trivial subsheaf $𝒪_X ⊂ ℰ$.
Using semistability arguments, group-theoretic considerations and the Bochner
principle just proven, we will show that every such section yields a $G$-fixed
point $σ_x ∈ E_x$ and vice versa.

\begin{thm}[Bochner principle for reflexive tensors]\label{thm:holonomy}
  Setting as in Theorem~\ref{generalized:holonomy}.  Then, the natural
  evaluation map induces an isomorphism $H⁰\bigl(X,\, ℰ\bigr) → E_x^G$.  In
  particular, the restriction of every holomorphic tensor to $X_{\reg}$ is
  parallel.
\end{thm}

\begin{rem}
  As $Ω_{\wtilde{X}}^{[p]}$ and $\Sym^{[p]} Ω¹_X$ are direct summands of
  $\bigl((𝒯_{\wtilde{X}}^*)^{⊗ p}\bigr)^{**}$, the Bochner principle also
  applies to (symmetric) differential forms.
\end{rem}

Theorems~\ref{generalized:holonomy} and \ref{thm:holonomy} will be shown in
Sections~\ref{ssec:pothb} and \ref{ssec:pothf} below.  The following corollary,
which we promised in Section~\ref{ssec:dofd}, gives an additional description of
the summands in the canonical decomposition of $𝒯_X$.  \Publication{The preprint
  version of this paper,
  \href{https://arxiv.org/abs/1704.01408}{arXiv:1704.01408}, contains a detailed
  proof.}

\begin{cor}[Summands in the canonical decomposition of $𝒯_X$]\label{cor:holbun2w}
  In the standard Setting~\ref{setting:holonomy}, consider the canonical
  decomposition of $𝒯_X$, as in \eqref{decomp}.  Then, the sheaf $𝒲_0$ is
  trivial and none of the sheaves $𝒲_1$, …, $𝒲_k$ and $𝒲^*_1$, …, $𝒲^*_k$ admits
  a section.  \Publication{\qed}
\end{cor}
\Preprint{%
  \begin{proof}
    The triviality of $𝒲_0$ has already been established in Item~\ref{il:Bx} of
    Proposition~\ref{prop:holbun2x}.  Now, let us assume that $𝒲_i$ has a
    non-zero section $s$ for some $i>0$.  By the Bochner principle for reflexive
    tensors, Theorem~\ref{thm:holonomy}, $s$ is parallel hence $s_x≠ 0$ and the
    line $ℂ·s_x ⊂ W_i$ is point-wise fixed by $G$.  Irreducibility of $W_i$
    implies that $ℂ·s_x = W_i$ which leads to $W_i ⊂ W_0$ and the contradiction
    follows.  The same argument applies to $𝒲_i^*$.
  \end{proof}
}

\subsection{Applications}\label{ssec:aobp}
\approvals{Daniel & yes \\ Henri & yes \\ Stefan & yes}

The Bochner principle allows us to carry over many arguments from the smooth to
our singular case.  As a sample application, we generalise \cite[Sect.~1, Item
(iii) of Cor.]{Bea82} to our setup.  See Corollary~\ref{qab1} as well as
\cite[Sect.~3]{Pe94} for related results.

\begin{cor}[Detecting finite quotients of Abelian varieties, I]\label{cor:forms}
  In the standard Setting~\ref{setting:holonomy}, the following inequality holds
  for all $p ∈ ℕ$,
  \begin{equation}\label{ineq2}
    h⁰ \bigl(X,\, Ω_X^{[p]} \bigr) ≤ \binom{n}{p}.
  \end{equation}
  If equality holds for one $0 < p < n$, then $X$ is of the form $A/Γ$ where $A$
  is an Abelian variety and $Γ$ is a finite group whose action on $A$ is free in
  codimension one.
\end{cor}

\begin{rem}
  If $X$ is smooth and equality holds in \eqref{ineq2} for some $0<p<n$,
  Beauville proves that $X$ is an Abelian variety already, so that no
  quasi-étale cover is needed.  However, this is false in general if $X$ is
  singular.  Indeed, take $X=A/\langle \pm 1 \rangle$, the quotient of an
  Abelian fourfold $A$ by the natural involution $x ↦ -x$.  As every holomorphic
  two-form on $A$ is preserved by the involution, the maximal value
  $h⁰ \bigl(X,\, Ω_X^{[2]} \bigr) = 6$ is attained, yet $X$ is not Abelian.
\end{rem}

\begin{proof}[Proof of Corollary~\ref{cor:forms}]
  \CounterStep Inequality~\eqref{ineq2} follows immediately from the Bochner
  principle for $ℰ := (𝒯^*_X)^{[p]}$.  Now assume that there exists one index
  $0<p<n$ for which equality holds.  We employ the decomposition
  $V = V_0 ⊕ V_1 ⊕ ⋯ ⊕ V_m$ induced by the restricted holonomy group
  $G° = G°_1 ⨯ ⋯ ⨯ G°_m$, as discussed in
  Construction~\ref{constr:canonDecompRestr}.  The Bochner principle gives an
  embedding $H⁰ \bigl( X,\, Ω_X^{[p]} \bigr) ↪ (Λ^pV^*)^{G°}$, and therefore
  shows that
  $$
  h⁰ \bigl( X,\, Ω_X^{[p]} \bigr) ≤ \dim (Λ^pV^*)^{G°} ≤ \sum_{\substack{k_0, …,
      k_m ∈ ℕ \\ \sum k_j = p}} (\dim Λ^{k_0} V_0^*)·\prod_{i=1}^m \dim
  \bigl(Λ^{k_i}V_i^*\bigr)^{G_i°}.
  $$
  If $m > 0$, then the factors on the right are obviously estimated as follows,
  \begin{equation}\label{ineq1}
    \dim \bigl(Λ^{k_i}V_i^* \big)^{G_i°} ≤ \dim Λ^{k_i}V_i^* \qquad \text{for all $i>0$ and all $k_i$.}
  \end{equation}
  To conclude, by using Vandermonde's Identity for binomial coefficients observe
  that equality in \eqref{ineq2} can only happen if equality happens in
  \eqref{ineq1} for all $i>0$ and all $k_i$.  But then the groups $G°_i$ would
  necessarily be trivial, given Proposition~\ref{prop:factors}.  The assumption
  that $m > 0$ is thus absurd.  Equality in \eqref{ineq2} therefore implies that
  $V_0=V$, and Corollary~\ref{cor:forms} follows from
  Theorem~\ref{thm:holonomyCover}.
\end{proof}

\begin{cor}[Restricted holonomy is independent of choice of polarisation]\label{cor:independence}
  In the standard Setting~\ref{setting:holonomy}, the isomorphism class of the
  restricted holonomy group $\Hol(X_{\reg},g_H)°$ does not depend on the ample
  polarisation $H$.
\end{cor}
\begin{proof} Assume we are given two ample Cartier divisors $H$, $H'$ on $X$.
  Theorem~\ref{thm:holonomyCover} allows to pass to a holonomy cover.  Since
  this does not affect the restricted holonomy, we may assume without loss of
  generality that the holonomy groups $G_H$ and $G_{H'}$ are connected already
  on $X$.  The classification of restricted holonomy,
  Proposition~\ref{prop:factors}, thus equips these groups with a product
  structure.  More precisely, we obtain two canonical decompositions of the
  tangent sheaf into stable subsheaves,
  \begin{equation*}\label{decomps}
    𝒯_X ≅ 𝒪_X^{⊕ r}⊕ \bigoplus_i ℰ_i ≅ 𝒪_X^{⊕ r'} ⊕ \bigoplus_j ℰ'_j,
  \end{equation*}
  where the respective holonomy groups act trivially on the trivial parts and
  are given by the products $G_H = \prod G_i $ and $G_{H'} = \prod G_j'$, where
  $G_i$ (resp.\ $G_j'$) is isomorphic either to the the special unitary group or
  the unitary symplectic group in dimension $\rank ℰ_i ≥ 2$ (resp.\
  $\rank ℰ_j' ≥ 2$).  Clearly, $r = r'$, and the remaining stable factors have
  to be pairwise isomorphic.  Up to renumbering, one can assume that
  $ℰ_i ≅ ℰ'_i$ and therefore
  $$
  \bigoplus_p H⁰ \bigl(X,\,Λ^{[p]}ℰ_i \bigr) ≅ \bigoplus_p H⁰\bigl(X,\,Λ^{[p]}{ℰ'_i}\bigr).
  $$
  The conclusion then follows from Bochner principle for reflexive tensors,
  Theorem~\ref{thm:holonomy}, once we observe that $G_i$ is special unitary if
  and only if
  \begin{align*}
    h⁰ \bigl( X,\, Λ^{[p]}ℰ_i \bigr) & =
                                       \begin{cases}
                                         1 & \text{if $p = 0$ or $p = \rank ℰ_i$} \\
                                         0 & \text{otherwise.}
                                       \end{cases} \\
    \intertext{and that $G_i$ is unitary symplectic if and only if}
 h⁰ \bigl( X,\, Λ^{[p]}ℰ_i \bigr) & =
    \begin{cases}
      1 & \text{if $0 \le p \le \rank ℰ_i$ and $p$ is even} \\
      0 & \text{otherwise.} \\
    \end{cases} \qedhere
  \end{align*}
\end{proof}

%
%
\svnid{$Id: 09-BPbundles.tex 760 2018-10-16 10:51:15Z greb $}

\section{Proof of Theorem~\ref*{generalized:holonomy} (``Bochner principle for bundles'')}\label{ssec:pothb}
\subversionInfo
\approvals{Daniel & yes \\ Henri & yes \\ Stefan & yes}

We maintain notation and assumptions of Theorem~\ref{generalized:holonomy} in
this section.  The proof is similar to the one of \cite[Thm.~A]{Guenancia},
which deals with the case $ℰ = 𝒯_X$.

\subsection*{Step 1.  Setup}
\approvals{Daniel & yes \\ Henri & yes \\ Stefan & yes}

Let $π: \wtilde{X} → X$ be a strong log-resolution of singularities of $X$.
Throughout the proof, we discuss the following objects on $\wtilde{X}$.

\subsubsection*{Discrepancy divisors}

Write the standard $ℚ$-linear equivalence
\begin{equation}\label{eq:discrepancy}
  K_{ \wtilde{X}}+ \sum a_i·D_i \sim_ℚ π^* K_X,
\end{equation}
where $D = \sum a_i D_i$ is a $π$-exceptional $ℚ$-divisor with simple normal
crossing support and coefficients $a_i ∈ (-∞, 1) ∩ ℚ$.  Choose sections
$s_i ∈ H⁰ \bigl( \wtilde{X},\, 𝒪_{\wtilde{X}}(D_i) \bigr)$ that vanish precisely
on $D_i$.

\subsubsection*{Metrics and currents}

We fix a Kähler reference metric $ω_0$ on $\wtilde{X}$, but we will also
consider the pull-back $ω_{\wtilde{H}} := π^*ω_H$.  \Preprint{Recall from
  Remark~\ref{rem:pboc} that ``pull-back'' is meaningful in our context.} Set
$\wtilde{H} := π^*H ∈ \Div(\wtilde{X})$ and observe that $\Ric ω_{\wtilde{H}}$
equals the current of integration of $D$, $\Ric ω_{\wtilde{H}} = [D]$.

Equip the line bundles $𝒪_{\wtilde{X}}(D_i)$ with Hermitian metrics $|·|_i$, and
write $Θ_i ∈ \cA^{1,1}_{ℝ}(\wtilde{X})$ for the associated curvature form.

\subsubsection*{Bundles}

In analogy to the definition of $E$, consider the bundle
$\wtilde{E} := T{\wtilde{X}}^{⊗ p}⊗ (T^*\wtilde{X})^{⊗ q}$, as well as the
associated locally free sheaf
$\wtilde{ℰ} := 𝒯_{\wtilde{X}}^{⊗ p}⊗ (𝒯_{\wtilde{X}}^*)^{⊗ q}$.  We aim to
establish polystability of $E$ by studying $\wtilde{E}$.

\subsection*{Step 2.  Construction of smooth metrics}
\approvals{Daniel & yes \\ Henri & yes \\ Stefan & yes}

As in \cite[p.~518]{Guenancia}, we aim to construct sequences of smooth metrics
converging to $ω_{\wtilde{H}}$.

\begin{construction}[Construction of smooth metrics $ω_{t,ε}$]\label{rem:ma}
  Given a pair of numbers $ε,t ∈ ℝ^+$, consider the following form,
  \begin{equation}\label{eq:xbfgg}
    θ_ε := \sum a_i \Bigl(\underbrace{Θ_i+dd^c \log \bigl(|s_i|²_i+ε² \bigr)}_{=:
      θ_{i,ε}} \Bigr) \quad \text{in } \cA^{1,1}_{ℝ}(\wtilde{X}).
  \end{equation}
  We view $θ_ε$ as a regularisation of the current of integration $[D]$.  The
  discrepancy formula \eqref{eq:discrepancy} implies that
  $\{θ_ε\} = c_1(\wtilde{X})$ in $H^{1,1}(\wtilde{X})$.  We let $ω_{t,ε}$ be the
  unique Kähler metric on $\wtilde{X}$ whose class equals
  \begin{equation}\label{eq:class_of_omega}
    \{ω_{t,ε}\} = c_1(\wtilde{H})+t·\{ω_0\} ∈ H^{1,1}\bigl(\wtilde{X}\bigr),
  \end{equation}
  and that solves the equation
  \begin{equation}\label{eq:cxvxc}
    \Ric ω_{t,ε}= θ_ε.
  \end{equation}
  For existence and uniqueness, see \cite{MR480350}.
\end{construction}

\begin{rem}[Convergence of $ω_{t,ε}$ for $(ε,t) → (0,0)$]
  It follows from the proof of \cite[Thm.~3.5]{MR2505296} and from the
  uniqueness result \cite[Thm.~3.3]{GZ07} that the smooth Kähler forms
  $ω_{t, ε}$ converge on $\wtilde{X} ∖ \supp D$ to the singular
  Kähler-Einstein metric $ω_{\wtilde{H}}$.  More precisely, we have
  $$
  \lim_{(ε,t) → (0,0)} ω_{t,ε} = ω_{\wtilde{H}}
  $$
  on $\wtilde{X} ∖ \supp D$ in the $\cC^{∞}_{\loc}$-topology.
\end{rem}

\begin{rem}\label{rem:hgfjgh}
  Let $∇'_{\negthinspace i}$ be the $(1,0)$-part of the Chern connection of
  $(𝒪_{\wtilde{X}}(D_i), |·|_i)$.  A direct computation shows that the form
  $θ_{i,ε} ∈ \cA^{1,1}_{ℝ}(\wtilde{X})$ decomposes as follows,
  $$
  θ_{i,ε} = \underbrace{\frac{ε²·|∇_{\negthinspace
        i}'s_i|²_i}{(|s_i|²_i+ε²)²}}_{=: β_{i,ε}} +
  \underbrace{\frac{ε²·Θ_i}{|s_i|²_i+ε²}}_{=: γ_{i,ε}}.
  $$
  We refer to \cite[beginning of Sect.~3]{MR3134683} for an analogous
  computation.  The summands $β_{i,ε}$ and $γ_{i,ε}$ are smooth forms in
  $\cA^{1,1}_{ℝ}(\wtilde{X})$.
\end{rem}

\begin{notation}[Hermitian metrics and curvature on $T\wtilde{X}$]
  We endow $T\wtilde{X}$ with the Kähler form $ω_{t,ε}$, so $T{\wtilde{X}}$ and
  more generally $\wtilde{E}$ can be equipped with the structure of a
  holomorphic Hermitian vector bundles.  We denote by $\wtilde{h}_{t,ε}$ the
  Hermitian metric on $\wtilde{E}$ induced by $ω_{t,ε}$ and write
  $Θ_{\wtilde{h}_{t,ε}}(\wtilde{E}) ∈ \cA^{1,1} \bigl( \wtilde{X},\,
  \End(\wtilde{E}) \bigr)$ for its Chern curvature.  The form
  $iΘ_{\wtilde{h}_{t,ε}}(\wtilde{E})$ is a real $(1,1)$-form with values in the
  Hermitian endomorphisms of $(E,\wtilde{h}_{t,ε})$.
\end{notation}

In the course of the proof we will need the following result, which is proved in
\cite[Lem.~3.7]{Guenancia}.  In our context, there is a much simpler proof
though, which we give below for the convenience of the reader.

\begin{claim}\label{claim:lem2}
  For every fixed $t>0$, and every index $i$ we have
  $$
  0 = \lim_{ε → 0} \: \int_{{\wtilde{X}}} \frac{ε²}{|s_i|²_i+ε²}·ω_0 Λ
  ω_{t,ε}^{n-1}.
  $$
\end{claim}
\begin{proof}
  As the total mass of $ω_0 Λ ω_{t,ε}^{n-1}$ is independent of $ε$, it is enough
  to prove that $\int_{\{|s_i|²<ε\}}ω_0 Λ ω_{t,ε}^{n-1}$ converges to $0$ as $ε$
  approaches zero.  The important observation is that the potentials $φ_{t,ε}$
  of $ω_{t,ε}$ are uniformly bounded in $ε$, when $t$ is fixed.  This is a
  consequence of Kołodziej's $L^p$-estimate, cf.\ \cite[proof of Thm.~2.4.2,
  Ex.~2]{MR1618325}, because $ω_{t,ε}^n$ has a density $f_{ε}$ with respect to a
  volume form satisfying $\lVert f_{ε} \rVert_{L^p} ≤ C$ where $C$ is
  independent of $ε$.
  
  Next, one can introduce a family of cut-off functions $(χ_ε)_{ε>0}$ for the
  divisor $D_i$ as in \cite[Sect.~9]{MR3134683}.  These functions satisfy the
  important property that their complex Hessian, $dd^c χ_ε$, is uniformly
  dominated by a metric $ω_P$ with Poincaré type along $D_i$.  Then, one can
  perform successive integrations by parts in a similar way as in the proof of
  Chern-Levine-Nirenberg inequality and see that there exists a uniform constant
  $C>0$ such that
  $\int_{\{|s_i|²<ε\}}ω_0 Λ ω_{t,ε}^{n-1} \le C \int_{\{|s_i|²<ε\}}ω_P^n$.  The
  constant incorporates the sup-norms of the various potentials above.
  Claim~\ref{claim:lem2} follows at once from the finiteness of the volume of
  the Poincaré type metric $ω_P$.
\end{proof}

\subsection*{Step 3.  Computing slopes using $ω_{t,ε}$}
\approvals{Daniel & yes \\ Henri & yes \\ Stefan & yes}
  
Given a saturated subsheaf $ℱ ⊆ ℰ$, we aim to lift $ℱ$ to a subsheaf of
$\wtilde{ℱ} ⊆ \wtilde{ℰ}$, and to compute the slope of $\wtilde{ℱ}$ with respect
to the Kähler metrics $ω_{t,ε}$.

\begin{setupnot}\label{sn:fghfg}
  Assume we are given a saturated subsheaf $ℱ ⊆ ℰ$, which will automatically be
  reflexive.  We write $\wtilde{ℱ} ⊆ \wtilde{ℰ}$ for the unique saturated
  subsheaf that agrees with $ℱ$ wherever the resolution morphism $π$ is
  isomorphic.  We denote the singularity set of $\wtilde{ℱ} ⊆ \wtilde{ℰ}$ by
  $W ⊂ \wtilde{X}$.  This is the minimal closed subset of $\wtilde{X}$ outside
  which $\wtilde{ℱ} ⊆ \wtilde{ℰ}$ corresponds to a subbundle, which we denote as
  $\wtilde{F} ⊆ \wtilde{E}|_{\wtilde{X} ∖ W}$.  Recall that $W$ has
  codimension at least two.  The restriction of $\wtilde{h}_{t,ε}$ endows
  $\wtilde{F}$ with a Hermitian structure.  Write
  $$
  Θ_{\wtilde{h}_{t,ε}}(\wtilde{F}) ∈ \cA^{1,1} \bigl( \wtilde{X} ∖ W ,\,
  \End(\wtilde{F}) \bigr)
  $$
  for its Chern curvature.
\end{setupnot}

\begin{claim}\label{claim:8.10}
  Setup and notation as in \ref{sn:fghfg}.  Then, the following inequality
  holds.
  \begin{equation}\label{prein}
    n·c_1(\wtilde{ℱ}) · \{ω_{t,ε}\}^{n-1} ≤
    \int_{\wtilde{X}∖ W} \tr_{\End} \Bigl(\pr_{\wtilde{F}} \bigl(
    \tr_{ω_{t,ε}} iΘ_{\wtilde{h}_{t,ε}}(\wtilde{E})|_{\wtilde{F}} \bigr) \Bigr)·ω_{t,ε}^n.
  \end{equation}
  We refer to the right hand side of \eqref{prein} as the \emph{error term}.
\end{claim}

\begin{explanation}\label{exp:dfgds}
  In Equation~\eqref{prein}, $\pr_{\wtilde{F}}$ is the orthogonal projection
  $\wtilde{E} → \wtilde{F}$.  The symbol $\tr_{ω_{t,ε}}$ denotes the trace
  relatively to the Kähler metric $ω_{t,ε}$.  Given a bundle $G$ and a
  $G$-valued form $α ∈ \cA^{1,1}(G)$, recall that $\tr_{ω_{t,ε}} α$ is the
  unique section of $G$ such that
  $(\tr_{ω_{t,ε}} α) ⊗ ω_{t,ε}^n = n·α Λ ω_{t,ε}^{n-1}$.  The object
  $\tr_{ω_{t,ε}} iΘ_{\wtilde{h}_{t,ε}}(\wtilde{E})|_{\wtilde{F}}$ in
  \eqref{prein} is therefore a section of $\Hom(\wtilde{F}, \wtilde{E})$.
\end{explanation}

\begin{proof}[Proof of Claim~\ref{claim:8.10}]
  \CounterStep We aim to relate the curvature of the subbundle
  $(\wtilde{F}, \wtilde{h}_{t,ε})$ to the one of
  $(\wtilde{E}, \wtilde{h}_{t,ε})$.  Classically, this is done by introducing
  the second fundamental form
  $ρ_{t,ε} ∈ \cA^{1,0} \bigl( \wtilde{X} ∖ W ,\, \Hom(\wtilde{F},
  \wtilde{F}^{\perp}) \bigr)$, cf.\ \cite[Sect.~V.14]{DemaillyBook2012}, which
  satisfies the relation
  $$
  Θ_{\wtilde{h}_{t,ε}}(\wtilde{F}) = \pr_{\wtilde{F}} \bigl(
  Θ_{\wtilde{h}_{t,ε}}(\wtilde{E})|_{\wtilde{F}} \bigr) + ρ_{t,ε}^* Λ ρ_{t,ε}
  \quad \text{in } \cA^{1,1} \bigl(\wtilde{X} ∖ W ,\, \End( \wtilde{F} )
  \bigr).
  $$
  Multiplying by $i$, taking the trace (as endomorphism) and wedging with
  $ω_{t,ε}^{n-1}$ ---and observing that the operations of taking the metric
  trace and taking the endomorphism trace commute--- we obtain the following
  identity,
  \begin{multline}\label{sff}
    c_1(\wtilde{F},\wtilde{h}_{t,ε}) Λ ω_{t,ε}^{n-1} = \tr_{\End}
    \Bigl(\pr_{\wtilde{F}} \bigl( \tr_{ω_{t,ε}} iΘ_{\wtilde{h}_{t,ε}}
    (\wtilde{E})|_{\wtilde{F}} \bigr) \Bigr)·\frac{ω_{t,ε}^n}{n} \\
    + \underbrace{\tr_{\End}(i ρ_{t,ε}^* Λ ρ_{t,ε} Λ ω_{t,ε}^{n-1}
      )}_{\text{seminegative}} \quad \text{in } \cA^{n,n}(\wtilde{X}∖
    W).
  \end{multline}

  To make use of \eqref{sff}, recall the following identity, which follows for
  instance from \cite[Eq.~(**) on p.~181]{Kob87}\footnote{see also
    \cite[Prop.~3.8, Case 1]{Guenancia}},
  \begin{equation}\label{eq:c1}
    \int_{\wtilde{X}∖ W}c_1(\wtilde{F},\wtilde{h}_{t,ε}) Λ ω_{t,ε}^{n-1}
    = c_1(\wtilde{ℱ}) · \{ω_{t,ε}\}^{n-1}.
  \end{equation}
  Indeed, Inequality~\eqref{prein} follows by integrating \eqref{sff} over
  $\wtilde{X}∖ W$ with the help of \eqref{eq:c1}.
  Claim~\ref{claim:8.10} follows.
\end{proof}

\subsection*{Step 4.  Analysis of the error term}
\approvals{Daniel & yes \\ Henri & yes \\ Stefan & yes}

The aim of the current step is to state the following claim.

\begin{claim}\label{claim:hhg}
  Setup and notation as in \ref{sn:fghfg}.  Then, the error term converges to
  zero as $t,ε → 0$.  More precisely.
  \begin{equation}\label{claim}
    \lim_{(ε,t)→(0,0)} \: \int_{\wtilde{X}∖ W} \tr_{\End} \Bigl(\pr_{\wtilde{F}} \bigl( \tr_{ω_{t,ε}} iΘ_{\wtilde{h}_{t,ε}}(\wtilde{E})|_{\wtilde{F}} \bigr) \Bigr)·ω_{t,ε}^n = 0.
  \end{equation}
\end{claim}

Before proving Claim~\ref{claim:hhg} in Step~7 below, we need to introduce
notation and establish a number of auxiliary results.

\begin{notation}\label{not:boxp}
  If $V$ is a complex vector space of dimension $n$ and $f∈ \End(V)$, we denote
  by $f^{\boxtimes p}$ the endomorphism of $V^{⊗ p}$ defined on pure tensors by
  $$
  f^{\boxtimes p}(v_1⊗⋯ ⊗ v_p) := \sum_{i=1}^p v_1 ⊗ ⋯ ⊗v_{i-1} ⊗ f(v_i) ⊗
  v_{i+1} ⊗ ⋯ ⊗ v_p
  $$
\end{notation}

\begin{obs}\label{obs:gh}
  In the setting of Notation~\ref{not:boxp}, one has
  $\tr \bigl(f^{\boxtimes p} \bigr) = p·n^{p-1}·\tr(f)$.  If $V$ has an
  Hermitian structure and if $f$ is Hermitian semipositive, then so is
  $f^{\boxtimes p}$.
\end{obs}

\subsection*{Step 5: Analysis of $\tr_{ω_{t,ε}} Θ_{\wtilde{h}_{t,ε}}(\wtilde{E})$}
\approvals{Daniel & yes \\ Henri & yes \\ Stefan & yes}

One fundamental object that appears in the error term is
$\tr_{ω_{t,ε}} iΘ_{\wtilde{h}_{t,ε}}(\wtilde{E})$, which is a Hermitian
endomorphism of the bundle $\wtilde{E}$.  The following claim relates it to the
Ricci curvature of $ω_{t,ε}$.  Its formulation uses the operator
``$\sharp_{t,ε}$''.  We briefly recall the definition.

\begin{consnot}\label{consnot:sharp}
  Using the Kähler metrics $ω_{t,ε}$, one constructs from any $(0,1)$-form $η$ a
  $(1,0)$-vector field $\sharp_{t,ε} η$, requiring that the relation
  $$
  η(\bar{ζ}):=g_{t,ε}(i\sharp_{t,ε}η,ζ)
  $$
  holds for any vector field $ζ$ of type $(1,0)$, where $g_{t,ε}$ is the
  Hermitian metric on $T^{1,0}\wtilde{X}$ associated with the Kähler form
  $ω_{t,ε}$.  Next, one extends the operator $\sharp_{t,ε}$ to vector-valued
  forms.  In particular, if $α ∈ \cA^{1,1}(\wtilde{X})$, one can see $α$ as a
  $(0,1)$-form with values in $(T^{1,0}\wtilde{X})^*$ and define
  $\sharp_{t,ε} α ∈ \End(T\wtilde{X})$ as follows: in local coordinates, let us
  write $ω_{t,ε} = i\sum_{j,k} g_{j\bar k}·dz_j Λ d\bar z_k$.  Let
  $(g^{j\bar k})$ be the inverse of $(g_{j\bar k})$ let
  $α = i\sum_{j,k} α_{j\bar k}·dz_j Λ d\bar z_{ k}$ be a $(1,1)$-form.  Then,
  $$
  \sharp_{t,ε} α = \sum_{j,k,ℓ} α_{j\bar k} g^{ℓ\bar k}·dz_j
  ⊗\frac{∂}{∂ z_{ℓ}}.
  $$
  It is easy to check the formula
  \begin{equation}\label{eq:trace}
    \tr_{\End} \sharp_{t,ε} α = \tr_{ω_{t,ε}} α.
  \end{equation}
  Finally, the endomorphism $\sharp_{t,ε} α$ is Hermitian (resp.\ Hermitian
  semipositive) with respect to $g_{t,ε}$ if $α$ is real (resp.\ semipositive).
\end{consnot}

\begin{claim}\label{claim:hghgf}
  Setup and notation as in \ref{sn:fghfg}.  Then,
  \begin{equation}\label{eq:E}
    \tr_{ω_{t,ε}} iΘ_{\wtilde{h}_{t,ε}}(\wtilde{E}) =
    (\sharp_{t,ε} θ_ε)^{\boxtimes p}⊗\Id_{T^*\wtilde{X}^{⊗ q}} -
    \Id_{T\wtilde{X}^{⊗ p}} ⊗ \overline{(\sharp_{t,ε} θ_ε)}^{\boxtimes q}
    \quad \text{in } \End( \wtilde{E} ).
  \end{equation}
\end{claim}
\begin{proof}[Proof of Claim~\ref{claim:hghgf}]
  Following the standard computations of \cite[p.~524]{Guenancia}, one obtains
  the following identities.
  \begin{align*}
    n·iΘ_{\wtilde{h}_{t,ε}}\bigl( T\wtilde{X}^{⊗ p} \bigr) Λ ω_{t,ε}^{n-1} & = \hphantom{-}(\sharp_{t,ε} \Ric ω_{t,ε})^{\boxtimes p}·ω_{t,ε}^n && \text{in } \cA^{n,n}\bigl(\wtilde{X}, \End( T\wtilde{X}^{⊗ p} )\bigr) \\
    n·iΘ_{\wtilde{h}_{t,ε}}\bigl( (T^*\wtilde{X})^{⊗ q} \bigr) Λ ω_{t,ε}^{n-1} & = -\overline{(\sharp_{t,ε} \Ric ω_{t,ε})}^{\boxtimes q}·ω_{t,ε}^n && \text{in } \cA^{n,n}\bigl(\wtilde{X}, \End( T^*\wtilde{X}^{⊗ q} )\bigr).
  \end{align*}
  In summary, we deduce the following identity in in
  $\cA^{n,n}\bigl(\wtilde{X}, \End( \wtilde{E} )\bigr)$,
  \begin{multline*}
    n·iΘ_{\wtilde{h}_{t,ε}}\bigl(\wtilde{E}\bigr) Λ ω_{t,ε}^{n-1} = \\
    \left((\sharp_{t,ε} \Ric ω_{t,ε})^{\boxtimes p}⊗\Id_{T^*\wtilde{X}^{⊗ q}} -
      \Id_{T\wtilde{X}^{⊗ p}}⊗\overline{(\sharp_{t,ε} \Ric ω_{t,ε})}^{\boxtimes
        q}\right)·ω_{t,ε}^n.
  \end{multline*}
  Equation~\eqref{eq:cxvxc} and the definition $\tr_{ω_{t,ε}}$, cf.\
  Explanation~\ref{exp:dfgds}, thus imply \eqref{eq:E}.  This finishes the proof
  of Claim~\ref{claim:hghgf}.
\end{proof}

\subsection*{Step 6: Convergence of integrals}
\approvals{Daniel & yes \\ Henri & yes \\ Stefan & yes}

Claim~\ref{claim:hghgf} reduces the study of the error term to an analysis of
the forms $θ_ε$.  We have seen in Step~2 that $θ_ε$ decomposes as
$θ_ε = \sum_i a_i θ_{i,ε} = \sum_i a_i (β_{i,ε} + γ_{i,ε})$.  The present step
analyses the contributions to the error term that come from the $γ_{i,ε}$ and
$β_{i,ε}$, respectively.

\begin{claim}\label{claim:dgds}
  Setup and notation as in \ref{sn:fghfg}.  Given any positive number $t$ and
  any index $i$, the following integrals converge to zero,
  \begin{align}
    \label{eq:k1} 0 & = \lim_{ε → 0} \: \int_{\wtilde{X} ∖ W} \tr_{\End} \Bigl( \pr_{\wtilde{F}} \Bigl((\sharp_{t,ε} γ_{i,ε})^{\boxtimes p}⊗ \Id_{T^*{\wtilde{X}}^{⊗ q}}\bigl|_{\wtilde{F}} \Bigr) \Bigr) · ω_{t,ε}^n \\
    \label{eq:k2} 0 & = \lim_{ε → 0} \: \int_{\wtilde{X} ∖ W} \tr_{\End} \Bigl( \pr_{\wtilde{F}} \Bigl(\Id_{T_{\wtilde{X}}^{⊗ p}}⊗ \overline{(\sharp_{t,ε} γ_{i,ε})}^{\boxtimes q}\bigl|_{\wtilde{F}} \Bigr) \Bigr) · ω_{t,ε}^n
  \end{align}
\end{claim}
\begin{proof}[Proof of Claim~\ref{claim:dgds}]
  Using the special form of $γ_{i,ε}$ found in Remark~\ref{rem:hgfjgh}, there
  exists a constant $C ∈ ℝ^+$ such that
  \begin{equation}\label{eq:gamma}
    \pm γ_{i,ε} ≤ \frac{C·ε²}{|s_i|²_i+ε²}·ω_0 \qquad \text{in $𝒜^{1,1}_{ℝ}(\wtilde{X})$, for all $ε ∈ ℝ^+$}
  \end{equation}
  The operations $\sharp_{t,ε} •$, $\overline{•}$, ${•}^{\boxtimes p}$ and
  $• ⊗\Id$ preserve (semi)positivity, cf.\ Observation~\ref{obs:gh} and
  Notation~\ref{consnot:sharp}.  The following inequalities of Hermitian
  endomorphisms of $(\wtilde{E}, \wtilde{h}_{t,ε})$ will thus again hold for all
  $ε ∈ ℝ^+$,
  \begin{align}
    \label{eq:dfgsd1} \pm (\sharp_{t,ε} γ_{i,ε})^{\boxtimes p} ⊗\Id_{T^*\wtilde{X}^{⊗ q}} & ≤ \frac{C·ε²}{|s_i|²_i+ε²}·(\sharp_{t,ε} ω_0)^{\boxtimes p}⊗\Id_{T^*\wtilde{X}^{⊗ q}} \\
    \label{eq:dfgsd2} \pm \Id_{T\wtilde{X}^{⊗ p}} ⊗ \overline{(\sharp_{t,ε} γ_{i,ε})}^{\boxtimes q} & ≤ \frac{C·ε²}{|s_i|²_i+ε²}· \Id_{T\wtilde{X}^{⊗ p}}⊗ \overline{(\sharp_{t,ε} ω_0)}^{\boxtimes q}
  \end{align}
  As $\sharp_{t,ε} ω_0$ is a positive endomorphism of $T\wtilde{X}$ whose trace
  is $\tr_{ω_{t,ε}}ω_0$, cf.~Eq.~\eqref{eq:trace}, an elementary
  computation\footnote{Use Observation~\ref{obs:gh} to compute the left hand
    side.}, shows that
  \begin{equation}\label{eq:piuopu}
    (\sharp_{t,ε} ω_0)^{\boxtimes p} ≤ p·n^{p-1}·\tr_{ω_{t,ε}}(ω_0)
    \,\Id_{T\wtilde{X}^{⊗p}} \qquad\text{in
      $\End(T\wtilde{X}^{⊗ p}, \widetilde h_{t,ε}^{⊗ p})$.}
  \end{equation}
  Consequently, there exists $C' ∈ ℝ^+$ such that the following inequalities of
  Hermitian endomorphisms of $(\wtilde{F}, \wtilde{h}_{t,ε})$ will hold for all
  $ε ∈ ℝ^+$,
  \begin{align*}
    & \pm \pr_{\wtilde{F}} \Bigl( (\sharp_{t,ε} γ_{i,ε})^{ \boxtimes p}⊗ \Id_{T^*\wtilde{X}^{⊗ q}} \bigl|_{\wtilde{F}} \Bigr) \\
    & \quad ≤ \frac{C·ε²}{|s_i|²_i+ε²} · \pr_{\wtilde{F}} \Bigl( (\sharp_{t,ε} ω_0)^{ \boxtimes p}⊗ \Id_{T^*\wtilde{X}^{⊗ q}} \bigl|_{\wtilde{F}} \Bigr) && \text{by \eqref{eq:dfgsd1}}\\
    & \quad ≤ \frac{C'·ε²}{|s_i|²_i+ε²} \, · \tr_{ω_{t,ε}}(ω_0) · \Id_{\wtilde{F}}.  && \text{by \eqref{eq:piuopu}}
  \end{align*}
  Recalling the definition of $\tr_{ω_{t,ε}}(ω_0)$ from
  Explanation~\ref{exp:dfgds}, we find $C'' ∈ ℝ^+$ such that the following
  inequality of real $(n,n)$-forms holds,
  \begin{equation}\label{eq:tr}
    \pm \tr_{\End} \Bigl( \pr_{\wtilde{F}} \Bigl( (\sharp_{t,ε} γ_{i,ε})^{\boxtimes p}⊗ \Id_{T^*{\wtilde{X}}^{⊗ q}} \bigl|_{\wtilde{F}} \Bigr) \Bigr)·ω_{t,ε}^n ≤ \frac{C''·ε²}{|s_i|²+ε²}· ω_0 Λ ω_{t,ε}^{n-1}.
  \end{equation}
  From Claim~\ref{claim:lem2} and Lebesgue's dominated convergence theorem, one
  deduces the convergence of \eqref{eq:k1}.  Convergence of \eqref{eq:k2}
  follows in a similar fashion, using \eqref{eq:dfgsd2} in place of
  \eqref{eq:dfgsd1}.  Claim~\ref{claim:dgds} follows.
\end{proof}

\begin{claim}\label{claim:dgds2}
  Setup and notation as in \ref{sn:fghfg}.  Given any index $i$, the following
  integrals converge to zero,
  \begin{align}
    \label{eq:k1x} 0 & = \lim_{(ε,t) → (0,0)} \int_{\wtilde{X} ∖ W} \tr_{\End} \Bigl( \pr_{\wtilde{F}} \Bigl((\sharp_{t,ε} β_{i,ε})^{\boxtimes p}⊗ \Id_{T^*{\wtilde{X}}^{⊗ q}}\bigl|_{\wtilde{F}} \Bigr) \Bigr) · ω_{t,ε}^n \\
    \label{eq:k2x} 0 & = \lim_{(ε,t) → (0,0)} \int_{\wtilde{X} ∖ W} \tr_{\End} \Bigl( \pr_{\wtilde{F}} \Bigl(\Id_{T{\wtilde{X}}^{⊗ p}}⊗ \overline{(\sharp_{t,ε} β_{i,ε})}^{\boxtimes q}\bigl|_{\wtilde{F}} \Bigr) \Bigr) · ω_{t,ε}^n
  \end{align}
\end{claim}
\begin{proof}[Proof of Claim~\ref{claim:dgds2}]
  Using is special form, we see that $β_{i,ε}$ is a semipositive, real
  $(1,1)$-form.  Using Observation~\ref{obs:gh}, and using again that the
  operations $\sharp_{t,ε} •$, $\overline{•}$, ${•}^{\boxtimes p}$ and $• ⊗\Id$
  preserve semipositivity, we hence obtain the following inequality of real
  forms in $\cA^{n,n}(\End(\wtilde{E}), \wtilde{h}_{t,ε})$,
  \begin{equation}\label{eq:zuiz}
    \bigl( (\sharp_{t,ε} β_{i,ε})^{\boxtimes p} ⊗\Id_{T^*{\wtilde{X}}^{⊗ q}}
    \bigr)·ω_{t,ε}^n ≤ pn^p·\Id_{\wtilde{E}}·β_{i,ε} Λ ω_{t,ε}^{n-1}.
  \end{equation}
  Therefore, there exists a constant $C ∈ ℝ^+$ such that the following
  inequalities hold for all values of $t$ , $ε$,
  \begin{align*}
    0 & ≤ \int_{\wtilde{X} ∖ W} \tr_{\End} \Bigl( \pr_{\wtilde{F}} \Bigl( (\sharp_{t,ε} β_{i,ε})^{\boxtimes p}⊗ \Id_{T^*{\wtilde{X}}^{⊗ q}} \bigl|_{\wtilde{F}} \Bigr) \Bigr)·ω_{t,ε}^n && \text{semipositivity of } β_{i,ε}\\
      & ≤ C·\int_{\wtilde{X}∖ W} β_{i,ε} Λ ω_{t,ε}^{n-1} && \text{Inequality~\eqref{eq:zuiz}} \\
      & = C·\left(\int_{\wtilde{X}} (β_{i,ε}+γ_{i,ε}) Λ ω_{t,ε}^{n-1} -\int_{{\wtilde{X}}} γ_{i,ε} Λ ω_{t,ε}^{n-1} \right)\\
      & = C· \Biggl( \underbrace{\vphantom{\int_{{\wtilde{X}}}} \{θ_{i,ε}\}· \{ω_{t,ε}\}^{n-1}}_{\underset{t→0}{\lim}= 0 \text{, see Eqn.~\eqref{eq:cohomologyterm}}}- \underbrace{\int_{{\wtilde{X}}} γ_{t,ε} Λ ω^{n-1}}_{\underset{ε→0}{\lim} = 0 \text{, by Claim~\ref{claim:dgds}}} \Biggr)
  \end{align*}
  As for the first term, since $\wtilde{H} = π^*H$ is orthogonal to $D_i$,
  Equations~\eqref{eq:xbfgg} and \eqref{eq:class_of_omega} imply
  \begin{equation}\label{eq:cohomologyterm}
    \{θ_{i,ε}\}· \{ω_{t,ε}\}^{n-1} = t^{n-1} \bigl(D_i·\{ω_0\}^{n-1}\bigr), \quad\text{for all $ε$, $t$ and $i$.}
  \end{equation}
  Consequently, the term converges to $0$ when $t$ goes to $0$.
  Equation~\eqref{eq:k1x} follows.  Equation~\eqref{eq:k2x} follows in a similar
  fashion.  This ends the proof of Claim~\ref{claim:dgds2}.
\end{proof}

\subsection*{Step~7.  Proof of Claim~\ref{claim:hhg}}
\approvals{Daniel & yes \\ Henri & yes \\ Stefan & yes}

Claim~\ref{claim:hhg} now follows from Claim~\ref{claim:hghgf}, and the
convergence results of Claims~\ref{claim:dgds} and \ref{claim:dgds2}.

\subsection*{Step~8.  Proof of Theorem~\ref{generalized:holonomy}, Item~\ref{il:l1}}
\approvals{Daniel & yes \\ Henri & yes \\ Stefan & yes}

We will first prove semistability of $ℰ$ with respect to $H$.  Since the ample
divisor $H ∈ \Div(X)$ was arbitrarily chosen when we fixed the standard
Setting~\ref{setting:holonomy}, this will in fact prove semistability of $ℰ$
with respect to any ample class.  To this end, let $ℱ ⊆ ℰ$ be any reflexive
sheaf.  Using Setup and Notation~\ref{sn:fghfg}, we need to show that the number
\begin{equation}\label{eq:wervx}
  c_1(ℱ)·[H]^{n-1} \: = \: c_1(\wtilde{ℱ})·[\wtilde{H}]^{n-1} \: \overset{\eqref{eq:class_of_omega}}{=} \: \lim_{(t,ε) → (0,0)}
  c_1(\wtilde{ℱ})·\{ω_{t,c}\}^{n-1}
\end{equation}
is seminegative.  But seminegativity of the right hand side follows immediately
from Claims~\ref{claim:8.10} and \ref{claim:hhg}.  Semistability follows.

Hence, arguing by induction, to prove the existence of a parallel decomposition
$ℰ = ℰ_1 ⊕ ⋯ ⊕ ℰ_k$ whose summands are stable of slope zero with respect to any
polarisation, it suffices to show the following claim.

\begin{claim}\label{claim:cbht}
  Any saturated subsheaf $ℱ ⊆ ℰ$ of slope $μ_H(ℱ) = 0$ is a direct summand and
  the associated subbundle $F := ℱ|_{X_{\reg}}$ is parallel with respect to the
  connection on $ℰ_{|X_{\reg}}$ induced by the Chern connection of
  $(TX_{\reg},h_H)$.
\end{claim}
\begin{proof}[Proof of Claim~\ref{claim:cbht}]
  Assume that one such $ℱ$ is given.  Both sides of the
  Equation~\eqref{eq:wervx} are then zero.  Recall from
  Construction~\ref{rem:ma} that $ω_{t,ε}$ converges to $ω_{\wtilde{H}}$ in the
  $\cC^{∞}_{\loc}(\wtilde{X} ∖ D)$-topology.  Therefore, the second
  fundamental form $ρ_{t,ε}$ converges locally smoothly on
  $\wtilde{X} ∖ (D∪ W)$ to a smooth form $ρ_{\wtilde{H}}$.  Moreover, we
  get from \eqref{sff} that
  $$
  \liminf_{(t,ε) → (0,0)}\,\int_{\wtilde{X}∖ W}\tr_{\End}(-iρ_{t,ε}^*Λ
  ρ_{t,ε} Λ ω_{t,ε}^{n-1}) = 0.
  $$
  As $-iρ_{t,ε}^*Λ ρ_{t,ε} Λ ω_{t,ε}^{n-1}$ and
  $-i ρ_{\wtilde{H}}^*Λ ρ_{\wtilde{H}} Λ ω_{\wtilde{H}}^{n-1}$ are top forms
  with values in the bundle of Hermitian semipositive endomorphisms of
  $\wtilde{F}$, the Fatou lemma shows that
  $-i ρ_{\wtilde{H}}^*Λ ρ_{\wtilde{H}} Λ ω_{\wtilde{H}}^{n-1}$ and hence the
  second fundamental form vanish identically,
  \begin{equation}\label{eq:2ndfundamentalformzero}
    ρ_{\wtilde{H}} = 0 \quad \text{ on } \wtilde{X} ∖ (D∪ W).
  \end{equation}
  This has two consequences.  First, by \cite[IV.\ Prop.~14.9]{DemaillyBook2012}
  on $\wtilde{X} ∖ (D∪W)$ one has a holomorphic splitting
  $\wtilde{E} = \wtilde{F} ⊕ \wtilde{F}^{\perp}$.  One can push that splitting
  down to $X$ to obtain a holomorphic splitting $E = F ⊕ F^{\perp}$ on the big
  open subset $X_{\reg} ∖ π(W)$ of $X$.  By reflexivity, this direct sum
  decomposition extends to $X$.

  Second, as $ℱ$ is a direct summand of $ℰ$, $F$ is a subbundle of $E$, and
  hence $π(W) ∩ X_{\reg}= ∅$.  In fact, we even have $W=∅$.  Consequently,
  \eqref{eq:2ndfundamentalformzero} implies that the second fundamental form of
  $F$ in $E$ vanishes, from which parallelism follows by definition,
  cf.~\cite[V.\ Prop.~14.3]{DemaillyBook2012}.
\end{proof}

\subsection*{Step~9.  Proof of Item~\ref{il:l2}}
\approvals{Daniel & yes \\ Henri & yes \\ Stefan & yes}

Let $ℱ$ be a direct summand of $ℰ$.  Necessarily, $μ_H(ℱ) = 0$.  By
Claim~\ref{claim:cbht}, the bundle $F:= ℱ|_{X_{\reg}}$ is holomorphically
complemented and parallel.  Therefore, $F_x ⊂ E_x$ is a $G$-invariant complex
subspace by the holonomy principle.  We are left to prove that the parallel
transport of a $G$-invariant subspace of $E_x$ induces a holomorphic subbundle
of $E$ over $X_{\reg}$ that extends to $X$ as a direct summand of $ℰ$.  This
follows from \Preprint{Proposition~\ref{prop:HPB} coupled with } the observation
that $G$ is unitary, so that the orthogonal complement of a $G$-invariant
complex subspace of $E_x$ is still $G$-invariant.  \qed

%
%
\svnid{$Id: 10-BPforms.tex 758 2018-10-01 15:48:05Z guenancia $}

\section{Proof of Theorem~\ref*{thm:holonomy} (``Bochner principle for tensors'')}\label{ssec:pothf}
\subversionInfo
\approvals{Daniel & yes \\ Henri & yes \\ Stefan & yes}

First, observe that parallel transport of a $G$-invariant $ℂ$-linear tensor
$t ∈ E_x$ induces a parallel section $τ$ of $E$.  As the $(0,1)$-part of the
connection coincides with the holomorphic structure $\bar{∂}_{E}$ on $E$, we
have $\bar{∂}_{E} (τ) = 0$, so that $τ$ is holomorphic\Preprint{, see
  Reminder~\ref{remi:cc}}.  As $ℰ$ is reflexive, the corresponding coherent
analytic sheaf $ℰ^{an}$ over $X^{an}$ is likewise reflexive, see for example
\cite[Lem.~2.16]{ExtApplications}.  It follows that $τ$ yields an element of
$H⁰\bigl(X^{an},\, ℰ^{an} \bigr)$ and hence of $H⁰\bigl(X,\, ℰ \bigr)$ by GAGA.
Therefore, every $G$-invariant element of $E_x$ produces a section of $ℰ$ over
$X$.

For the converse, we need to show that the evaluation $σ_x$ of any section
$σ ∈ H⁰(X,ℰ)$ is a $G$-invariant element of $E_x$.  The proof is carried out in
two steps.

\subsubsection*{Step 1: Proof if $G$ is connected}
\approvals{Daniel & yes \\ Henri & yes \\ Stefan & yes}

Let $σ ∈ H⁰(X,ℰ)\smallsetminus \{0\}$ and let $ℱ$ be the saturation in $ℰ$ of
the trivial subsheaf $\langle σ \rangle⊂ ℰ$ generated by $σ$.  We claim that
$\codim \supp (ℱ/\langle σ \rangle ) ≥ 2$.  Indeed, otherwise one would have
$μ_H(ℱ)>μ_H(\langle σ \rangle)=0$, which would contradict the semistability of
$ℰ$ with respect to $H^{n-1}$ proved in Item \ref{il:l1} of
Theorem~\ref{generalized:holonomy}.  From this, we conclude that
$\langle σ \rangle$ coincides with $ℱ$, as both sheaves are reflexive and agree
on a big open subset of $X$.  In other words, $\langle σ \rangle$ is saturated,
and hence a direct summand of $ℰ$, cf.\ Claim~\vref{claim:cbht} for detailed
arguments of this.  It therefore follows from Item~\ref{il:l2} of
Theorem~\ref{generalized:holonomy} that $σ_x$ generates a $G$-invariant complex
line in $E_x$.

As $G$ is connected, it follows from Proposition~\ref{prop:factors} that $G$ is
a product of $\SU$'s and $\Sp$'s.  In particular, $G$ is semisimple and
therefore equal to its own commutator subgroup, \cite[Thm.~23.2]{MR2062813}.  It
follows that every homomorphism $χ: G → ℂ^*$ is trivial.  As a result, every
$G$-invariant line in $E_x$ has to be point-wise fixed, and hence $σ_x ∈ E_x$ is
fixed by $G$, as claimed.

\subsubsection*{Step 2: Proof in general}
\approvals{Daniel & yes \\ Henri & yes \\ Stefan & yes}

Theorem~\ref{thm:holonomyCover} provides us with a holonomy cover, that is, a
quasi-étale morphism $γ: Y → X$ such that $\Hol(Y_{\reg}, g_{H_Y})$ is
connected, for $H_Y := γ^* H$.  Set
$$
ℰ_Y := \bigl(𝒯_Y^{⊗ p}⊗ (𝒯_Y^*)^{⊗ q}\bigr)^{**}
$$
and take $σ ∈ H⁰ \bigl(X,\, ℰ \bigr)$.  Let $Y° := γ^{-1}(X_{\reg})$.  Then,
$γ|_{Y°} : Y° → X_{\reg}$ is a locally biholomorphic map between complex
manifolds, so there is a well-defined pull back tensor
$(γ|_{Y°})^* (σ|_{X_{\reg}}) ∈ H⁰ \bigl( Y°,\, ℰ_Y \bigr)$ which extends to a
section $σ_Y$ of $ℰ_Y$ on the whole of $Y$.  By Step~1 and the holonomy
principle, $\wtilde{σ}|_{Y_{\reg}}$ is parallel with respect to
$g_{H_Y}$\Preprint{, see Reminder~\ref{remi:cc}}.  It follows from the universal
property of the EGZ construction, Proposition~\ref{prop:univ}, that on $Y°$ we
have $γ^*ω_H = ω_{H_Y}$, which induces the analogous equality for the associated
Riemannian metrics.  The claim follows from the holonomy principle together with
the observation that vanishing of covariant derivatives and hence parallelism is
a local property.  This concludes the proof of Theorem~\ref{thm:holonomy}.  \qed

%
%
\svnid{$Id: 11-augmentedRegularityrevisited.tex 756 2018-09-30 09:41:37Z kebekus $}

\section{Augmented irregularity revisited}
\subversionInfo
\approvals{Daniel & yes \\ Henri & yes \\ Stefan & yes}

Combining our findings on covering constructions and Bochner principles, we
obtain two new characterisations of varieties with (non-)vanishing augmented
regularity, which do not rely on the computation of invariants on quasi-étale
covers, but only on invariants of the variety under investigation.
Additionally, we use the Bochner principle to provide two criteria for detecting
finite quotients of Abelian varieties.

\begin{thm}[Augmented regularity and symmetric differentials]\label{thm:augm_irreg_rev}
  In the Standard setting~\ref{setting:holonomy}, the following are equivalent.
  \begin{enumerate}
  \item\label{il:irreg} The augmented irregularity does not vanish:
    $\wtilde{q}(X) ≠ 0$.
  \item\label{il:hol} The restricted holonomy leaves a non-zero vector of $V$
    invariant: $V_0 ≠ \{0\}$.
  \item\label{il:symm} There exists a non-trivial symmetric differential on $X$.
    In other words, there exists $m ∈ ℕ^+$ such that
    $h⁰ \bigl(X,\, \Sym^{[m]} Ω_X¹ \bigr) ≠ 0$.
 \end{enumerate}
\end{thm}

\begin{rem}
  In the smooth case, similar results were proven by Kobayashi in
  \cite[Thm.~6]{MR567422}.  Following the argumentation of \cite[Thm.~6 and
  7]{MR567422}, our methods even give an upper bound for the number of symmetric
  differentials,
  $h⁰\bigl(X, \, \Sym^{[m]} Ω_X¹ \bigr) ≤ \binom{m+\wtilde{q}(X)-1}{m}$ for all
  $m > 0$.
\end{rem}

\begin{rem}
  We emphasise that the equivalence ``\ref{il:irreg} $⇔$
  \ref{il:symm}'' gives a purely algebro-geometric characterisation of
  non-vanishing augmented regularity in terms of invariants of $X$ alone.  This
  underlines yet again the importance of this concept in the structure theory of
  klt varieties with numerically trivial canonical divisor.
\end{rem}

As a corollary we obtain the following vanishing theorem, which generalises
\cite[Thm.~2.1(2)]{MR1603624} to our setup.

\begin{cor}[Detecting finite quotients of Abelian varieties, II]\label{qab1}
  Let $X$ be a klt variety with numerically trivial canonical divisor.  Assume
  that $𝒯_X$ is stable with respect to some ample $ℚ$-divisor.  Then,
  $$
  h⁰ \bigl(X,\, \Sym^{[m]} Ω_X¹ \bigr) = 0 \qquad \text{for all } m ∈ ℕ^+,
  $$
  unless $X$ is of the form $A/G$ where $A$ is an Abelian variety and $G$ is a
  finite group whose action on $A$ is free in codimension one.
\end{cor}
\begin{proof}
  If $h⁰ \bigl(X,\, \Sym^{[m]} Ω_X¹ \bigr) ≠ \{0\}$, it follows from
  Theorem~\ref{thm:augm_irreg_rev} that $V_0 ≠ \{0\}$.  As $𝒯_X$ is stable,
  using Corollary~\ref{cor:sirr0} we see that the representation of the full
  holonomy group $G$ on $V$ is irreducible.  Looking at
  Observation~\ref{obs:ccd} we conclude that $V_0 = W_0 = V$.  Therefore,
  Theorem~\ref{thm:holonomyCover} implies that there exists a quasi-étale cover
  $γ': A' → X$, where $A'$ is an Abelian variety.  By taking Galois closure of
  $γ'$, \cite[App.~B in the preprint version]{GKP13}, and observing that finite
  étale covers of Abelian varieties are Abelian varieties themselves, we
  conclude that there exists an Abelian variety $A$ together with a quasi-étale
  \emph{Galois} cover $γ: A → X$.
\end{proof}

\begin{rem}
  Setup as in Corollary~\ref{qab1}.  If non-zero symmetric differentials exist
  on $X$, the cotangent sheaf is flat with finite monodromy and stable, but not
  strongly stable.  The group $G$ is isomorphic to the holonomy group of a
  singular Kähler-Einstein metric on $X$, and the holonomy representation of $G$
  on the fibre over some smooth point of $X$ is irreducible.  Smooth examples
  exhibiting this behaviour can be found in \cite{MR1868164}.
\end{rem}

\subsection{Preparation for the proof of Theorem~\ref*{thm:augm_irreg_rev}}
\approvals{Daniel & yes \\ Henri & yes \\ Stefan & yes}

We note the following two simple representation-theoretic lemmata\Publication{\
  whose proof can be found in the preprint version of this paper,
  \href{https://arxiv.org/abs/1704.01408}{arXiv:1704.01408}}.

\begin{lem}\label{lem:no_inv_polynomials}
  Let $n ≥ 2$.  Let $G=\SU(n)$ or let $n$ be even and
  $G=\Sp\!\left(\frac n2\right)$.  Let $W$ be the complex standard
  representation of $G$.  Then, $\bigl( \Sym^m W^* \bigr)^G = \{0\}$ for all
  $m ≥ 0$.  \Publication{\hfill \qed}
\end{lem}
\Preprint{%
  \begin{proof}
    We have $\bigl( \Sym^m W^* \bigr) ≅ ℂ[W]_{m}$ as $G^ℂ$-representations.
    Moreover, $W$ does not admit any non-constant $G^ℂ$-invariant (homogeneous)
    polynomial, as in both cases $G^ℂ = \SL(n, ℂ)$ or $\Sp(ℂ^{n}, ω_{std})$,
    respectively, has an open orbit in $W$.  Consequently,
    $\bigl( \Sym^m (W^* \bigr)^G ≅ ℂ[W]_{m}^G = ℂ[W]_{m}^{G^ℂ} = \{0\}$, as
    claimed.
  \end{proof}}

\begin{lem}\label{lem:finite_inv_symm}
  Let $Γ$ be a finite group of order $m := \# Γ$ and let $V ≠ \{0\}$ be a
  finite-dimensional complex $Γ$-representation.  Then,
  $\bigl( \Sym^m V \bigr)^Γ ≠ \{0\}$.  \Publication{\hfill \qed}
\end{lem}
\Preprint{%
  \begin{proof}
    Let $\{ e=γ_0, γ_1, …, γ_{m-1} \}$ be the elements of $Γ$ and let
    $π: V^{⊗ m} → \bigl( \Sym^m V\bigr)$ be the $Γ$-equivariant projection map,
    $v_0 ⊗ ⋯ ⊗ v_{m-1} ↦ v_0 \odot ⋯ \odot v_{m-1}$.  Pick any non-zero vector
    $v ∈ V$.  Then,
    $0 ≠ v \odot γ_1(v) \odot ⋯ \odot γ_{m-1}(v) ∈ \bigl( \Sym^mV \bigr)^Γ$.
  \end{proof}}

\subsection{Proof of Theorem~\ref*{thm:augm_irreg_rev}}

---

\subsubsection*{Equivalence \ref{il:irreg} $⇔$ \ref{il:hol}}
\approvals{Daniel & yes \\ Henri & yes \\ Stefan & yes}
  
This follows from Item~\ref{il:FDC2} of Corollary~\ref{cor:flatDecomp}.

\subsubsection*{Implication \ref{il:hol} $⇒$ \ref{il:symm}}
\approvals{Daniel & yes \\ Henri & yes \\ Stefan & yes}

The action of $G$ on $V_0$ is given by the representation of a finite group $Γ$
on $V_0$, see Item~\ref{il:FDC1} of Corollary~\ref{cor:flatDecomp}.  Set
$m := \#Γ$.  Hence, by Lemma~\ref{lem:finite_inv_symm}, there exists a
$G$-invariant non-zero vector in $(\Sym^m V_0^*) ⊂ (\Sym^{m} V^*)$.  The Bochner
principle for reflexive tensors, Theorem~\ref{thm:holonomy}, then implies that
$H⁰ \bigl(X,\, \Sym^{[m]} Ω¹_X \bigr) ≠ \{0\}$, as claimed.

\subsubsection*{Implication $\neg$\ref{il:hol} $⇒$ $\neg$\ref{il:symm}}
\approvals{Daniel & yes \\ Henri & yes \\ Stefan & yes}

We suppose that $V_0 = \{0\}$ and let $σ ∈ H⁰ \bigl(X,\, \Sym^{[p]} Ω¹_X \bigr)$
for some $p > 0$.  We aim to show that $σ$ vanishes identically.  The Bochner
principle for tensors, Theorem~\ref{thm:holonomy}, implies that
$σ_x ∈ (\Sym^p V^*)^{G°}$.  As the action of $G°$ on $V^*$ is totally
decomposed, cf.\ Construction~\ref{constr:canonDecompRestr}, and as
$V_0 = \{0\}$, the standard decomposition of the symmetric product of a direct
sum of representations, \cite[Chap.~II, (3.1)]{MR781344}, yields
$$
\bigl( \Sym^p V^* \bigr)^{G°} = \bigoplus_{\substack{k_1, …, k_m ∈ ℕ \\ \sum k_j =
    p}} \bigl( \Sym^{k_1} V_1^* \bigr)^{G°_1} ⊗ ⋯ ⊗ \bigl( \Sym^{k_m} V_m^*
\bigr)^{G°_m}.
$$
Set $n_i := \dim V_i$ and recall from Proposition~\ref{prop:factors} for each
$i = 1, …, m$, either $G_i° ≅ \SU(n_i)$, or $n_i$ is even and
$G_i° ≅ \Sp\!\left(\frac {n_i}2\right)$.  In either case, observe that the
action $G°_i ↺ V_i^*$ is isomorphic to the dual of the standard action of the
respective group.  Lemma~\ref{lem:no_inv_polynomials} hence implies that
$σ_x ∈ (\Sym^p V^*)^{G°} = \{0\}$.  We conclude that $σ = 0$, as desired.  \qed

\part{Varieties with strongly stable tangent sheaf}\label{part:IV}
%
%
\svnid{$Id: 12-dichotomy.tex 758 2018-10-01 15:48:05Z guenancia $}

\section{The basic dichotomy: CY and IHS}\label{sec:bDCYIHS}
\subversionInfo
\approvals{Daniel & yes \\ Henri & yes \\ Stefan & yes}

If $X$ is a smooth, simply connected, irreducible compact Kähler manifold with
trivial first Chern class, then $X$ is either an irreducible Calabi-Yau manifold
or an irreducible holomorphic symplectic variety, where these two classes are
distinguished by the algebra of holomorphic forms.  The goal of this section is
to show that after passing to a quasi-étale cover, any projective klt variety
with numerically trivial canonical divisor and strongly stable tangent sheaf
falls into one of the two classes introduced in Definition~\ref{def:CY_and_IHS}.
We will also give the proof of Proposition~\ref{prop:intrinsic} and relate our
discussion to \emph{algebraic holonomy}, a concept introduced by Balaji and
Kollár in \cite{BalajiKollar08}.

\subsection{Differential forms on varieties with strongly stable tangent sheaf}\label{sec:diff}
\approvals{Daniel & yes \\ Henri & yes \\ Stefan & yes}

From the irreducible case of the general results obtained in Part~\ref{part:II},
we obtain the following description of varieties with strongly stable tangent
sheaf.

\begin{thm}[Holonomy dichotomy for strongly stable varieties]\label{thm:dichotomy-A}
  Assume the standard Setting~\ref{setting:holonomy}.  Then, the sheaf $𝒯_X$ is
  strongly stable if and only if the restricted holonomy group is one of the
  following two groups and the action of $G°$ on $V$ is the standard action of
  the respective group.
  \begin{enumerate}
  \item\label{il:SU} The group $G°$ is isomorphic to $\SU(n)$.
  \item\label{il:Sp} The dimension $n$ is even, and the group $G°$ is isomorphic
    to $\Sp\!\left(\frac n2\right)$.
  \end{enumerate}
  In either case, there exists a quasi-étale cover $γ : Y→ X$ such that
  restricted holonomy and holonomy agree on $Y$.  More precisely, using
  Notation~\ref{not:qec} we have $G° = G_Y° =G_Y$.
\end{thm}
\begin{proof}
  Recall from Corollary~\ref{cor:sirr2} that $𝒯_X$ is strongly stable if and
  only if the restricted holonomy representation $G° ↺ V$ is irreducible.
  Proposition~\ref{prop:factors} (``Classification of restricted holonomy'')
  then yields the claimed dichotomy.  The last claim is just the existence of
  the holonomy cover, Theorem~\ref{thm:holonomyCover}.
\end{proof}

\begin{thm}[Reflexive differentials on strongly stable varieties]\label{thm:dichotomy-B}
  In the standard Setting~\ref{setting:holonomy}, assume that the sheaf $𝒯_X$ is
  strongly stable.  Then, the spaces of holomorphic $p$-forms can be controlled
  as follows.
  \begin{enumerate}
  \item\label{il:starsky} If $G°$ is isomorphic to $\SU(n)$, then
    $$
    h⁰ \bigl( X,\, Ω^{[p]}_X \bigr) ≤
    \begin{cases}
      1 & \text{if $p = 0$ or $p = n$\hphantom{and if $p$ is e}} \\
      0 & \text{otherwise.}
    \end{cases}
    $$
  \item\label{il:hutch} If the dimension $n$ is even and $G°$ is isomorphic to
    $\Sp\!\left(\frac n2\right)$, then
    $$
    h⁰ \bigl( X,\, Ω^{[p]}_X \bigr) ≤
    \begin{cases}
      1 & \text{if $0 ≤ p ≤ n$ and if $p$ is even} \\
      0 & \text{otherwise.}
    \end{cases}
    $$
  \end{enumerate}
  If $G = G°$, then the inequalities are in fact equalities.
\end{thm}

\begin{rem}
  If $G°$ is isomorphic to $\SU(n)$ and if $K_X$ is linearly trivial, then we
  already have $G=G°$.  In fact, if $K_X \sim 0$, then there exists a non-zero
  holomorphic top-form on $X_{\reg}$, hence by the Bochner principle,
  Theorem~\ref{thm:holonomy}, we have $G ⊂ \SU(n)$.  But then $G = G° = \SU(n)$.
\end{rem}

\begin{proof}[Proof of Theorem~\ref{thm:dichotomy-B}]
  We handle both cases simultaneously.  Let $γ: Y → X$ be a quasi-étale cover
  such that holonomy and restricted holonomy of $Y_{\reg}$ agree, as recalled in
  Theorem~\ref{thm:dichotomy-A}.  The tangent sheaf $𝒯_Y$ is then likewise
  strongly stable, and the restricted holonomies of $Y_{\reg}$ and $X_{\reg}$
  agree.  We also observe that the reflexive pullback morphisms
  $γ^{[*]}: H⁰\bigl( X,\, Ω^{[p]}_X \bigr) → H⁰\bigl( Y,\, Ω^{[p]}_Y \bigr)$ are
  injective for all $p$.  In order to establish all claims made, it therefore
  remains to show that equality holds in Inequalities~\ref{il:starsky},
  \ref{il:hutch} for $h⁰ \bigl( Y,\, Ω^{[p]}_Y \bigr)$.  The Bochner principle
  for tensors, Theorem~\ref{thm:holonomy}, applies to show that the natural
  evaluation map establishes a linear isomorphism
  $$
  H⁰ \bigl( Y,\, Ω^{[p]}_Y \bigr) \overset{≅}{\longrightarrow} \bigl(
  \bigwedge\nolimits^p V^*\bigr)^{G_Y}.
  $$
  In addition, we have \Preprint{seen in Reminder~\ref{remi:cc}} that
  $(Λ^pV^*)^{G_Y}≅ \overline{(Λ^pV)^{G_Y}}$.  The desired equalities hence
  follow from classical invariant theory and representation theory for the
  groups $\SL_{ℂ}(V)= \SU(n)^{ℂ}$ and
  $\Sp_{ℂ}(V) = \Sp\!\left(\frac n2\right)^{ℂ}$\Publication{.}\Preprint{: by
    \cite[Thm.~5.5.11]{MR2522486} the non-trivial $\SL_{ℂ}(V)$-representations
    $\bigwedge^p V$ are all irreducible and therefore do not contain non-zero
    invariant vectors; the computation of $\Sp_{ℂ}(V)$-invariants in
    $\bigwedge^p V$ is given in \cite[Thm.~5.3.3]{MR2522486}.}
\end{proof}

\begin{defn}[Holomorphic symplectic form]
  Let $X$ be a normal variety.  A reflexive differential two-form
  $σ ∈ H⁰\bigl( X,\, Ω^{[2]}_X \bigr)$ on $X$ is called \emph{holomorphic
    symplectic} if
  \begin{enumerate}
  \item $σ|_{X_{\reg}}$ is everywhere non-degenerate,
  \item $σ|_{X_{\reg}}$ is closed: $d(σ|_{X_{\reg}}) = 0$, and
  \item $σ|_{X_{\reg}}$ extends regularly to any resolution of singularities of
    $X$.
  \end{enumerate}
\end{defn}

\begin{lem}[Two-forms are holomorphic symplectic]\label{lem:symplectic}
  In the standard Setting~\ref{setting:holonomy}, assume that the sheaf $𝒯_X$ is
  strongly stable.  If there exists a non-vanishing form
  $0 ≠ σ ∈ H⁰ \bigl( X,\, Ω^{[2]}_X \bigr)$, then $σ$ is holomorphic symplectic,
  and any other reflexive differential form on $X$ is a constant multiple of the
  appropriate wedge power $σ Λ ⋯ Λ σ$ of $σ$.
\end{lem}
\begin{proof}
  The existence of $0 ≠ σ ∈ H⁰ \bigl( X,\, Ω^{[2]}_X \bigr)$ implies that we are
  in case \ref{il:Sp} of Theorem~\ref{thm:dichotomy-A}.  As $X$ is assumed to be
  projective and klt, the restriction of $σ$ to $X_{\reg}$ extends to any
  resolution of $X$ and is therefore automatically closed, see
  \cite[Prop.~1.4]{ExtApplications}.  The assertion that $σ|_{X_{\reg}}$ is
  everywhere non-degenerate has been shown in \cite[Cor.~8.10]{GKP16}.  Together
  with Item~\ref{il:hutch} of Theorem~\ref{thm:dichotomy-B} we conclude that
  every reflexive differential form on $X$ is a constant multiple of a wedge
  power of $σ$.
\end{proof}

Combine Corollary~\ref{cor:sirr2}, Theorem~\ref{thm:dichotomy-B}, and
Lemma~\ref{lem:symplectic} to obtain the following.

\begin{cor}
  In the standard Setting~\ref{setting:holonomy}, if $n$ is even and
  $G ≅ \Sp\!\left(\frac n2\right)$, then $X$ carries a holomorphic
  symplectic form $σ$ with the following property: if $γ:Y→ X$ is any
  quasi-étale cover, we have an isomorphism of algebras
  $$
  \bigoplus_{p=0}^{n} H⁰\bigl(Y,\,
  Ω_Y^{[p]}\bigr) = ℂ[γ^{[*]}σ].  \eqno\qed
  $$
\end{cor}

\subsection{Calabi-Yau and irreducible holomorphic symplectic varieties}\label{subsect:CY_and_IHS_cases}
\approvals{Daniel & yes \\ Henri & yes \\ Stefan & yes}

One can reformulate the results obtained in the previous subsection using the
terminology introduced in Definition~\ref{def:CY_and_IHS} as follows.

\begin{cor}[Dichotomy for varieties with strongly stable tangent sheaf]\label{cor:dichotomy_CY_IHS}
  In the standard Setting~\ref{setting:holonomy}, assume that $𝒯_X$ is strongly
  stable.  Then, one of the following cases occurs.
  \begin{enumerate}
  \item\label{il:dicho1} The restricted holonomy group is equal to $\SU(n)$, and
    if $γ: Y → X$ denotes a quasi-étale cover making $K_Y \sim 0$, then $Y$ is
    Calabi-Yau.
  \item\label{il:dicho2} The dimension of $X$ is even, the restricted holonomy
    group is equal to $\Sp\!\left(\frac n2\right)$, and if $γ: Y → X$ denotes a
    quasi-étale cover making $G° = G_Y° = G_Y$, then $Y$ is irreducible
    holomorphic symplectic.
  \end{enumerate}
\end{cor}

\begin{rem}
  If $n =2$, the definition of CY and IHS varieties coincide.  However, if
  $n ≥ 3$, then \ref{il:dicho1} and \ref{il:dicho2} are mutually exclusive.  The
  tangent sheaf of a CY or IHS variety is strongly stable by
  \cite[Prop.~8.20]{GKP16}.
\end{rem}

\begin{rem}[Varieties with linearly trivial canonical divisor and strongly stable tangent sheaf]
  Corollary~\ref{cor:dichotomy_CY_IHS} implies that a normal projective variety
  with at worst canonical singularities, \emph{linearly trivial canonical
    divisor}, and strongly stable tangent sheaf is either Calabi-Yau, or admits
  a finite, quasi-étale cover that is an irreducible holomorphic symplectic
  variety; Example~\ref{ex:no2form} describes a variety with linearly trivial
  canonical divisor and no two-form that admits a quasi-étale cover that is IHS.
  This shows that in Item~\ref{il:dicho2} above taking a quasi-étale cover in
  general cannot be avoided.  The reader is encouraged to compare this
  observation with the smooth situation, see \cite[Rem.~8.22]{GKP16}.
\end{rem}

\subsection{Characterisation of IHS and CY varieties in terms of holonomy}\label{prop:page}
\approvals{Daniel & yes \\ Henri & yes \\ Stefan & yes}

We recall that the definition of CY and IHS varieties,
Definition~\ref{def:CY_and_IHS}, is formulated in purely algebro-geometric
terms.  We are now in a position to give a complementary characterisation of
these two types of varieties purely in terms of differential-geometric holonomy,
as formulated in Proposition~\ref{prop:intrinsic} and in complete accordance
with the smooth theory.

\begin{prop}[Characterisation of CY and IHS varieties in terms of holonomy]\label{prop:intrinsicn}
  In the standard Setting~\ref{setting:holonomy}, the following conditions are
  equivalent.
  \begin{enumerate}
  \item\label{iln:intr1} $X$ is a Calabi-Yau variety.
  \item\label{iln:intr2} $\Hol(X_{\reg},g_H)$ is connected and
    $H⁰ \bigl( X, Ω_{ X}^{[p]} \bigr) = \{0\}$ for all $0 < p < n$.
  \item\label{iln:intr3} $\Hol(X_{\reg},g_H)$ is isomorphic to $\SU(n)$.
  \end{enumerate}
  Analogously, the following conditions are equivalent.
  \begin{enumerate}
  \item\label{iln:intr4} $X$ is an irreducible holomorphic symplectic variety.
  \item\label{iln:intr5} $\Hol(X_{\reg},g_H)$ is connected, and there exists a
    holomorphic symplectic two-form $σ ∈ H⁰ \bigl(X, Ω_X^{[2]} \bigr)$ such that
    $⊕_{p=0}^n H⁰(X,Ω_X^{[p]})=ℂ[σ]$.
  \item\label{iln:intr6} $\Hol(X_{\reg},g_H)$ is isomorphic to
    $\Sp\!\left(\frac n2\right)$.
  \end{enumerate}
\end{prop}
\begin{proof}
  Notice first that any of the conditions \ref{iln:intr1}--\ref{iln:intr6}
  implies that $K_X$ is linearly equivalent to zero, and in particular that $X$
  has canonical singularities.  This is clear for \ref{iln:intr1},
  \ref{iln:intr4} and \ref{iln:intr5}.  For \ref{iln:intr3} and \ref{iln:intr6},
  this is a consequence of the Bochner principle for reflexive forms,
  Theorem~\ref{thm:holonomy}.  Finally, for \ref{iln:intr2} this follows from
  Proposition~\ref{prop:factors} and the Bochner principle for reflexive forms.

  So in any case, there exists a nowhere vanishing $n$-form on $X$.  The Bochner
  principle hence implies that $G ⊆ \SU(n)$.  Using the notation of
  Proposition~\ref{prop:factors} (``Classification of restricted holonomy''), we
  obtain that $G° = \bigtimes_{i=1}^m G_i°$ acts as a product, and if
  $n_i = \dim V_i$ then either $G_i° = \SU(n_i)$ or $n_i$ is even and
  $G_i° = \Sp\!\left(\frac{n_i}2\right)$.

  \subsubsection*{Implication \ref{iln:intr1} $⇒$ \ref{iln:intr3}.}

  We know that $G° ⊆ G ⊆ \SU(n)$ and we need to prove that equality holds in
  both steps.  Using Theorem~\ref{thm:holonomyCover}, one can find a quasi-étale
  holonomy cover $γ : Y → X$ such that $G_Y = \bigtimes_{i=1}^m G_i°$.  By the
  Bochner principle for reflexive forms, this yields $m$ independent reflexive
  holomorphic forms of positive degree on $Y$.  Given the restrictions on the
  algebra of reflexive differential forms on $Y$ dictated by the CY condition,
  one gets successively that $m = 1$ and $G° = \SU(n)$.  The conclusion follows.

  \subsubsection*{Implication \ref{iln:intr3} $⇒$ \ref{iln:intr2}}
  
  This is a direct application of the Bochner principle for reflexive forms.

  \subsubsection*{Implication \ref{iln:intr2} $⇒$ \ref{iln:intr1}.}
  
  We have $G°=G ⊆ \SU(n)$.  Using Proposition~\ref{prop:factors}, one gets that
  for every $1≤ i ≤ m$, $G_i=G_i°$ is either $\SU(n_i)$ or
  $\Sp\!\left(\frac{n_i}2\right)$, where $n_i=\dim V_i$.  By the Bochner
  principle for reflexive forms, we have that $h⁰(X,Ω_X^{[n_i]}) ≥ 1$, which
  implies that for all $1≤ i≤ m$, one either has $n_i=0$ or $n_i=n$.  In
  particular, $m=1$.  Moreover, given the restrictions on the algebra of
  reflexive differential forms on $X$ dictated by the CY condition, $G$ cannot
  be the unitary symplectic group; it follows that $G=\SU(n)$.  Now, if $γ:Y→ X$
  is any quasi-étale cover, then $G°_Y=G°=\SU(n)$, and hence the Bochner
  principle implies that $X$ is CY.

  \Preprint{%
    \subsubsection*{Implication \ref{iln:intr4} $⇒$ \ref{iln:intr6}.}
    
    By the Bochner principle, $G°⊆ G ⊆ \Sp\!\left(\frac n2\right)$ and we need
    to prove the previous inclusions are equalities.  A short computation shows
    that all $n_i$'s are even and $G_i°⊆ \Sp(n_i/2)$.  Using
    Theorem~\ref{thm:holonomyCover}, one can then find a quasi-étale holonomy
    cover $γ : Y → X$ such that
    $G_Y = \bigtimes_{i=1}^m \Sp\!\left(\frac {n_i}2\right)$.  By the Bochner
    principle, this yields $m$ independent reflexive holomorphic two-forms on
    $Y$.  As $X$ is IHS, this implies that $m=1$ and
    $G°= \Sp\!\left(\frac n2\right)$.

    \subsubsection*{Implication \ref{iln:intr6} $⇒$ \ref{iln:intr5}.}

    Again, this is a direct application of the Bochner principle for reflexive
    forms.

    \subsubsection*{Implication \ref{iln:intr5} $⇒$ \ref{iln:intr4}.}

    First, the Bochner principle shows that one has
    $G° = G ⊂ \Sp\!\left(\frac n2\right)$.  As in the Implication
    ``\ref{iln:intr2} $⇒$ \ref{iln:intr1}'' above, one finds that the $n_i$'s
    are even and that $G_i = G°_i = \Sp\!\left(\frac {n_i}2\right)$.  By the
    Bochner principle again, $h⁰(X,Ω_X^{[2]})= m$.  From this it follows that
    $m = 1$, that is, $G= \Sp\!\left(\frac n2\right)$.  Now, if $γ:Y→ X$ is any
    quasi-étale cover, then $G°_Y=G°=\Sp\!\left(\frac n2\right)$ and the Bochner
    principle shows that $X$ is IHS.

    This concludes the proof of Proposition~\ref{prop:intrinsic}.}%
  \Publication{The second set of equivalences is proven analogously.  The
    preprint version of this paper,
    \href{https://arxiv.org/abs/1704.01408}{arXiv:1704.01408}, spells out all
    details.}
\end{proof}

\subsection{Characterisation in terms of algebraic holonomy groups}\label{sec:alghol}
\approvals{Daniel & yes \\ Henri & yes \\ Stefan & yes}
  
In this section, we characterise the two cases of the dichotomy in terms of
stability properties of powers of the (co)tangent sheaf.

\subsubsection{Holonomy groups of stable bundles}
\approvals{Daniel & yes \\ Henri & yes \\ Stefan & yes}

We start by giving a quick introduction to the theory of algebraic holonomy
groups of stable reflexive sheaves, as developed in \cite{BalajiKollar08}.  We
follow \cite[Sect.~6.19]{Dru16} and refer the reader to these two sources for
references of the classical results mentioned below.
  
\begin{thmdef}[Algebraic holonomy]\label{defn:alg_holonomy}
  Let $X$ be a normal projective variety, and let $ℰ$ be a reflexive sheaf on
  $X$, locally free away from a small subset $B ⊂ X$.  Suppose that $ℰ$ is
  stable and of slope $μ_H(ℰ) = 0$ with respect to an ample Cartier divisor $H$
  on $X$.  Let $x ∈ X_{\reg} ∖ B$ and let $ℰ_x$ be the fibre of $ℰ$ over
  $x$.

  Then, there exists a unique smallest subgroup $H_x(ℰ) ⊂ \GL_{ℂ}(ℰ_x)$, called
  \emph{algebraic holonomy group of $ℰ$ at $x$}, such that the following holds.
  For every smooth, pointed, projective curve $(D, y)$ and every pointed
  morphism $g: (D, y) → (X, x)$ where $ℰ$ is locally free along $g(D)$ and where
  $g^*(ℰ)$ is polystable (and hence unitary flat by a theorem of Narasimhan and
  Seshadri), the image of the resulting representation of
  $π_1(D, y) → \GL_ℂ((g^* ℰ)_y) = \GL_ℂ(ℰ_x)$ is contained in $H_x(ℰ)$.  \qed
\end{thmdef}
  
We emphasise that in contrast to the discussion of differential-geometric
holonomy groups in previous parts of the paper, the above construction is
algebraic.

\begin{rem}
  \Publication{We refer the reader to \cite[Thm.~20]{BalajiKollar08} for a
    further characterisation of algebraic holonomy groups.  More explanations
    regarding this result are given in the preprint version of this
    paper.}\Preprint{Refining the restriction theorem of Mehta and Ramanathan,
    Graf has shown the following in \cite[Thm.~1.1]{MR3520712}.  If $m ≫ 0$ is
    sufficiently divisible and $C ⊂ X_{\reg} ∖ B$ is a curve containing
    $x$ that is obtained as a complete intersection of sufficiently general
    elements of $|mH|$ passing through $x$, the restriction $ℰ|_C$ is a stable
    vector bundle on $C$.  Again by the theorem of Narasimhan-Seshadri, $ℰ|_C$
    corresponds to a unique unitary representation $ρ: π_1(C, x) → U(ℰ_x)$ with
    respect to some Hermitian form on $ℰ_x$.  Balaji and Kollár prove in
    \cite[Thm.~20]{BalajiKollar08} that if $m ≫ 0$ is sufficiently large and if
    $C$ is a sufficiently general complete intersection curve through
    $x ∈ X_{\reg} ∖ B$, then the algebraic holonomy $H_x(ℰ)$ equals the
    Zariski-closure of $ρ\bigl(π_1(C, x) \bigr)$ in $\GL_ℂ(ℰ_x)$.  Moreover,
    $H_x(ℰ)$ is the smallest algebraic subgroup such that for every curve
    $x ∈ C ⊂ X$ such that $ℰ|_C$ is stable, the image of the Narasimhan-Seshadri
    representation is contained in $H_x(ℰ)$.  In particular, $H_x(ℰ)$ is
    independent of the ample divisor $H$.}
\end{rem}
    
\begin{rem}
  Connectivity properties of the algebraic holonomy groups are closely connected
  to the question whether a given stable sheaf of degree zero is actually
  strongly stable, see \cite[Lem.~6.22]{Dru16}.
\end{rem}
  
The following Bochner principle for the algebraic holonomy group provides a link
to differential-geometric holonomy groups.
  
\begin{prop}[\protect{Bochner principle for algebraic holonomy, \cite[Thm.~20(3)]{BalajiKollar08}}]\label{prop:algebraicholonomyprinciple}
  Setup as in Definition~\ref{defn:alg_holonomy}.  Then, for every $m, n ∈ ℕ$,
  the evaluation map gives a one-to-one correspondence between direct summands
  of the reflexive tensor product $\bigl(ℰ^{⊗ m} ⊗ (ℰ^*)^{⊗ n}\bigr)^{**}$ and
  $H_x(ℰ)$-invariant subspaces of $ℰ_x^{⊗ m} ⊗ (ℰ_x^*)^{⊗ n}$.  \qed
\end{prop}
  
Observing that Tannakian duality and the knowledge of a small number of
representations determines a reductive group completely,
\cite[Sect.~4]{BalajiKollar08}, one obtains the following result.

\begin{prop}\label{prop:two_cases_for_algebraic_holonomy}
  Setup as in Definition~\ref{defn:alg_holonomy}.  Assume additionally that
  $\det ℰ_x ≅ 𝒪_X$ and that $H_x(ℰ)$ is connected.  Then, the following are
  equivalent.
  \begin{enumerate}
  \item For some (and a posteriori all) $m ≥ 2$, the $m$-th reflexive symmetric
    power $\Sym^{[m]} (ℰ) := \Sym^m(ℰ)^{**}$ is indecomposable.
  \item The algebraic holonomy $H_x(ℰ)$ is either $H_x(ℰ) ≅ \SL_{ℂ}(ℰ_x)$,
    or $H_x(ℰ) ≅ \Sp_{ℂ}(ℰ_x)$ for a suitable complex-symplectic form on
    $ℰ_x$.  In the second case, $\rank ℰ$ is even.
  \end{enumerate}
\end{prop}
\begin{proof}
  From the proof of \cite[Prop.~41]{BalajiKollar08}, it follows that $H_x(ℰ)$ is
  one of the following: $\SL_{ℂ}(ℰ_x)$, $\GL_{ℂ}(ℰ_x)$, $\Sp_{ℂ}(ℰ_x)$, or
  $\GSp_{ℂ}(ℰ_x)$.  The two groups not appearing in our list are excluded by the
  Bochner principle, Proposition~\ref{prop:algebraicholonomyprinciple}, and the
  assumption on the determinant of $ℰ$.
\end{proof}

\subsubsection{The basic dichotomy in terms of algebraic holonomy}
\approvals{Daniel & yes \\ Henri & yes \\ Stefan & yes}
  
Using the theory of algebraic holonomy groups summarised in the previous
section, we may now give another characterisation of the two cases in the basic
dichotomy.

\begin{thm}\label{thm:alg_holonomy_dichotomy}
  In the Standard setting~\ref{setting:holonomy}, additionally assume that
  $ω_X ≅ 𝒪_X$, that $𝒯_X$ is strongly stable, and that $G = G°$.  Then, the
  following holds.
  \begin{enumerate}
  \item The variety $X$ is CY if and only if the connected component of
    $H_x(𝒯_X)$ is equal to $\SL_{ℂ}(T_x X)$.
  \item The variety $X$ is IHS if and only if the connected component of
    $H_x(𝒯_X)$ is equal to $\Sp_{ℂ}(T_xX)$, where the linear complex-symplectic
    form on $T_x X$ is the evaluation of the holomorphic symplectic two-form.
  \end{enumerate}
\end{thm}

\begin{rem}
  We thank Stéphane Druel for explaining the following to us: in the setup of
  Theorem~\ref{thm:alg_holonomy_dichotomy}, it can be shown by a more detailed
  differential-geometric analysis that the structure group of $𝒯_{X_{\reg}}$ can
  be reduced to the complexification $G^{ℂ}$ of $G$, which is reductive as $G$
  is compact.  Therefore, by \cite[Thm.~1(4)]{BalajiKollar08} the subgroup
  $H_x(𝒯_X) ⊂ G^{ℂ}$ is in fact already connected.
\end{rem}

\begin{proof}[Proof of Theorem~\ref{thm:alg_holonomy_dichotomy}]
  From \cite[Lem.~6.20]{Dru16} and the proof of \cite[Lem.~40]{BalajiKollar08}
  it follows that there exists a quasi-étale cover $γ: Y → X$ and a point
  $y ∈ Y$ mapping to $x$ such that $H_{y}(γ^{[*]} 𝒯_X) = H_{y}(𝒯_Y)= H_x(𝒯_X)°$.
  Note that $X$ is CY (resp.\ IHS) if and only if $Y$ is.  Therefore one can
  assume without loss of generality that $H_x(𝒯_X)$ is connected.
  
  We claim that $\Sym^{[2]} 𝒯_X$ is indecomposable.  Indeed, by the Bochner
  principle for bundles, Theorem~\ref{generalized:holonomy}, a direct summand
  would give rise to a $G$-stable subspace of the $G$-representation
  $\Sym²(T_x X)$, which is irreducible by \cite[Sect.~24.1 and 24.2]{MR1153249},
  a contradiction.
 
  We may hence apply Proposition~\ref{prop:two_cases_for_algebraic_holonomy} to
  conclude that $H_x(𝒯_X)$ is either $\SL_{ℂ}(T_x X) $ or $ \Sp_{ℂ} (T_x X)$.
  We claim that the latter case occurs if and only if $X$ carries a holomorphic
  symplectic form.  Indeed, if $H_x(𝒯_X) = \Sp_{ℂ}(T_x X)$, then the Bochner
  principle for algebraic holonomy,
  Proposition~\ref{prop:algebraicholonomyprinciple}, implies that $Ω^{[2]}_X$
  has a one-dimensional trivial direct summand; the corresponding two-form is
  holomorphic symplectic owing to Lemma~\ref{lem:symplectic}.  Conversely, if
  there exists a non-vanishing two-form on $ X$, then the direct summand
  $Ω_X^{[2]} ⊂ \left(𝒯_X ⊗ 𝒯_X\right)^{**}$, which is polystable by
  Theorem~\ref{generalized:holonomy}, is decomposable.  As a consequence,
  $\bigwedge²(T_x^* X)$ is decomposable as $H_x(𝒯_X)$-representation, which in
  turn excludes the algebraic holonomy from being equal to $\SL_{ℂ}(T_x X)$ by
  \cite[Thm.~5.5.11]{MR2522486}.
\end{proof}

%
%
\svnid{$Id: 13-fundamentalGroups.tex 763 2018-10-27 13:39:14Z kebekus $}

\section{Fundamental groups}
\subversionInfo
\approvals{Daniel & yes \\ Henri & yes \\ Stefan & yes}

This section is devoted to studying the fundamental group of Calabi-Yau and
irreducible holomorphic symplectic varieties.  In the smooth case, these
varieties are by definition simply connected but in our singular setup, this
might not be the case anymore.  Actually, there are two relevant fundamental
groups to look at, $π_1(X)$ and $π_1(X_{\reg})$, and we will obtain finiteness
result concerning both of them.

\subsection{Fundamental groups of even-dimensional CY and IHS varieties}
\approvals{Daniel & yes \\ Henri & yes \\ Stefan & yes}

The following theorem summarises our results for varieties of even dimension.
Theorem~\ref{thm:finite_fundamental} is expected to hold also in odd dimensions,
and even more generally for varieties with vanishing augmented regularity, see
\cite[Conj.~4.16]{Kollar95s} and \cite[Qu.~5.12]{Ca95}.
Section~\vref{subsect:oddCY} contains partial results in this direction.

\begin{thm}[Fundamental groups of even-dim.~strongly stable varieties]\label{thm:finite_fundamental}
  In the standard Setting~\ref{setting:holonomy}, if $𝒯_X$ is strongly stable
  and if $\dim X$ is even, then $π_1(X)$ is finite.  In particular, the
  topological fundamental group of an IHS variety is finite.
\end{thm}
\begin{proof}
  Let $γ: Y → X$ be a global index-one cover, as given by
  Proposition~\ref{prop:global_index_one_cover}.  The image of $π_1(Y)$ in
  $π_1(X)$ has finite index by \cite[Prop.~1.3]{Campana91}.  To prove finiteness
  of $π_1(X)$, it therefore suffices to show that $π_1(Y)$ is finite.  As $Y$
  has canonical singularities, this will follow from \cite[Prop.~8.23]{GKP16}
  once we show that $χ(Y, 𝒪_{Y}) ≠ 0$.  First, we note that $\dim Y$ is even and
  that the tangent sheaf $𝒯_Y$ is strongly stable.  In addition,
  \cite[Prop.~6.9]{GKP16} shows that
  $$
  h^q \bigl(Y,\, 𝒪_Y \bigr) = h⁰ \bigl(Y,\, Ω^{[q]}_Y\bigr) \qquad \text{for
    all } 0 ≤ q ≤ n,
  $$
  and so Theorem~\ref{thm:dichotomy-B} yields that
  $$
  χ \bigl(Y,\, 𝒪_Y \bigr) = h⁰\bigl(Y,\, 𝒪_Y\bigr) + h⁰\bigl(Y,\,
  Ω^{[2]}_Y\bigr) + ⋯ + (-1)^n·h⁰\bigl(Y,\, Ω^{[n]}_Y \bigr) ≥ 2.
  $$
  This concludes the proof of Theorem~\ref{thm:finite_fundamental}.
\end{proof}

Combining Theorem~\ref{thm:finite_fundamental} and \cite[Prop.~7.3]{GKP16} we
obtain the following result.

\begin{cor}[Étale fundamental groups of smooth locus]\label{cor:finite_fundamental}
  In the setting of Theorem~\ref{thm:finite_fundamental}, the étale fundamental
  group $\what{π}_1 (X_{\reg})$ is finite.  In particular, the étale fundamental
  group of the smooth locus of any IHS variety is finite.  \qed
\end{cor}

\Publication{Once finiteness of the fundamental group is established, following
  the arguments of \cite[proof of Prop.~4(2)]{Bea83} one sees that the
  fundamental group is actually trivial; see the preprint version of this paper
  for a detailed proof.}

\begin{cor}\label{cor:IHS_simplyconnected}
  Let $X$ be a CY variety of even dimension or an IHS variety.  Then, $X$ is
  simply connected.  \Publication{\qed}
\end{cor}
\Preprint{%
  \begin{proof}
    The argument follows \cite[proof of Prop.~4(2)]{Bea83}.  Let $γ: Y → X$ be
    the universal cover of $X$ and let $d ∈ ℕ$ be its degree, which is finite by
    Theorem~\ref{thm:finite_fundamental}.  Since klt spaces have rational
    singularities, the holomorphic Euler characteristics of $𝒪_X$ and $𝒪_Y$ can
    be computed by passing to a resolution of singularities and differ by a
    factor of $d$, that is,
    $χ \bigl(Y,\, 𝒪_Y \bigr) = d· χ\bigl(X,\, 𝒪_X \bigr)$.  On the other hand,
    as $Y$ is projective and $γ$ is étale, $Y$ is likewise CY or IHS, depending
    on the type of $X$.  It follows from \cite[Prop.~6.9]{GKP16} that
    $$
    χ \bigl(Y,\, 𝒪_Y \bigr) = χ\bigl(X,\, 𝒪_X \bigr) =
    \begin{cases}
      2 & \text{if $X$ and $Y$ are CY} \\
      \frac{1}{2}·\dim X+1 & \text{if $X$ and $Y$ are IHS} \\
    \end{cases}
    $$
    Either way, it follows that $d=1$.  This concludes the proof.
  \end{proof}}

\begin{rem}[Corollary~\ref{cor:IHS_simplyconnected} is optimal]
  There are examples of smooth odd-dimensional CY-manifolds with non-trivial,
  finite fundamental group.  For instance, there exists a fixed-point free
  action of $ℤ_5$ on the diagonal quintic threefold in $ℙ⁴$ such that the
  quotient has trivial canonical divisor.
\end{rem}

\begin{rem}
  The arguments and techniques of \cite[Sect.~2]{OguisoSchroeer} can be easily
  adapted to study klt varieties with numerically trivial canonical class that
  admit a quasi-étale cover by an IHS variety, or, in other words, to varieties
  whose restricted holonomy equals $\Sp\!\left(\frac n2\right)$.  In analogy
  with the smooth case, one might call these \emph{Enriques varieties}.  An
  analogue of \cite[Lem.~8.14]{GKP16} holds in this singular setup.
\end{rem}

\subsection{Fundamental groups of odd-dimensional CY varieties}\label{subsect:oddCY}
\approvals{Daniel & yes \\ Henri & yes \\ Stefan & yes}

After the discussion in the previous subsection, it remains to consider the
fundamental group of odd-dimensional CY varieties.  In the smooth case, it
follows from the Theorem of Cheeger-Gromoll that such varieties have finite
fundamental group, see \cite[proof of Thm.~1]{Bea83}.  Here, we gather some
partial information for the singular case.  All of these go back to
Hodge-theoretical arguments and use the non-existence of reflexive symmetric
differentials.  In fact, they hold for arbitrary varieties with vanishing
augmented irregularity.

\begin{thm}[Fundamental group of varieties with $\wtilde{q} = 0$, I]\label{thm:oddCY}
  Let $X$ be a projective klt variety with numerically trivial canonical divisor
  and vanishing augmented irregularity, $\wtilde{q}(X) = 0$.  Then, the
  following hold.
  \begin{enumerate}
  \item\label{il:no_infinite_representations} The fundamental group $π_1(X)$
    does not have any finite-dimensional representation with infinite image
    (over any field).
  \item\label{il:JH_finite} For each $n ∈ ℕ$, the fundamental group $π_1(X)$ has
    only finitely many $n$-dimensional complex representations up to
    conjugation.
  \item\label{il:polygrowth} If infinite, the group $π_1(X)$ cannot have weakly
    polynomial growth in the sense of \cite[Def.~1.1]{MR2629989}.
 \end{enumerate}
\end{thm}

\begin{rem}
  There exist finitely generated, infinite groups that do not admit
  finite-dimensional representations.  One example is Higman's group,
  \cite[Ex.~1.1]{Berrick}.
\end{rem}

\begin{rem}
  Concerning Item~\ref{il:polygrowth}, cf.~also the discussion in
  \cite[p.~500]{Ca95}.
\end{rem}

\begin{proof}[Proof of Theorem~\ref{thm:oddCY}]
  Let $π: \wtilde{X} → X$ be a resolution of $X$ and recall from
  \cite[Thm.~1.1]{Takayama2003} that $π_1(\wtilde{X}) = π_1(X)$.

  \subsubsection*{Proof of \ref*{il:no_infinite_representations}}

  Argue by contradiction and suppose that there exists a representation
  $π_1(\wtilde{X}) → \GL_r(\bK)$ with infinite image, for some positive integer
  $r$ and some field $\bK$.  It then follows from a recent result of
  Brunebarbe-Klingler-Totaro, \cite[Thm.~0.1]{MR3127814}, that there exists a
  number $m > 0$ and a non-zero element
  $$
  0 ≠ \wtilde{σ} ∈ H⁰\bigl(\wtilde{X},\, \Sym^m Ω¹_{\wtilde{X}} \bigr).
  $$
  Restricting $\wtilde{σ}$ to the complement of the exceptional divisor of $π$
  we obtain a non-trivial element $σ ∈ H⁰\bigl(X,\, \Sym^{[m]}Ω¹_{X} \bigr)$,
  which by Theorem~\ref{thm:augm_irreg_rev} yields $\wtilde{q}(X) ≠ 0$,
  contradiction.

  \subsubsection*{Proof of \ref*{il:polygrowth}}

  Item~\ref{il:polygrowth} follows from \ref{il:no_infinite_representations} and
  from the (deep) fact that an infinite, finitely generated group with weakly
  polynomial growth admits a finite-dimensional (real) representation with
  infinite image, cf.\ \cite[Sect.~4]{MR2629989}.

  \subsubsection*{Proof of \ref*{il:JH_finite}}

  We saw in the proof of \ref{il:no_infinite_representations} that
  $H⁰\bigl(\wtilde{X},\, \Sym^m Ω_{\wtilde{X}}¹ \bigr) = \{0\}$ for all $m > 0$.
  It hence follows from \cite[Prop.~2.4]{MR1978714} or
  \cite[Thm.~1.6(i)]{MR3044124} that the variety
  $\Hom\bigl(π_1(X), \GL_r(ℂ)\bigr) {/\hspace{-0.12cm}/} \GL_r(ℂ)$, which
  parametrises representations up to Jordan-Hölder equivalence, consists of
  finitely many points.  By Item~\ref{il:no_infinite_representations} and
  Maschke's Theorem every complex representation of $π_1(X)$ is in fact
  semisimple, in which case Jordan-Hölder equivalence reduces to equivalence up
  to conjugation.  This establishes the claim made in item~\ref{il:JH_finite}.
\end{proof}

\begin{rem}
  The converse of Item~\ref{il:no_infinite_representations} is however false.
  The singular Kummer surface $X=A/\langle \pm 1\rangle$ that will be discussed
  in Example~\ref{ex:kummer} satisfies $\wtilde{q}(X)=2$ but $π_1(X)=\{1\}$.
  Indeed, the minimal resolution of $X$ is a $K3$ surface, hence it is simply
  connected and therefore so is $X$ by Takayama's result
  \cite[Thm.~1.1]{Takayama2003}.
\end{rem}

\begin{cor}[Fundamental group of varieties with $\wtilde{q} = 0$, II]\label{thm:oddCY2}
  Let $X$ be a projective klt variety with numerically trivial canonical divisor
  and vanishing augmented irregularity.  Then, the following hold.
  \begin{enumerate}
  \item\label{il:no_infinite_representations_smooth_part} The fundamental group
    $π_1(X_{\reg})$ does not have any finite-dimensional representation with
    infinite image (over any field).
  \item\label{il:JH_finite_smooth_part} For each $n ∈ ℕ$, the fundamental group
    $π_1(X_{\reg})$ has only finitely many $n$-dimensional complex
    representations up to conjugation.
  \item\label{il:polygrowth_smooth_part} If infinite, the group $π_1(X_{\reg})$
    cannot have weakly polynomial growth in the sense of
    \cite[Def.~1.1]{MR2629989}.
   \end{enumerate}
\end{cor}
\begin{proof}
  ---

  \subsubsection*{Proof of \ref*{il:no_infinite_representations_smooth_part}}
  
  Let $γ: Y → X$ be a Galois, maximally quasi-étale cover, as constructed in
  \cite[Thm.~1.5]{GKP13}.  As $γ$ is quasi-étale, the Galois group $G$ of $γ$
  fits into an exact sequence as follows
  \begin{equation}\label{eq:exact_group_sequence}
    1 \rightarrow π_1(Y_{\reg}) \overset{γ_*}{\longrightarrow} π_1(X_{\reg}) → G → 1.
  \end{equation}
  We let $ι: Y_{\reg} ↪ Y$ denote the inclusion map.  Let
  $ρ: π_1(X_{\reg}) → \GL_r(\bK)$ be any representation of $π_1(X_{\reg})$ over
  any field $\bK$.  As $Y$ is maximally quasi-étale, it follows from
  \cite[Sect.~8.1]{GKP13} or \cite[Thm.~1.2]{MR0262386} there exists a
  representation $ρ_Y: π_1(Y) → \GL_r(\bK)$ making the following diagram
  commutative
  $$
  \begin{xymatrix}
    { π_1(Y_{\reg}) \ar@{^(->}[d]_{γ_*}\ar@{->>}[r]^{ι_*} \ar^{ρ_Y°}[rd] & π_1(Y) \ar[d]^{ρ_Y} \\
      π_1(X_{\reg}) \ar[r]^{ρ}& \GL_r(\bK) .
    }
  \end{xymatrix}
  $$
  Let $Γ := \img(ρ)$.  Furthermore, as $\wtilde{q}(Y) = 0$, we may apply
  part~\ref{il:no_infinite_representations} of Theorem~\ref{thm:oddCY} to
  conclude that $Γ_Y := \img(ρ_Y°) = \img(ρ_Y)$ is a finite, normal subgroup of
  $Γ$.  Using the exact sequence \eqref{eq:exact_group_sequence}, from $ρ$ we
  obtain a surjective group homomorphism
  $\overline{ρ}: G \twoheadrightarrow Γ/Γ_Y$.  As $G$ is finite by definition,
  $Γ/Γ_Y$ is hence finite.  Together with the finiteness of $Γ_Y$ observed
  above, Item~\ref{il:no_infinite_representations_smooth_part} follows.

  \subsubsection*{Proof of \ref*{il:polygrowth_smooth_part}}
  
  The claim follows as in the proof of Item~\ref{il:polygrowth} above.

  \subsubsection*{Proof of \ref*{il:JH_finite_smooth_part}}

  Fix $n ∈ ℕ$ and set $Σ := π_1(X_{\reg})$ and $Σ_Y := π_1(Y_{\reg})$.  As every
  finite-dimensional complex representation of $Σ$ is semisimple by
  \ref{il:no_infinite_representations_smooth_part} and Maschke's Theorem, it
  suffices to show the claim for simple $Σ$-representations.  Because each
  $Γ_Y$-representation factors over $π_1(Y)$, Theorem~\ref{thm:oddCY} implies
  that the group $Σ_Y$ has only finitely many $n$-dimensional complex
  representations up to conjugation.  Therefore, if $ρ: Σ → \GL_{ℂ}(V)$ is a
  simple $n$-dimensional representation, there are only finitely many
  possibilities for $ρ|_{Σ_Y}$ up to conjugation.  Consequently, again up to
  conjugation, there are only finitely many possibilities for
  $\operatorname{ind}_{Σ_Y}^{Σ} (ρ|_{Σ_Y})$, the $Σ$-module obtained via
  induction\footnote{See for example \cite[Sect.~3.3]{MR1153249}} from $Σ_Y$ to
  $Σ$.  There exists a natural $Σ$-module isomorphism
  $$
  η: ℂ[Σ / Σ_Y] ⊗_{ℂ} V \overset{≅}{\longrightarrow}
  \operatorname{ind}_{Σ_Y}^{Σ} (ρ|_{Σ_Y}),
  $$
  where $ℂ[Σ / Σ_Y]$ is the $Σ$-module of $ℂ$-valued functions on the
  homogeneous $Σ$-space $Σ/Σ_Y$, see \cite[Ex.~3.6]{MR1153249}.  As the
  $Σ$-equivariant map $v ↦ η(\b1 ⊗ v)$ realises the simple representation $V$ as
  a direct summand of the semisimple representation
  $\operatorname{ind}_{Σ_Y}^{Σ} (ρ|_{Σ_Y})$, there are only finitely many
  possibilities for $ρ$ up to isomorphism, as claimed in
  \ref{il:JH_finite_smooth_part}.
\end{proof}

%
%
\svnid{$Id: 14-example.tex 764 2018-10-27 19:17:17Z greb $}

\section{Examples}\label{sec:examples}
\subversionInfo
\approvals{Daniel & yes \\ Henri & yes \\ Stefan & yes}

The present section gathers examples that illustrate the main results of this
paper.  As announced in the introduction, we begin in Section~\ref{ssec:loupbir}
with two examples that show how the holonomy changes under birational
transformation.  Perhaps more importantly, Section~\ref{ssec:finquot}
illustrates the classification scheme established in the previous sections,
underlines the necessity of using quasi-étale covers, and points out the
differences to other suggestions for a definition of ``irreducible holomorphic
symplectic variety'' that are found in the literature.  Finally,
Section~\ref{ssec:msosoK3} discusses moduli of sheaves on K3 surfaces.

\subsection{Change of holonomy under crepant resolutions}\label{ssec:loupbir}
\approvals{Daniel & yes \\ Henri & yes \\ Stefan & yes}

The singular Kähler-Einstein metric $ω_H$ discussed in the standard
Setting~\ref{setting:holonomy} does depend on the choice of the ample divisor
$H$.  However, we have seen in Proposition~\ref{prop:independence} that the
isomorphism class of the restricted holonomy group $G°$ is in fact independent
of $H$.  We can therefore speak of \emph{the} restricted holonomy, and ask how
it changes under birational transformation.  The following two examples show
that holonomy does in fact change, even for crepant resolutions of
singularities.

\begin{example}[Singular Kummer surface]\label{ex:kummer}
  Let $X := A/\langle \pm 1\rangle$ where $A$ is an Abelian surface, and let
  $π:\wtilde X → X$ be the (crepant) minimal resolution of $X$, which is a $K3$
  surface.  We analyse the relevant (singular) Kähler-Einstein metrics.
  \begin{description}
  \item[On the crepant resolution] If $ω_{\wtilde X}$ is any Ricci-flat Kähler
    metric on $\wtilde X$ with associated Riemannian metric $g_{\wtilde X}$,
    then the associated holonomy group
    $\Hol \bigl( \wtilde{X},\, g_{\wtilde X} \bigr)$ is isomorphic to $\SU(2)$.

  \item[On the singular Kummer surface] If $ω_A$ is any flat metric on $A$
    induced by a constant metric from $ℂ^n$, then $ω_A$ is invariant under the
    action of $\pm 1$ on $A$, and hence descends to a singular Ricci-flat metric
    $ω_X$ on $X$, in the sense of Theorem~\ref{thm:EGZ}.  The metric $ω_X$ is
    flat on $X_{\reg}$.
  \end{description}
\end{example}
  
\begin{example}[Symmetric square of a K3]\label{ex:symk3}
  Let $S$ be a $K3$ surface, and let $X := S⨯S /\langle i \rangle$ where
  $i : (s_1,s_2) ↦ (s_2,s_1)$.  The quotient map $γ: S⨯S → X$ is quasi-étale and
  Galois with group $ℤ_2$.  Recall from \cite[Sect.~6]{Bea83} that the Hilbert
  scheme $\wtilde{X}$ parametrising zero-dimensional subschemes of length two is
  an irreducible holomorphic symplectic manifold and admits a birational,
  crepant map $π: \wtilde{X} → X$.  Once again, we analyse the relevant
  Kähler-Einstein metrics.
  \begin{description}
  \item[On the resolution] Any smooth Ricci-flat Kähler metric $ω$ on
    $\wtilde{X}$ with associated Riemannian metric $g_{\wtilde X}$, satisfies
    $\Hol \bigl( \wtilde{X},\, g_{\wtilde X} \bigr) ≅ \Sp(2)$.
  \item[On the singular symmetric square] On the other hand, if $ω_S$ is a
    Ricci-flat Kähler metric on $S$, then $\pr_1^*ω_S + \pr_2^*ω_S$ defines a
    Kähler Ricci-flat metric on $S⨯S$ that descends to a singular Kähler
    Ricci-flat metric $ω_X$ on $X$, with associated Riemannian metric
    $g_{X_{\reg}}$.  One computes\footnote{See Remark~\ref{rem:bhuc} and
      Proposition~\ref{prop:dtkd} below, as well as the first few lines of
      Sect.~\ref{ssec:pothf}} that $\Hol \bigl( X_{\reg},\, g_{X_{\reg}} \bigr)$
    is an extension of $\SU(2)⨯\SU(2)$ by $ℤ_2$, hence the restricted holonomy
    is reducible.
  \end{description}
\end{example}

\subsection{Finite quotients}\label{ssec:finquot}
\approvals{Daniel & yes \\ Henri & yes \\ Stefan & yes}

Singular varieties with trivial canonical class can be easily be constructed by
taking quotients.  For an example, consider quotient of an even-dimensional
Abelian variety by the involution $x ↦ -x$ or a symmetric product of an
irreducible holomorphic symplectic manifold (resp.\ an even dimensional
Calabi-Yau manifold) as in Examples~\ref{ex:kummer} and \ref{ex:symk3}.
However, exhibiting strongly stable singular varieties with trivial canonical
class seems to require more work.

\subsubsection{Quotient of Abelian varieties}\label{ex:ab}
\approvals{Daniel & yes \\ Henri & yes \\ Stefan & yes}

We construct quotients of Abelian varieties that have the algebra of reflexive
holomorphic forms of a Calabi-Yau variety or of a irreducible holomorphic
symplectic variety.

\begin{example}[A fake IHS variety with canonical singularities]\label{ex:kummer2}
  This is a higher dimensional generalisation of the singular Kummer surfaces
  discussed in Example~\ref{ex:kummer}.  Take $A$ an Abelian surface, and
  consider $Y:=A/\langle \pm 1\rangle$.  A non-zero holomorphic two-form $ω_A$
  on $A$ descends to a symplectic form $ω_Y$ of $Y_{\reg}$.  Now, let us
  consider $X:=Y^{(n)}$ the $n$-th symmetric product of $Y$ for some $n ≥ 2$.
  The variety $X$ is realised as the quotient $Y^n/\mathfrak S_n$.  The two-form
  $\sum \pr_i^* ω_Y$ is a $\mathfrak S_n$-invariant symplectic form on
  $(Y_{\reg})^n$, where $\pr_i:Y^n→ Y$ is the projection to the $i$-th factor.
  Hence it descends to a symplectic form $ω$ on $X_{\reg}$, which we can
  interpret as a reflexive two-form on $X$.  It is not hard to see that
  $\bigoplus_{p=0}^{2n} H⁰(X,Ω_X^{[p]}) = ℂ[ω]$.  In particular, $X$ has the
  same algebra of reflexive forms as a smooth irreducible holomorphic symplectic
  manifold.  However, its augmented irregularity is maximal, that is, equal to
  the dimension of $X$.  The tangent sheaf $𝒯_{X_{\reg}}$ is flat and $X$ admits
  a quasi-étale, Galois cover that is an Abelian variety.
\end{example}

In \cite{MR3319924}, Matsushita studies what he calls \emph{cohomologically
  irreducible symplectic (CIHS) varieties}.  By definition, these are
projective, holomorphic symplectic varieties $X$ satisfying the following two
conditions.
\begin{itemize}
\item The variety $X$ has $ℚ$-factorial, terminal singularities.
\item We have an isomorphism of algebras,
  $\bigoplus_{p=0}^{2n} H⁰ \bigl(X,\, Ω_X^{[p]} \bigr) = ℂ[ω]$, where $ω$ is a
  holomorphic symplectic form.
\end{itemize}
While these varieties share many properties with smooth irreducible holomorphic
symplectic manifolds, they are not necessarily IHS, see the next example.  This
should be compared with the smooth case, where requiring the second condition to
hold already forces the manifold to be simply connected, and hence IHS,
see~\cite[Rem.~8.19]{GKP16}.

\begin{example}[A CIHS variety that has maximal augmented irregularity]\label{ex:cihs1}
  Let $A$ be an Abelian surface and $t_0 ∈ A$ be a two-torsion point.  Let
  $\tr_{t_0}$ be the translation by $(0, t_0)$, consider the morphism
  $τ : (s,t) ↦ (t,s)$ as well as $\varphi := \tr_{t_0}◦τ$.  Observe that
  $\varphi$ induces a free action of $ℤ_4$ on $A⨯A$.  We notice that $\varphi$
  commutes with $(-1): (s,t) ↦ (-s,-t)$, so that $\varphi$ and $(-1)$ generate
  an action of $G:= ℤ_4 ⨯ ℤ_2$ on $A⨯A$.  We set $X := (A⨯A)/G$ and let
  $π: A⨯A → X$ be the quotient map.  As the $ℤ_4$-action is free, the map $π$ is
  quasi-étale and the singularities of $X$ are exactly the images of the 256
  two-torsion points under $π$, which are therefore isolated.  Either by direct
  computation or by a result of
  Namikawa~\cite[Cor.~1]{NamikawaTerminalCodimensionSymplectic}, we see that $X$
  has terminal singularities, which are in addition obviously $ℚ$-factorial.  On
  the other hand, by construction
  $$
  H⁰ \bigl( A⨯A,\, Ω²_{A⨯ A}\bigr)^G = ℂ·ω \quad\text{where}\quad ω =
  \pr_1^*(dz_1Λdz_2) + \pr_2^*(dz_1Λdz_2).
  $$
  Observe that $ω$ is symplectic, and that
  $H⁰ \bigl( A⨯A,\, Ω^p_{A⨯ A}\bigr)^G = \{0\}$ for $p = 1,3$.
\end{example}

\begin{example}[A fake CY threefold]\label{ex:igusa}
  The Calabi-Yau case is a bit more involved, but still well-known,
  cf.~\cite[Ex.~2.17]{MR1868164}.  That example yields a free action of
  $G:=ℤ_2⨯ ℤ_2$ on a product $A = E_1⨯E_2⨯E_3$ of three elliptic curves such
  that
  $$
  H⁰\bigl(A,\,Ω_A^p\bigr)^{G} =
  \begin{cases}
    1 &\mbox{if } p=0,3\\
    0 &\mbox{if } p=1,2
  \end{cases}
  $$
  In particular, $A/G$ is a \emph{smooth} manifold with trivial canonical bundle
  and the algebra of holomorphic form of a Calabi-Yau threefold.
\end{example}

\subsubsection{Quotients of CY manifolds}
\approvals{Daniel & yes \\ Henri & yes \\ Stefan & yes}

A classical way to produce singular CY varieties with quotient singularities is
to start with a Fano manifold $X$ of dimension at least $3$ and a finite group
$G$ acting on $X$.  Then, one considers a general element $Y∈ |-K_X|^G$.  If $Y$
is smooth, then it is an irreducible Calabi-Yau manifold endowed with an action
of $G$.  Indeed, it has trivial canonical bundle by adjunction, it is simply
connected by the Lefschetz hyperplane theorem and as $X$ is Fano, one has
$h⁰\bigl(X,\,Ω^p_X\bigr) = 0$ for $p>0$ by Kodaira-Nakano vanishing, which in
turn implies that $h⁰ \bigl(Y,\, Ω_Y^p \bigr) = 0$ for $0 < p < \dim Y$ by the
Lefschetz theorem for Hodge groups, see~\cite[Lem.~4.2.2]{Laz04-I}.  The variety
$Y/G$ has klt singularities.  Moreover, if $G$ preserves the holomorphic volume
form on $Y$, then $Y/G$ has trivial canonical bundle; in particular it is
Gorenstein with canonical singularities.  In that case, $Y/G$ is automatically a
Calabi-Yau variety in the sense of Definition~\ref{def:CY_and_IHS}.

\begin{example}[A terminal quotient with non-Gorenstein isolated singularities]\label{ex:fav1}
  In \cite[Ex.~1]{Favale16}, Favale shows that $X = ℙ²⨯ℙ²$ admits an action of
  $G=ℤ_3$ such that a general element $Y ∈ |-K_X|^G$ is smooth and such that $G$
  \emph{does not} preserve the holomorphic volume form on $Y$.  The singular
  locus of the variety $Y/G$ consists of $9$ points, each of them being
  terminal.  The restricted holonomy is equal to $\SU(3)$, yet $Y/G$ is not CY.
\end{example}

The following example is constructed in a similar fashion, although the details
are more technical to work out.  It is Gorenstein, has trivial canonical bundle,
and possesses non-isolated singularities.

\begin{example}[A CY threefold with a one-dimensional singular locus]\label{ex:fav2}
  Now, following \cite[Ex.~4]{Favale16}, take $X=ℙ¹⨯ℙ¹⨯ℙ¹⨯ℙ¹$.  It admits an
  action of $G = D_{16}⨯ℤ_2$ where $D_{16}$ is the dihedral group.  As explained
  in \emph{loc.\ cit.}, a general element $Y ∈ |-K_X|^G$ is smooth and $G$
  preserves the holomorphic volume form on $Y$.  Moreover, the singular locus of
  the variety $Y/G$ has dimension one ---it is not irreducible and some of its
  components may be zero-dimensional though.
\end{example}

\subsubsection{Quotients of $K3 ⨯ K3$ and of IHS manifolds}\label{ex:ihs}
\approvals{Daniel & yes \\ Henri & yes \\ Stefan & yes}

The following example is an example of cohomologically irreducible holomorphic
symplectic variety in the sense of Matsushita~\cite{MR3319924}, cf.\ also
Section~\ref{ex:ab} and Example~\ref{ex:cihs1} above, yet it is covered by a
product of two $K3$-surfaces.

\begin{example}[A CIHS variety with restricted holonomy $\SU(2) ⨯ \SU(2)$]\label{ex:cihs2}
  Let $S$ be a $K3$-surface with a symplectic involution $τ$.  The fixed point
  locus of $τ$ consists of isolated points.  Consider the action of $ℤ_4$ on
  $S⨯S$ generated by the automorphism $σ$ defined by $(x,y) ↦ (τ(y), x)$.  The
  fixed points of $σ$ are of the form $(x,x)$ where $x ∈ \Fix(τ)$, hence
  $X:=(S⨯S)/\langle σ \rangle$ has isolated, $ℚ$-factorial singularities.
  Moreover, we have by construction
  $$
  \bigoplus\nolimits_{p} H⁰ \bigl( S⨯S,\, Ω^p_{S⨯ S}\bigr)^{\langle σ \rangle} =
  ℂ[ω] \quad\text{where}\quad ω = \pr_1^*ω_S+\pr_2^*ω_S
  $$
  for some symplectic form $ω_S$ on $S$.  Using Namikawa's
  result~\cite[Cor.1]{NamikawaTerminalCodimensionSymplectic} as before, we see
  that $X$ has terminal singularities, and therefore it is a CHIS variety.
  However, it is covered by the product of two $K3$-surfaces, so it is not a
  quotient of an IHS variety.
\end{example}

We conclude with an example that shows that given a variety with trivial
canonical bundle and restricted holonomy $\Sp\!\left(\frac n2\right)$, taking a
finite, quasi-étale cover is indeed necessary before a symplectic form will
necessarily exist, cf.\ Section~\ref{subsect:CY_and_IHS_cases} above.

\begin{example}[A quotient of an IHS manifold with $K_X$ trivial but no two-form]\label{ex:no2form}
  Let $S$ be a $K3$ surface endowed with an anti-symplectic involution $τ$.  For
  instance, take $S$ to be the minimal resolution of the the quotient
  $(E⨯E)/\langle σ\rangle$ where $E$ is the elliptic curve $ℂ/(ℤ⊕iℤ)$ and $σ$
  acts on $E⨯E$ by $\diag(i,-i)$.  Define then $τ$ to be the lift of
  $\diag(-1,1)$ to $S$, cf.~\cite[Ex.~2]{Oguiso-Zhang}.  Now, let us consider
  $S^{[2]}$ the Hilbert scheme parametrising length $2$ zero-dimensional
  subschemes of $S$.  The variety $S^{[2]}$ is an irreducible holomorphic
  symplectic manifold endowed with an anti-symplectic involution that we will
  still call $τ$.  The fixed locus of $τ$ is a smooth Lagrangian submanifold of
  $S^{[2]}$.  If $ω$ is a symplectic form on $S^{[2]}$, then $τ^*ω = -ω$ hence
  $τ^*ω² = ω²$.  In particular $X := S^{[2]}/\langle τ\rangle$ has canonical
  singularities (concentrated along a surface), trivial canonical bundle, but no
  non-zero two-form.  However, it has a quasi-étale cover that is an IHS
  manifold.
\end{example}

\subsection{Moduli spaces of sheaves on $K3$-surfaces}\label{ssec:msosoK3}
\approvals{Daniel & yes \\ Henri & yes \\ Stefan & yes}

Let $S$ be projective K3-surface.  As usual, we equip the even integral
cohomology of $S$ with the pairing
$$
\langle v,w \rangle:= -\int_X vw^*,
$$
where $w^* = (-1)ⁱ w$ for $w ∈ H^{2i}\bigl(S,\, ℤ \bigr)$.  To each coherent
sheaf $E$ we associate its Mukai vector
$v(E) = \ch(E) \sqrt{\td(S)} ∈ H^{even}\bigl(S,\, ℤ \bigr)$.  For a given $v$
and an ample Cartier divisor $H$ on $S$, we denote by $M_v(H)$ the
Gieseker-Maruyama moduli space of $H$-semistable sheaves with Mukai vector $v$
on $S$.  Any given $v$ can be decomposed as $v = mv_0$, where
$v_0 = (r_0, c_0, a_0)$ is primitive and $m ∈ ℕ^{+}$.  For simplicity, we assume
that $r_0 > 0$ in the following.  Under the additional assumption that $H$ is
``$v$-general'', every $H$-semistable sheaf is $H$-stable, and the corresponding
moduli space is non-empty if and only if $c_0 ∈ \NS(S)$ and
$\langle v_0, v_0\rangle ≥ -2$.  In \cite{MR2221132}, Kaledin, Lehn, and Sorger
show that if either $m ≥ 2$ and $\langle v_0, v_0 \rangle > 2$, or $m > 2$ and
$\langle v_0, v_0 \rangle ≥ 2$, then $M_{mv_0}(H)$ is a projective variety with
locally factorial, symplectic (and hence canonical) singularities that admits a
holomorphic symplectic two-form, but no symplectic resolution.  Assuming that
$\langle v_0, v_0\rangle ≥2$ and $m≥ 1$, Perego and Rapagnetta recently proved
in \cite[Thm.~1.19]{PeregoRapagnetta} that $M_{mv_0}(H)$ is an IHS-variety.
They also prove a similar statement for moduli of sheaves on abelian surfaces.


\end{document}